\documentclass[reqno,12pt,letterpaper]{amsart}

\usepackage{amsmath,amssymb,amsthm,graphicx,mathrsfs,url,bm,subfigure,accents,paralist,xargs,slashed,enumitem,mathtools,array}
\usepackage[usenames,dvipsnames]{xcolor}
\usepackage[colorlinks=true,linkcolor=Red,citecolor=Green]{hyperref}
\usepackage[colorinlistoftodos,prependcaption,textsize=tiny]{todonotes}
\presetkeys{todonotes}{inline}{}

\numberwithin{equation}{section}

\newtheorem{theo}{Theorem}
\newtheorem{prop}{Proposition}[section]
\newtheorem{lemm}[prop]{Lemma}
\newtheorem{corr}[prop]{Corollary}

\theoremstyle{definition}
\newtheorem{defi}[prop]{Definition}
\newtheorem{rema}[prop]{Remark}

\newcommand{\be}{\mathrm{b}}
\newcommand{\bsymbol}[1]{\sigma_{\be,{#1}}}

\newcommand{\Diffb}{\mathrm{Diff}_{\be}}

\newcommand{\ellb}{\mathrm{ell}_\be}
\newcommand{\fourier}{\mathcal{F}}

\newcommand{\Ltwo}{\mathcal{L}^2}
\newcommand{\mellin}{\mathcal{M}}
\newcommand{\Opb}{\mathrm{Op}_{\mathrm b}}
\newcommand{\opWFb}{\WFb'}
\newcommand{\Psib}{\Psi_{\mathrm{b}}}
\newcommand{\Psibc}{\Psi_{\mathrm{b},\comp}}
\newcommand{\Psibeven}[1][\mathrm{even}]{\Psi_{\mathrm{b},#1}}
\newcommand{\sgn}{\mathrm{sgn}}
\newcommand{\tdotSob}{\dot{\mathcal H}}

\newcommand{\tDiff}{\mathrm{Diff}_\nu}

\newcommand{\tSob}{\mathcal H}
\newcommand{\WFb}{\mathrm{WF}_{\mathrm{b}}}

\DeclareMathOperator{\Diff}{Diff}
\DeclareMathOperator{\supp}{supp}

\newcommand{\divop}{\mathrm{div}}

\newcommand{\RR}{\mathbb{R}}
\newcommand{\CC}{\mathbb{C}}
\newcommand{\NN}{\mathbb{N}}

\newcommand{\pa}{{\partial}}

\newcommand{\hyp}{\mathcal{H}}
\newcommand{\gl}{\mathcal{G}}
\newcommand{\ellip}{\mathcal{E}}

\newcommand{\loc}{\mathrm{loc}}
\newcommand{\comp}{\mathrm{c}}

\newcommand{\chare}{\mathcal{N}}
\newcommand{\cchare}{\dot{\mathcal{N}}}
\newcommand{\bT}{{}^\mathrm{b} T}
\newcommand{\bS}{{}^\mathrm{b} S}
\newcommand{\bdotT}{{}^{\mathrm{b}} \dot T}

\newcommand{\CI}{\mathcal{C}^\infty}
\newcommand{\CcI}{\mathcal{C}_{\comp}^\infty}
\newcommand{\CmI}{\mathcal{C}^{-\infty}}
\newcommand{\CdI}{\dot{\mathcal{C}}^{\infty}}

\newcommand{\CdmI}{\dot{\mathcal{C}}^{-\infty}}

\newcommand{\CIeven}{\mathcal{C}^\infty_{\mathrm{even}}}
\newcommand{\CIodd}{{\mathcal{C}^\infty_{\mathrm{odd}}}}

\newcommand{\maxdom}{\mathcal{X}}

\newcommand{\GBB}{\mathrm{GBB}}

\newcommand{\diag}{\mathrm{diag}}

\renewcommand{\Im}{\operatorname{Im}}
\renewcommand{\Re}{\operatorname{Re}}

\newcommand{\WFbop}{\WFb^\mathrm{Op}}

\newcommand{\beq}{\begin{equation}}
\newcommand{\eeq}{\end{equation}}
\newcommand{\bea}{\begin{aligned}}
\newcommand{\eea}{\end{aligned}}

\def\rr{{\mathbb R}}

\newcommand{\pt}{({\rm PT})}

\def\wf{{\rm WF}}

\def\cWb{\mathcal{W}_{\rm b}^{-\infty}(X)}
\def\bee{{}^{\rm b}}
\def\b{{\rm b}}
\def\cN{\mathcal{N}}

\def\cf{\mathcal{C}^\infty}
\def\pX{\partial X}

\newtheorem{hypo}[prop]{Hypothesis}

\def\Robin{R}

\title[Propagation of singularities on AdS spacetimes]
{Propagation of singularities on {AdS} spacetimes  for general boundary conditions and the holographic {H}adamard condition}

\author{Oran Gannot}
\email{gannot@northwestern.edu}
\address{Department of Mathematics, Lunt Hall, Northwestern University,
	Evanston, IL 60208, USA}

\author{Micha\l{} Wrochna}
\email{michal.wrochna@univ-grenoble-alpes.fr}
\address{Universit\'e Grenoble Alpes, CNRS, Institut Fourier, F-38000 Grenoble,
	France}

\thanks{\emph{Acknowledgments.} OG was partially supported by NSF grant DMS–1502632. MW gratefully acknowledges support from the grant ANR-16-CE40-0012-01. The authors are also grateful to Andr\'as Vasy and Claude Warnick for helpful conversations}

\begin{document}

\begin{abstract} We consider the Klein-Gordon equation on asymptotically anti-de Sitter spacetimes subject to Neumann or Robin (or Dirichlet) boundary conditions, and prove propagation of singularities along generalized broken bicharacteristics. The result is formulated in terms of conormal regularity relative to a twisted Sobolev space. We use this to show the uniqueness, modulo regularising terms, of parametrices with prescribed $\b$-wavefront set. Furthermore, in the context of quantum fields, we show a similar result for two-point functions satisfying a holographic Hadamard condition on the $\b$-wavefront set. 
	\end{abstract}
	
	\maketitle
	
\section{Introduction \& main results} 	The Klein-Gordon equation on asymptotically anti-de Sitter (aAdS) spacetimes was studied in a number of works in the last several years. We refer the reader to \cite{yagdjian2009klein,galstian2010lp,bachelot2011klein,holzegel2012well,warnick2013massive,holzegel2013decay,enciso2015singular,holzegel2014boundedness,holzegel2016unique,holzegel2017unique}  to mention only a few. Notably, the results include well-posedness for the Klein--Gordon equation by Vasy \cite{vasy2012wave} and Holzegel \cite{holzegel2012well} in the case of Dirichlet boundary conditions, as well as well-posedness for Neumann and Robin boundary conditions by Warnick \cite{warnick2013massive}. 

In applications to Quantum Field Theory the main objects of interest are  \emph{propagators}, which are singular distributions in the two spacetime variables. The key additional ingredient that is needed is a microlocal propagation of singularities theorem. In the case of Dirichlet boundary conditions, this was established by Vasy \cite{vasy2012wave}, and applied in \cite{wrochna2017holographic} to yield a result on distinguished parametrices largely analogous to that of Duistermaat and H\"ormander \cite{duistermaat1972fourier}.    

	The goal of the present paper is to provide these type of theorems in the case of Neumann and Robin boundary conditions on the boundary $\partial X$ of an aAdS spacetime $(X,g)$. These boundary conditions appear frequently in the physics literature; see e.g.~\cite{dappiaggi2018algebraic,dappiaggi2018ground}. Further motivation comes from the study of the Dirichlet-to-Neumann operator, which was recently shown to coincide with a power of the wave operator in the case of a static metric \cite{enciso2017fractional}.  
	
	The main difficulties are two-fold. First of all, the positive commutator estimates in \cite{vasy2012wave} have no direct generalization outside of the Dirichlet case, for the same reason that the associated energy is ill-defined for Neumann and Robin conditions. Secondly, the boundary conditions must be understood in terms of weighted traces $\gamma_\pm$ related to the polyhomogeneous expansions of solutions to the Klein--Gordon equation near the boundary (the corresponding elliptic setting is well-understood thanks to \cite{mazzeo1991elliptic,mazzeo2014elliptic} and \cite{gannot2018elliptic}). 
	
%	We also remark that the well-posedness theory of Kreiss--Sakomoto for \emph{inhomogeneous} boundary value problems in the smooth case has not been extended to the aAdS setting. Even for the inhomogeneous Dirichlet problem, sharp Sobolev estimates (in the sense of estimating weighted traces of solutions by the source terms) are not known. \todo[inline]{Add reference to the appropriate section containing some heuristics}

	In \cite{vasy2012wave}, Vasy proves propagation of conormal regularity relative to a scale of $0$-Sobolev spaces whose weights (or lack thereof) correspond directly to the form domain of the wave operator. Unless one introduces additional weights, it is not possible to pose the Neumann or Robin problems in these spaces. The use of weighted 0-Sobolev spaces in turn makes it difficult to apply the quadratic form techniques first developed in \cite{vasy2008propagation} (and subsequently applied applied in \cite{melrose2008propagation,vasy2010diffraction,vasy2012wave}). To circumvent these problems we adopt an approach based on microlocalizing a certain \emph{twisted Sobolev space} $\tSob^1_\loc(X) \subset x^{-1}L^2_\loc(X,dg)$, introduced in the present context by Warnick \cite{warnick2013massive}. 
	
	We show that $\b$-pseudo\-differential operators have good mapping properties on these Sobolev spaces.  This allows us to consider a \emph{$\b$-wavefront set} $\wf^{1,s}_{\b}(u)$ which microlocalizes the space $\tSob^{1,s}_\loc(X)$ of conormal distributions of order $s$ with respect to $\tSob^{1}_\loc(X)$ (i.e., the subspace of $\tSob^1_\loc(X)$ stable under applications of at most $s$ vector fields tangent to the boundary). 
	
	On an aAdS spacetime $(X,g)$ of dimension $n\geq 2$ (see Section \ref{sec:aAdS} for the precise definition), we consider the Klein-Gordon operator	
	\begin{equation} \label{eq:kg}
	P = \Box_g - \tfrac{(n-1)^2}{4}+\nu^2, \quad  \nu > 0.
	\end{equation}
The condition $\nu >0$ corresponds to the well-known Breitenlohner--Freedman mass bound. 
Let $\nu_\pm = \tfrac{n-1}{2} \pm \nu$ denote the indicial roots of $P$. The definition of $\tSob^1_\loc(X)$ (which depends on the Klein--Gordon parameter $\nu$) is based on stability under first order differential operators $Q$ on $X^\circ$ which are twisted in the sense that $x^{-\nu_-}Qx^{\nu_-}$ is \emph{smooth} up to the boundary for any boundary defining function (bdf) $x$. In other words, we work with the largest space of first-order differential operators preserving $x^{\nu_-}\CI(X)$. This is motivated by the following observation: if $F \in x^{\nu_-}\CI(X)$ is any function satisfying 
\begin{equation} \label{eq:twistingfunction}
F^{-1}P(F) \in x^2 \CI(X),
\end{equation}
then one has
\[
P = -(F^{-1}D_{z^i}F)^\dagger g^{ij} (F D_{z^j} F^{-1}) + F^{-1}P(F),
\]
where $\dagger$ refers to the $dg$-adjoint.  Following the terminology in \cite{holzegel2014boundedness}, we call $F$ an admissible twisting function (the existence of admissible twisting functions is discussed in Section \ref{subsect:kg}). Thus, modulo zeroth order terms, $P$ is the sum of twisted derivatives composed with their adjoints.

Compared with \eqref{eq:kg}, one gains two powers of $x$ in the zeroth order terms, which turns out to be crucial. This observation was first employed systematically in the study of AdS spacetimes by Warnick \cite{warnick2013massive}, and also appeared earlier in the closely related asymptotically hyperbolic setting \cite{gonzalez2013fractional,chang2011fractional}. We then associate to $P$ a Dirichlet form on $\tSob^1_\loc(X)$, and following the philosophy of \cite{vasy2008propagation} carry out a positive commutator argument at the level of quadratic forms.

For $x$ a bdf, the Dirichlet data $\gamma_- u$ of $u$ is simply $x^{-\nu_-}u$ restricted to the boundary; this restriction exists as a distribution (and transforms simply under changes of bdf). The Neumann data $\gamma_+ u$ is slightly more difficult to define; this is achieved in Sections \ref{subsect:asymptoticexpansion} and \ref{subsect:greens} for $\nu \in (0,1)$. When $\nu \in (0,1)$, we consider the Robin or Neumann realization $P_R$, corresponding to 
\[
\gamma_+ u - \beta\gamma_- u = 0
\]
for $\beta \in \CI(\pa X)$ real-valued. When $\nu > 0$, we can also consider the Dirichlet realization $P_D$ of $P$, corresponding to imposing $\gamma_- u =0$.  We prove the following propagation of singularities theorem.

% In the Neumann and Robin case we assume $\nu\in (0,1)$, which is the well-known Breitenlohner--Freeman bound. In that case we also assume that there exists $F\in x^{\nu_-}\CI(X;\RR )$ (what we call a \emph{twisting function}) such that:
%\[
%x^{-\nu_-}F > 0 \text{ on } X, \ \ F^{-1}P(F) \in x^2\CI(X). 
%\]
%This allows us to effectively replace $x^{\nu_-}$ by $F$ and  obtain this way better control of multiplication operators arising when writing $P_R$ in terms of \emph{twisted differentials}.
%\todo{this needs to be rewritten to actually explain what a twisting function does and when they exist. Mention something about the $\nu=1/2$ case?}

\begin{theo}[{Propagation of singularities}] \label{theo:GBBpropagation} 
Let $\nu \in (0,1)$. If $u \in \tSob^{1,m}_{\loc}(X)$ for some $m \leq 0$ and $s \in \RR \cup \{+\infty\}$, then $\WFb^{1,s}(u) \setminus \WFb^{-1,s+1}(P_R u)$ is the union of maximally extended $\GBB$s within the compressed characteristic set $\cchare$.

The same result holds for all $\nu > 0$ if $u \in \tdotSob^{1,m}_\loc(X)$ and we consider $P_D$ instead of $P_R$.
\end{theo}

	The notions of \emph{compressed characteristic set} and \emph{generalized broken bicharacteristics} (or $\GBB$s, see Definition \ref{def:gbb}) are defined relative to the conformally rescaled metric $\hat g = x^2 g$, and are exactly the same as used to describe propagation of singularities for smooth boundary value problems. Results  analogous to Theorem \ref{theo:GBBpropagation} were obtained in those settings by  Melrose, Sj\"ostrand, and Taylor \cite{taylor1976grazing,melrose1978singularities,sjostrand1981analytic,melrose1982singularities}; cf.~the works of Lebeau \cite{lebeau1997propagation} and Vasy \cite{vasy2004propagation} for the case of manifolds with corners, and of Melrose, Vasy and Wunsch \cite{melrose2008propagation} for edge manifolds. Here, the behaviour of $P$ at the boundary and the different nature of the boundary conditions pose particular difficulties, which we cope with by a systematic study of continuity properties of $\gamma_\pm$ and of the interactions of the $\b$-pseudo\-differential calculus with twisted derivatives.

%	For technical reasons, when $\nu \in (0,1)$ we actually only establish Theorem \ref{theo:GBBpropagation} for Neumann or Robin boundary conditions when there exists $F \in x^{\nu_-}\CI(X)$ satisfying \eqref{eq:twistingfunction}. Such an $F$ always exists when $\nu \neq 1/2$ (one can choose $F$ such that $F^{-1}P(F)$ vanishes to infinite order at $\pa X$), reflecting the fact that the indicial roots of $P$ do not differ by integer. In general, when $\nu = 1/2$ one can no longer choose $F \in x^{\nu_-}\CI(X)$ due to logarithmic corrections. Necessary and sufficient conditions are best stated in terms of a special bdf satisfying $\hat g^{-1}(dx,dx) = 1$ near $\pa X$, where $\hat g = x^2 g$ is the conformally rescaled metric. A valid twisting function $F \in x^{\nu_-}\CI(X)$ exists when $\pa X$ has vanishing mean curvature with respect to $\hat g$ (this is independent of the choice of special bdf). In particular, this is always the case when $g$ is a solution of the Einstein equations with negative cosmological constant. 

	Theorem \ref{theo:GBBpropagation} encodes the law of reflection when a $\GBB$ from the interior (where it is just an ordinary null-bicharacteristic of $g$ up to reparametrization) is transversally incident upon the boundary. In the case of tangential incidence our theorem is likely not optimal in the sense that it does not rule out null-bicharacteristics with finite-order contact sticking to the boundary (this problem was studied in a model case by Pham \cite{pham2013simple} for Dirichlet boundary conditions). The latter propagation phenomenon is automatically ruled out  when $g$ is a solution of the Einstein equations with negative cosmological constant, since the boundary is necessarily conformally totally geodesic (for null-geodesics, which is a conformally invariant notion).
	
Our framework allows us to define an operatorial $\b$-wave front set $\WFbop(\Lambda)$ for continuous operators $\Lambda : \tdotSob^{-1,-\infty}_{\rm c}(X)\to \mathcal{H}^{1,-\infty}_{\loc}(X)$, and to study the wave front set of \emph{induced operators on the boundary}, i.e.~of the form $\gamma_\pm \Lambda \gamma_\pm^*$. 

In Quantum Field Theory, one is particularly interested in \emph{two-point functions}, which in the setting of Robin or Neumann boundary conditions are pairs of continuous operators $\Lambda^\pm_R$ such that
\[
P_R \Lambda^\pm_R=  \Lambda^\pm_R P_R=0, \quad \Lambda^\pm_R\geq 0, \quad \Lambda^+_R-\Lambda^-_R  =i G_R,
\]
where $G_R$ is the difference between the retarded and advanced propagators (fundamental solution) for $P_R$. Following \cite{wrochna2017holographic}, we introduce a condition on the $\b$-wave front set of $\Lambda^\pm_R$ which we call the \emph{holographic Hadamard condition}, namely:
\[
\WFbop(\Lambda^\pm_R )\subset \dot\cN^\pm\times \dot\cN^\pm.
\]
Here $\dot\cN^\pm$ are the future- and past-directed components of $\dot\cN$ relative to a given time-orientation (the additional global geometric hypotheses are described in Section \ref{subsect:propagators}; note that these are not needed for Theorem \ref{theo:GBBpropagation}, which is purely local). We show: 

\begin{theo}[cf.~Theorems \ref{thm:holo1} and \ref{thm:holo2}]\label{thm:maintheo2} Two-point functions $\Lambda^\pm_R$ satisfying the holographic Hadamard condition exist and  are unique modulo terms whose Schwartz kernels are smooth in the interior of $X$. Furthermore, $\gamma_- \Lambda^\pm_R \gamma_-^*$ and $\gamma_+ \Lambda^\pm_R \gamma_+^*$ are unique modulo terms with smooth Schwartz kernels.  
\end{theo}

This extends the results of \cite{wrochna2017holographic} to Neumann and Robin boundary conditions, and thus provides the fundamental ingredients for constructing linear quantum fields and renormalized non-linear quantities in our setting. We remark that local conditions on singularities of two-point functions in the interior $X^\circ$ were studied on special examples of aAdS spacetimes  by several authors \cite{kent2015hadamard,belokogne2016stueckelberg,dappiaggi2016hadamard}. Our holographic Hadamard condition has, however, the advantage of giving enough information to define and study the induced operators on the boundary, $\gamma_- \Lambda^\pm_R \gamma_-^*$ and $\gamma_+ \Lambda^\pm_R \gamma_+^*$. We also stress that the proof of the existence statement in Theorem \ref{thm:maintheo2} crucially relies on the fact that the holographic Hadamard condition propagates well thanks to Theorem \ref{theo:GBBpropagation}. 

Finally, we obtain a similar result on the uniqueness modulo $\b$-regularising terms of parametrices for $P_R$; see Theorem \ref{prop:wfs} for the precise statement.

The paper is structured as follows. In Section \ref{sec:bpsdo} we recall elementary definitions and facts on the $\b$-pseudo\-differential calculus. Section \ref{sec:fspaces} introduces the weighted Sobolev spaces and reviews continuity results for the weighted trace $\gamma_-$. In Section \ref{sec:KG} we discuss Green's formula for the Klein-Gordon operator and asymptotic expansions for approximate solutions, which allow us to define the weighted trace $\gamma_+$. In Section \ref{sec:bvp} we introduce the Dirichlet, Neumann and Robin problems and derive some microlocal estimates. The main steps of the proof of Theorem \ref{theo:GBBpropagation} are contained in Section \ref{sec:propagation}. Section \ref{sec:propagators} is devoted to propagators and their operatorial $\b$-wave front sets, and in particular to the proof of Theorem \ref{thm:maintheo2}.  

%Finally, Appendix \ref{app:elliptic} briefly discusses some properties of the elliptic analogue of the operator $P_R$ which are needed in Section \ref{sec:propagators}.

\section{b-Pseudodifferential operators}\label{sec:bpsdo}

\subsection{Basic definitions}

In this section we briefly discuss the theory of $\b$-pseudo\-differential operators, mostly to fix the relevant notation. The presentation closely follows \cite[Sections 2, 3]{vasy2008propagation}, where additional details and complete proofs can be found.

If $X$ is a manifold with boundary, let $\Psib^m(X)$ denote the algebra of properly supported b-pseudodifferential operators of order $m$. If $k \in \NN$, then $\Diffb^k(X) \subset \Psib^k(X)$. The corresponding symbol space is $S^m(\bT^*X)$, where $\bT^*X$ is the b-cotangent bundle over $X$. A priori, we consider $A \in \Psib^m(X)$ as a continuous map 
\begin{equation} \label{eq:boundedonCdI}
A: \CdI(X) \rightarrow \CdI(X)
\end{equation}
which extends to a continuous endomorphism of $\CI(X)$. 

The abstract sesquilinear pairing of $u \in \CmI(X)$ with $\varphi \in \CdI(X)$ (or more generally for the pairing between a space and its anti-dual) will be written $\langle u, \varphi \rangle$. 
For the remainder of this section we fix a positive $\CI$ density $\mu$ on $X$ in order to trivialize the density bundle. 
Then $u \in L^2_\loc(X)$ determines an element of $\CmI(X)$ by 
\[
\langle u, \varphi \rangle = \int_X u \cdot \bar \varphi \, d\mu
\]
 As discussed below, $\Psib^m(X)$ is closed under adjoints, so $A \in \Psib^m(X)$ also extends to continuous endomorphisms of $\CmI(X)$ and $\CdmI(X)$. The fact that the action of $A$ on $\CdmI(X)$ extends that on $\CI(X)$ comes from the fact that
\[
\langle Au, v\rangle = \langle u, A^* v \rangle
\] 
for $u, v\in \CI(X)$; in other words, there are no boundary terms when integrating by parts.

We recall the symbol isomorphisms for b-pseudodifferential operators. There is a principal symbol map $\bsymbol{m} : \Psib^m(X) \rightarrow S^m(\bT^*X)$ which descends to an isomorphism
\[
\bsymbol{m} : (\Psib^m/\Psib^{m-1})(X) \rightarrow (S^m/S^{m-1})(\bT^*X).
\]
The symbol map can be inverted explicitly by fixing a non-canonical quantization map $\Opb : S^m(\bT^*X) \rightarrow \Psib^m(X)$ such that $\bsymbol{m}(\Opb(A)) = a$ in $(S^m/S^{m-1})(\bT^*X)$.

If $A \in \Psib^{m}(X)$ and $B \in \Psib^{m'}(X)$, then $AB \in \Psib^{m+m'}(X)$ with principal symbol 
\[
\bsymbol{m+m'}(AB) = \bsymbol{m}(A)\cdot \bsymbol{m'}(B).
\]
Furthermore their commutator satisfies $[A,B] \in \Psib^{m+m'-1}$. To describe the principal symbol of $[A,B]$, observe that the Poisson bracket of $a \in S^m(\bT^*X)$ and $b \in S^{m'}(\bT^*X)$ restricted to the interior $T^*X^\circ$ extends by continuity up to the boundary as an element of $S^{m+m'-1}(\bT^*X)$. In local coordinates $(x,y^1,\ldots,y^{n-1})$ with dual b-momenta $(\sigma, \eta_1,\ldots, \eta_{n-1})$, this is just the expression
	\[
	\{a,b\} = \pa_\sigma a \cdot x\pa_{x} b - x\pa_x a \cdot \pa_\sigma b + \sum_{i=1}^{n-1} \pa_{y^i}a\cdot \pa_{\eta_i} b - \pa_{\eta_i}a \cdot \pa_{y^i}b.
	\]
	Then $\bsymbol{m+m'-1}([A,B]) = \{ \bsymbol{m}(A), \bsymbol{m'}(B)\}$. 
If $A^*$ denotes the formal adjoint of $A$ with respect to $\mu$, then $A^* \in \Psib^m(X)$, and $\bsymbol{m}(A^*) = \overline{\bsymbol{m}(A)}$. Finally, if $x$ is a bdf, then $x^{-s} A x^s \in \Psib^m(X)$ for each $s\in \CC$, and $\bsymbol{m}(x^{-s}Ax^s) = \bsymbol{m}(A)$.

We will frequently use the following terminology: a continuous map $\CdI(X) \rightarrow \CmI(X)$ is said to be supported in $U \subset X$ if its Schwartz kernel has support in $U\times U$. Then any compactly supported operator $A \in \Psib^0(X)$ defines a bounded map 
\[
L^2_\loc(X) \rightarrow L^2_\comp(X).
\]
More precisely, suppose that $A$ is supported in $K$ for $K \subset X$ compact and $U \subset X$ is a neighborhood of $K$. Then there exists $\chi \in \CcI(U)$ and a compactly supported $A' \in \Psib^{-\infty}(X)$ such that
\[
\| Au \|_{L^2(X)} \leq 2 \sup |\bsymbol{0}(A)| \| \chi u \|_{L^2(X)} + \| A'u \|_{L^2(X)}.
\]
Since $\Psib^0(X)$ is invariant under conjugation by powers of a bdf $x$, the same result is true if $L^2_\loc(X)$ is replaced by any weighted space $x^{r}L^2_\loc(X)$, where $r \in\RR$.

\subsection{Microlocalization} \label{subsect:b-microlocalization}
We say that $A \in \Psib^m(X)$ is elliptic at a point $q_0 \in \bT^*X\setminus 0$ if there exists $b \in S^{-m}(\bT^*X)$ such that 
\[
\bsymbol{m}(A) \cdot b - 1 \in S^{-1}(\bT^*X)
\] 
in a conic neighborhood of $q_0$. The set of elliptic points of $A$ will be written $\ellb(A) \subset \bT^*X \setminus 0$. We say that $A$ is elliptic on a conic set $U \subset \bT^*X \setminus 0$ if $U \subset \ellb(A)$.

Next, we define the operator b-wavefront set (or microsupport) of $B \in \Psib^m(X)$. Following \cite[Section 3]{vasy2008propagation}, it is important to give a uniform definition for \emph{bounded families} of operators (since b-pseudodifferential operators have conormal Schwartz kernels on a certain blow-up of $X\times X$, there is a natural Fr\'echet topology on $\Psib^m(X)$ which roughly corresponds to symbol seminorms. Thus it makes sense to speak of bounded subsets of $\Psib^m(X)$).

If $\mathcal{B} \subset \Psib^m(X)$ is bounded, we say that $q\notin \opWFb(\mathcal{B})$ if there exists $A \in \Psib^0(X)$ with $q \in \ellb(A)$ such that $A\mathcal{B}$ is \emph{bounded} in $\Psib^{-\infty}(X)$. This agrees with the usual definition of $\opWFb(B)$ when $\mathcal{B} = \{B\}$ consists of a single operator.
For bounded families $\mathcal{A}, \mathcal{B}$ we have the usual relations
\begin{equation} \label{eq:WFfamilyrelations} 
\begin{gathered}
\opWFb(\mathcal{A}+\mathcal{B}) \subset \opWFb(\mathcal{A}) \cup \opWFb(\mathcal{B}),\\
\opWFb(\mathcal{A}\mathcal{B}) \subset \opWFb(\mathcal{A}) \cap \opWFb(\mathcal{B}).
\end{gathered}
\end{equation}
Furthermore, for operators,
\begin{equation} \label{eq:WFrelations}
\opWFb(A^*) = \opWFb(A), \quad \opWFb(x^{-s}Ax^s) = \opWFb(A).
\end{equation}
Next, we introduce some useful but non-standard terminology: if $\mathcal S \subset \Psib^m(X)$ is a closed subspace, we say that a bounded linear map $M : \mathcal S \rightarrow \Psib^{k}(X)$ is \emph{microlocal} if 
\[
\WFb'(M(A)) \subset \WFb'(A)
\]
for all $A \in \mathcal S$. A typical $\mathcal{S}$ is the set of operators with support in a fixed compact set. Note that $M$ necessarily preserves bounded families as well. According to \eqref{eq:WFfamilyrelations}, \eqref{eq:WFrelations}, multiplication by a fixed operator, taking adjoints, and conjugation by a bdf are all microlocal maps.

Let $A\in \Psib^m(X)$ and $B \in \Psib^{m'}(X)$. If $A$ is elliptic on $\WFb(B)$, then the standard symbolic parametrix construction yields $F \in \Psib^{m'-m}(X)$ and $R, R'\in \Psib^{-\infty}(X)$ such that
\[
B = AF + R = FA + R'.
\]

We also need to mention the \emph{indicial family} of $A \in \Psib^m(X)$. For a fixed bdf $x$ and $v \in \CI(X)$, define
\[
\widehat{N}(A)(s)v = x^{-is} A(x^{is} u)|_{\pa X},
\] 
where $u \in \CI(X)$ is any function restricting to $v$. This definition is independent of the choice of extension $u$, and depends only mildly on $x$.  Given $u \in x^{is} \CI(X)$,
\[
(x^{-is} Au)|_{\pa X} = \widehat{N}(A)(s)(x^{-is}u|_{\pa X}).
\]
The indicial family $\widehat{N}(s)$ is an algebra homomorphism; in particular it satisfies 
\[
\widehat{N}(AB)(s) = \widehat{N}(A)(s) \circ \widehat{N}(B)(s).
\]
Furthermore, $\widehat{N}(A^*)(s) = \widehat{N}(A)(\bar s)^*$. Here the adjoint on the left is with respect to a $\CI$ density $\mu$ on $X$, whereas the adjoint on the right is with respect to the pullback $\mu|_{\pa X}$. This can also be rewritten as
\[
\widehat{N}(A^*)(s) = \widehat{N}(x^{-2\Im s}Ax^{2\Im s})(s)^*.
\]
Using the indicial family it is possible to prove the following facts: let $U$ be a boundary coordinate patch with coordinates $(x,y^1,\ldots,y^{n-1})$, and suppose $A \in \Psib^m(X)$ has support in a compact set $K \subset U$. As emphasized in \cite[Section 2]{vasy2008propagation} (see \cite[Lemma 2.2]{vasy2008propagation} in particular), there exist $A', A'' \in \Psib^m(X)$ such that 
\begin{equation} \label{eq:Dxswap}
D_x A = A' D_x + A''.
\end{equation}
Indeed, one can take $A' = x^{-1}Ax$ and $A'' = x^{-1}[xD_x,A]$, the key here being that $[xD_x, A] \in x\Psib^m(X)$, as seen by analyzing its indicial family. In particular, the maps $A \mapsto A'$ and $A \mapsto A''$ are microlocal.
Since $\bsymbol{m}(A') = \bsymbol{m}(A)$, one deduces as in \cite[Lemma 2.8]{vasy2008propagation} that the commutator $[D_x,A]$ can be written in the form
\begin{equation} \label{eq:Dxcommutator}
[D_x, A] = A_1 D_x + A_0,
\end{equation}
where $A_j \in \Psib^{m-j}(X)$. Here $\bsymbol{m-1}(A_1) = (1/i)\pa_\sigma a$ and $\bsymbol{m}(A_0) = (1/i)\pa_x a$. Again, the maps $A\mapsto A_0$ and $A \mapsto A_1$ are microlocal.

\section{Function spaces}\label{sec:fspaces}

\subsection{Twisted derivatives} \label{subsect:twistedderivatives}
Let $X$ be an $n$-dimensional manifold with boundary. It will also be convenient to fix some bdf $x$. Motivated by \cite{warnick2013massive}, we define certain \emph{twisted} differential operators. Given $\nu\in \RR$, define $\nu_\pm = \tfrac{n-1}{2} \pm \nu$, and then set
\[
\tDiff^1(X) = \{x^{\nu_-}B x^{-\nu_-}: B \in \Diff^1(X)\}.
\] 
This space is independent of the choice of bdf $x$. The dimension-dependent shift between $\nu$ and $\nu_\pm$ is merely due to our eventual choice of weighted $L^2$ space.

Of course $\tDiff^1(X) \subset \Diff^1(X^\circ)$, but twisted differential operators do not necessarily have coefficients that are smooth up to $\pa X$. On the other hand, if $Q \in \tDiff^1(X)$, then
\[
Q: \CmI(X) \rightarrow \CmI(X)
\]
is continuous, since this is true for multiplication by any power of $x$. 
One should think of $\tDiff^1(X)$ as the largest space of differential operators preserving $x^{\nu_-}\CI(X)$. This is in contrast to the much smaller space $\Diffb^1(X)$, which preserves $x^s\CI(X)$ for \emph{every} $s$. More precisely, we have the following:

\begin{lemm} \label{lem:twistedareb} $\Diffb^1(X) \subset \tDiff^1(X) \subset x^{-1}\Diffb^1(X)$ for each $\nu \in \RR$.
	\end{lemm} 
\begin{proof} 
For the first inclusion, any $A \in \Diffb^1(X)$ can be written as 
\[
A = x^{\nu_-}(x^{-\nu_-} A x^{\nu_-}) x^{-\nu_-} \in \tDiff^1(X).
\] 
For the second inclusion it suffices to work in local coordinates $(x,y^1,\ldots,y^{n-1})$ near the boundary, where this is just the observation that $x(x^{\nu_-}\pa_x x^{-\nu_-}) = x\pa_x - \nu_- \in \Diffb^1(X)$.
\end{proof}

Finally, notice that $\tDiff^1(X)$ is a $\CI(X)$-module under both left and right multiplication, hence also closed under commutators. Since every operator in $\Diff^1(X)$ is the sum of a vector field and a multiplication operator, and since the space of $\CI$ vector fields is finitely generated over $\CI(X)$, it follows that $\tDiff^1(X)$ is also finitely generated.

\subsection{Twisted Sobolev spaces}

We now restrict to $\nu \in (0,1)$ and define the relevant function spaces. Given a $\CI$ density $\mu_0$ (to be fixed later on) on an $n$-dimensional manifold $X$ with boundary, define $\mu = x^{2-n}\mu_0$ and set
\[
\Ltwo(X) := x^{(n-2)/2}L^2(X,\mu_0) = L^2(X,\mu),
\] 
where $x$ is a fixed global bdf. This space depends on $\mu$ and $x$, but of course the local versions $\Ltwo_\loc(X)$ or $\Ltwo_\comp(X)$ do not. We also sometimes write $\tSob^0_\loc(X)$ for $\Ltwo_\loc(X)$.

 The twisted Sobolev spaces corresponding to $\tDiff^1(X)$ are defined by the following condition: given $u \in \CmI(X)$,
 \[
 u \in \tSob^1_\loc(X) \Longleftrightarrow Qu \in \Ltwo_\loc(X) \text{ for all } Q \in \tDiff^1(X).
 \]
 Given a compact subset $K \subset X$ write $\tSob^1(K)$ for the set of $u \in \tSob^1_\loc(X)$ with support in $K$. We also set 
 \[
 \tSob^1_\comp(X) = \tSob^1_\loc(X) \cap \CmI_\comp(X).
 \]
 Note that $\tSob^1_\loc(X)$ is a local space in the sense that $u \in \tSob^1_\loc(X)$ implies $\phi u \in \tSob^1_\loc(X)$ for every $\phi \in \CcI(X)$.

 As noted in the previous section, $\tDiff^1(X)$ is finitely generated over $\CI(X)$. Fixing a generating set $Q_1,\ldots, Q_N$, we equip $\tSob^1_\loc(X)$ with the family of seminorms
 \[
 u\mapsto \| \phi u \|_{\Ltwo(X)} + \sum_{i=1}^N \| \phi  Q_i  u \|_{\Ltwo(X)}, \quad \phi \in \CcI(X).
 \]
The restriction of these seminorms to $\tSob^1(K)$ for $K \subset X$ compact defines a Hilbert space topology; we write
\[
\| u \|^2 _{\tSob^1(X)} = \| u \|_{\Ltwo(X)}^2 + \sum_{i=1}^N \| Q_i u \|_{\Ltwo(X)}^2
\]
when $u$ has compact support. Then $\tSob^1_\comp(X)$ is equipped with the inductive limit topology corresponding to $\tSob^1(K)$ as $K$ ranges over all compact subsets of $X$. Because we are assuming $\nu \in (0,1)$,
\[
x^{\nu_-} \CI(X) \subset \Ltwo_\loc(X).
\]
Since by construction $\tDiff^1(X)$ preserves $x^{\nu_-}\CcI(X)$, it follows that $x^{\nu_-}\CI(X) \subset \tSob^1_\loc(X)$.

\begin{rema} \label{rem:globalhilbertspace} 
Sometimes it is useful to have a globally defined Hilbert space, at least when $X=\RR^n_+$. We use standard coordinates $(z^0,\ldots,z^{n-1})$ on $\RR^n_+ = \RR_+ \times \RR^{n-1}$, where $z^0 = x \in \RR_+$, and define 
\[
\Ltwo(\RR^n_+) = x^{(n-2)/2} L^2(\RR^n_+)
\]
with respect to Lebesgue measure. Consider the coordinate vector fields conjugated by $x^{\nu_-}$,
\begin{equation} \label{eq:flatQi}
Q_i = x^{\nu_-}D_{z^i} x^{-\nu_-}.
\end{equation}
Along with the constant function these generate $\tDiff^1(\RR^n_+)$. Of course in this case $Q_i = D_{z^i}$ for $i \neq 0$. 

We say that $u \in \tSob^1(\RR^n_+)$ if $u \in \Ltwo(\RR^n_+)$ and $Q_i u \in \Ltwo(\RR^n_+)$ for each $i = 0,\ldots,n-1$. This is a Hilbert space under the obvious norm. Note that $u \in \tSob^1(\RR^n_+)$ is equivalent to
\begin{equation} \label{eq:H1characterization}
u \in x^r L^2(\RR_+; H^1(\RR^{n-1})), \quad (xD_x  + i\nu_-)u \in x^{r+1}L^2(\RR_+;L^2(\RR^{n-1}))
\end{equation}
when $r = (n-2)/2$ and the Euclidean $L^2$ spaces are defined with respect to Lebesgue measure. If $u \in \CmI(X)$ has compact support in a boundary coordinate patch $U_\kappa$, then $u \in \tSob^1_\loc(X)$ is equivalent to $u \circ \kappa \in \tSob^1(\RR^n_+)$.
\end{rema}

Next, we discuss the relationship between $\tSob^1_\loc(X)$ and some weighted Sobolev spaces. For simplicity we first consider the case $X = \RR^n_+$.  If $u \in \tSob^1(\RR^n_+)$, then setting $v = x^{-\nu_-}u$ we have
\[
v \in x^{-\nu_-}\Ltwo(\RR^n_+), \quad D_{z^i} v \in x^{-\nu_-}\Ltwo(\RR^n_+)
\] 
for $i= 1,\ldots, n$. Weighted Sobolev spaces of this kind are well studied; see \cite{grisvard1963espaces} for example. In particular $\CcI(X)$ is dense in the corresponding weighted space by the usual translation, truncation, and mollification arguments. Since $\tSob^1(\RR^n_+)$ is a local space, this implies the density of $x^{\nu_-}\CcI(X)$ in $\tSob^1_\comp(X)$ for an arbitrary manifold with boundary $X$, hence also in $\tSob^1_\loc(X)$. 

From the identification with a weighted space as in the previous paragraph, it is possible to show that $u \in \tSob^1_\loc(\RR^n_+)$ admits a weighted trace
\begin{equation} \label{eq:traceminus}
 (x^{-\nu_-}u)|_{\pX} \in H^{\nu}(\RR^n_+),
\end{equation}
extended by continuity from $x^{\nu_-}\CcI(\RR^n_+)$. We give an alternative proof in Lemma \ref{lem:H1expansion} below, since the same methods will be used later on.

\begin{lemm} \label{lem:H1expansion}
Let $\nu \in (0,1)$, and set $r = (n-2)/2$. If $u \in \tSob^1(\RR^n_+)$, then the restriction of $u$ to any half-space $\{ x < \varepsilon\}$ admits an asymptotic expansion
	\begin{equation}  \label{eq:H1expansion}
	u = x^{\nu_-}  u_- +  x^{r+1} H^1_\be([0,\varepsilon); L^2(\RR^{n-1})),
	\end{equation}
	where $u_- \in H^\nu(\RR^{n-1})$. The map $u \mapsto \gamma_- u := u_-$ is continuous $\tSob^1(\RR^n_+) \rightarrow H^{\nu}(\RR^{n-1})$.
\end{lemm} 
\begin{proof}
Given $\varepsilon > 0$, let $\phi \in \CcI(\RR_+)$ be such that $\phi = 1$ near $\{x < \varepsilon\}$. Replacing $u$ with $\phi u \in \tSob^1(\RR^n_+)$, we can assume $u$ has compact support. Taking the Mellin transform, it follows from the second equation in \eqref{eq:H1characterization} that
\[
\mellin u (s) = (s + i \nu_-)^{-1}\mellin v(s),
\]
where $v = (xD_x + i\nu_-)u  \in x^{r+1}L^2(\RR_+;L^2(\RR^{n-1}))$ has compact support. Thus $\mellin v(s)$ is holomorphic in $\{\Im s > r-1/2\}$. 

If $\nu \in (0,1)$, then $(s+ i\nu_-)^{-1}\mellin v(s)$ has precisely one pole at $s = -i\nu_-$, so by standard contour deformation arguments we obtain
\[
u = x^{\nu_-}u_- + x^{r+1} H^1_\be([0,\varepsilon); L^2(\RR^{n-1}))
\]
when $x < \varepsilon$.
Note that $u_- \in L^2(\RR^{n-1})$ is just a scalar multiple of $\mellin v(-i\nu_-)$. By the complex interpolation method as described in \cite{mazzeo1991elliptic}, one actually has that $u_- \in H^{\nu}(\RR^{n-1})$. 

Continuity of the map $u \mapsto \gamma_- u$ follows by the closed graph theorem. Indeed, suppose that $u_j \rightarrow u$ in $\tSob^1(\RR^n_+)$ and $\gamma_- u_j \rightarrow \tilde u$ in $H^\nu(\RR^{n-1})$ (hence in distributions). If $\phi_1 \in \CcI(\RR_+)$ is identically on on $\supp \phi$, then we can replace $u_j$ with $\phi_1 u_j$ (which also converges to $u$), and hence assume $u_j$ has compact support. As noted above, $\gamma_-(u_j - u)$ is a multiple of $\mellin(v_j - v)(-i\nu_-)$, where 
\[
v_j = (xD_x + i\nu_-)u_j,
\] 
and $v_j \rightarrow v$ in $x^{r+1}L^2(\RR^n_+)$. Since $v_j - v$ has compact support, it follows by the Cauchy--Schwarz inequality that
\begin{align*}
|\mellin(v_j - v)(-i\nu_-)(y)| &\leq \int_{\RR_+} x^{\nu-1/2} |x^{-1-r}(v_j -v)(x,y)| \, dx 
\\ &\leq C\| x^{-1-r}(v_j-v)(\cdot,y) \|_{L^2(\RR_+)}
\end{align*}
for $\nu \in (0,1)$, which shows that $\gamma_- u_j \rightarrow \gamma_- u$ in $L^2(\RR^{n-1})$, hence also in distributions. This implies $\tilde u =\gamma_- u$, completing the proof.
\end{proof}

If $u$ admits a partial asymptotic expansion as in \eqref{eq:H1expansion}, then the coefficient $u_-$ is uniquely determined by $u$ in the sense that 
\[
u = 0 \Longrightarrow u_- = 0.
\]
Since elements of $x^{\nu_-}\CcI(\RR^n_+)$ certainly have an expansion as in \eqref{eq:H1expansion}, the map $\gamma_-$ agrees with the extension by continuity of the weighted restriction \eqref{eq:traceminus}.

Passing from $\RR^n_+$ to an arbitrary manifold with boundary $X$ by a partition of unity, Lemma \ref{lem:H1expansion} shows the existence of a continuous trace map 
\[
\gamma_-: \tSob^1_\comp(X) \rightarrow H^{\nu}_\comp(\pa X),
\]
hence also between the corresponding local spaces, extending $u \mapsto (x^{-\nu_-}u)|_{\pX}$ when $u \in x^{\nu_-}\CcI(X)$. When $\nu \neq 1/2$, the map $\gamma_-$ on $\tSob^1_\loc(X)$ depends on the choice of bdf $x$, but only in the mildest way: if $\tilde x = ax$ is another bdf with $a \in \CI(X)$ and $a > 0$, then
\[
\tilde \gamma_- u = (a|_{\pa X})^{-\nu_-} \cdot \gamma_- u.
\]
Here $\gamma_-$ and $\tilde \gamma_-$ are the traces defined with respect to $x$ and $\tilde x$, respectively.

\begin{lemm} \label{lem:traceinverse}
	There exists a continuous right inverse 
	\[
	\rho : H^\nu_\loc(\pa X) \rightarrow \tSob^1_\loc(X)
	\]
	for $\gamma_-$,  which restricts to a continuous map $H^\nu_\comp(\pa X) \rightarrow \tSob^1_\comp(X)$, and moreover takes $\CcI(\pa X)$ into $x^{\nu_-}\CcI(X)$. 
\end{lemm}
\begin{proof}
	First we construct a continuous right inverse $\rho_0: H^\nu(\RR^n_+) \rightarrow \tSob^1(\RR^n_+)$ for $\gamma_-$ on $\RR^n_+$ as follows: let $\phi \in \CcI([0,\varepsilon))$ be one near $x= 0$, and then specify the Fourier transform of $\rho_0(v)$ by
	\[
	\fourier(\rho_0(v))(x,\zeta) = \widehat v(\zeta)\cdot \phi(\langle \zeta \rangle x) x^{\nu_-}.
	\]
	Standard manipulations using a partition of unity allows for the construction of a suitable extension $\rho$ on a general manifold as well.
\end{proof}

From \eqref{eq:H1expansion}, the kernel of $\gamma_-$ on $\tSob^1_\loc(X)$ is just
\[
\ker \gamma_- = \tSob^1_\loc(X) \cap x\Ltwo_\loc(X),
\] 
and this latter space is independent of $\nu \in (0,1)$ (cf.~\eqref{eq:H1characterization}). In particular, taking $\nu = 1/2$, 
\begin{align*}
\tSob^1_\loc(X) \cap x\Ltwo_\loc(X) &= x^{n/2} H^1_{\loc}(X) \cap x\Ltwo_\loc(X),
\end{align*}
where $H^1_{\loc}(X)$ is the ordinary space of extendible Sobolev distributions on $X$.
In view of Hardy's inequality in one dimension, this equality also holds at the level of topologies. Applying Lemma \ref{lem:H1expansion} again, 
\[
\tSob^1_\loc(X) \cap x\Ltwo_\loc(X) =  x^{n/2} \dot H^1_\loc(X)
\]
by the well-known characterization of $\dot H^1_\loc(X)$ as the kernel of the usual smooth trace 
\[
\gamma : H^1_\loc(X) \rightarrow H^{1/2}_\loc(\pa X).
\] 
Indeed, if $\nu = 1/2$, then $\gamma_-$ agrees with $\gamma \circ x^{-n/2}$. If we let $\tdotSob^1_\loc(X)$ denote the closure of $\CdI_\comp(X)$ in $\tSob^1_\loc(X)$, it follows that
\[
\ker \gamma_- = \tdotSob^1_\loc(X).
\]
We then let $\tdotSob^1_\comp(X) = \tdotSob^1_\loc(X) \cap \CmI_\comp(X)$ as a subspace of $\tSob^1_\comp(X)$. We also obtain the following corollary of Hardy's inequality:

\begin{lemm} \label{lem:1/xbounded}
	If $x \in \CI(X)$ is any bdf, then multiplication by $x^{-1}$ defines a bounded linear map $\tdotSob^{1}_\loc(X) \rightarrow \Ltwo_\loc(X)$.
\end{lemm}

We will also need a trace interpolation inequality. From  \cite[Lemma 4.2]{gannot2018elliptic}, if $u \in \tSob^1_\comp(X)$, then
\begin{equation} \label{eq:traceinterpolation}
\| \gamma_- u \|_{L^2(\pa X)} \leq C \| u \|_{\Ltwo(X)}^\nu \| u \|_{\tSob^1(X)}^{1-\nu},
\end{equation}
where $C>0$ depends only on the support of $u$. Using Young's inequality we conclude that for any $\varepsilon>0$ there exists $C_\varepsilon>0$ such that 
\begin{equation} \label{eq:traceinterpolation2}
\| \gamma_- u \|_{L^2(\pa X)}^2 \leq \varepsilon \| u \|^2_{\tSob^1(X)} + C_{\varepsilon} \| u\|_{\Ltwo(X)}^2.
\end{equation}
For a different proof of \eqref{eq:traceinterpolation2} see \cite[Appendix B.2]{warnick2015quasinormal}. We will use \eqref{eq:traceinterpolation2} frequently.
\begin{rema} \label{rem:traceinterpolation}
In the smooth setting one gets an even stronger version of \eqref{eq:traceinterpolation} by noting that for $u,v \in \CcI(X)$,
\begin{equation} \label{eq:smoothinterpolation}
\langle \gamma u, \gamma v\rangle_{L^2(\partial X)}  = \langle R_0 u , v \rangle_{L^2(X)}   - \langle u, R_1 v \rangle_{L^2(X)}
\end{equation}
for suitable $R_0,R_1 \in \Diff^1(X)$. Suppose that $A \in \Psib^m(X)$ is formally self-adjoint and $v = A u$. Trivially estimating \eqref{eq:smoothinterpolation} by Cauchy--Schwarz yields an upper bound of the form 
\[
 \| u \|_{H^1(X)} \| Au \|_{L^{2}(X)} + \| u \|_{L^2(X)} \| Au \|_{H^{1}(X)}.
\] 
On the other hand, one can also commute $A$ through $R_1$ to bound \eqref{eq:smoothinterpolation}  (modulo lower order terms) by $ \| u \|_{H^1(X)} \| Au \|_{L^{2}(X)}$ alone. One cannot obtain an analogous type of estimate in the twisted setting from \eqref{eq:traceinterpolation} directly, which will cause some slight complications later on.
\end{rema}

One can also define the spaces $\tSob^1_\loc(X)$ for $\nu \geq 1$ in exactly the same way, the important observation being that $\tSob^1_\loc(X) = \tdotSob^1_\loc(X)$ since $x^{\nu_-}$ is not square integrable for $\nu \geq 1$.

\subsection{Dual spaces}
For the dual spaces, we \emph{define}
\[
\tdotSob^{-1}_{\comp}(X) = [\tSob^1_\loc(X)]', \quad \tSob^{-1}_\comp(X) = [\tdotSob^1_\loc(X)]'  
\]
with their strong dual topologies. The duals of the spaces with compact supports are defined by exchanging the roles of the subscripts $\comp$ and $\loc$.
We use the notation $\tdotSob^{-1}(X)$ in analogy with the smooth setting, where the dual of $H^1_\loc(X)$ is $\dot H^{-1}_\comp(X) \subset \CdmI_\comp(X)$, the corresponding space of \emph{supported} distributions. In this slightly more general setting, when $\nu \in (0,1)$,
\[
\tdotSob^{-1}_\comp(X) \subset x^{-\nu_-} \CdmI_\comp(X)
\]
by transposition of the dense inclusion $x^{\nu_-}\CI(X) \subset \tSob^1_\loc(X)$. Similarly, there is an inclusion $\tdotSob^{-1}_\loc(X) \subset x^{-\nu_-} \CdmI(X)$. Identifying $\Ltwo(X)$ with its own anti-dual gives rise to inclusions 
\[
\Ltwo_\comp(X) \subset \tdotSob^{-1}_\comp(X), \quad \Ltwo_\loc(X) \subset \tdotSob^{-1}_\loc(X).
\] 
The situation for $\tSob^{-1}_\loc(X)$ is simpler, since it can be identified with a subspace of $\CmI(X)$ by transposing the dense inclusion $\CdI_\comp(X) \subset \tSob^1_\comp(X)$. A similar comment applies to $\tSob^{-1}_\comp(X) \subset \CmI_\comp(X)$. Given a compact set $K \subset X$, we write $\tdotSob^{-1}(K)$ and $\tSob^{-1}(K)$ for the subsets of elements with compact support in $K$.

\subsection{Interaction with the b-calculus} \label{subsect:interaction}
In this section we discuss the interaction between $\tDiff^1(X)$ and $\Psib^m(X)$. Throughout, we will use the notion of a \emph{smooth twisting function} $F$, meaning that
\[
F \in x^{\nu_-}\CI(X;\RR ), \quad x^{-\nu_-}F > 0 \text{ on } X.
\]
Note that $F$ itself is not a smooth function, but rather its polyhomogeneous expansion is smooth in the sense of not containing any logarithmic terms.
Thus $Q \in \tDiff^1(X)$ if and only if it is of the form $Q = F B F^{-1}$ for some $B \in \Diff^1(X)$. Of course $x^{\nu_-}$ is itself a valid smooth twisting function.

We also introduce some notation for the coordinate vector fields twisted by $F$ (cf.~Remark \ref{rem:globalhilbertspace}). Given fixed local coordinates $(z^0,\ldots,z^{n-1})$ on a local coordinate patch $U \subset X$ (not necessarily a boundary coordinate patch) we use the notation
\[
Q_i = FD_{z^i}F^{-1}.
\]
Suppose that $A\in \Psib^m(X)$ has compact support in $U$. First, note that
\begin{align}
\begin{split} \label{eq:Qcommutator}
[Q_i, A] &= F(D_{z^i}F^{-1}A F -  F^{-1} A F D_{z^i})F^{-1} 
\\ &= F[D_{z^i}, F^{-1}AF]F^{-1}.
\end{split}
\end{align}
Because $F$ is a smooth twisting function, $F^{-1}A F \in \Psib^m(X)$ with the same principal symbol as $A$. Now consider the special case of boundary coordinates $(x,y^1,\ldots,y^{n-1})$, so that $Q_0 = F D_x F^{-1}$ in the notation above. Combined with \eqref{eq:Dxcommutator}, we obtain the following:

\begin{lemm} \label{lem:Q0commutator}
	Let $A \in \Psib^{m}(X)$ have compact support in $U$ with $a = \bsymbol{m}(A)$. There exist $A_1 \in \Psib^{m-1}(X)$  and $A_0 \in \Psib^m(X)$ such that
	\[
	[Q_0,A] = A_1 Q_0 + A_0,
	\]
	where $\bsymbol{m-1}(A_1) = (1/i)\pa_\sigma a$ and $\bsymbol{m}(A_0) = (1/i)\pa_x a$. The maps $A \mapsto A_0$ and $A \mapsto A_1$ are microlocal. Furthermore,
	\begin{equation} \label{eq:Q0swap}
	Q_0 A = A' Q_0 + A''
	\end{equation} 
	for some $A', A'' \in \Psib^m(X)$. The maps $A \mapsto A'$ and $A\mapsto A''$ are microlocal. 
\end{lemm}

Recall in the next lemma that all pseudodifferential operators are assumed to have proper support.

\begin{lemm}[{cf.~\cite[Lemma 3.2, Corollary 3.4]{vasy2008propagation}}] \label{lem:boundedonsobolev}
	Each $A \in \Psib^0(X)$ defines a continuous linear map
	\[
	\tSob^1_\loc(X) \rightarrow \tSob^1_\loc(X), \quad \tdotSob^1_\loc(X) \rightarrow \tdotSob^1_\loc(X).
	\] 
	By duality, $A$ extends to a continuous map 
	\[
	\tdotSob^{-1}_\comp(X) \rightarrow \tdotSob^{-1}_\comp(X),\quad \tSob^{-1}_\comp(X) \rightarrow \tSob^{-1}_\comp(X).
	 \]
	 The same result holds with the roles of the subscripts $\loc$ and $\comp$ reversed.
\end{lemm}

Here the action of $A \in \Psib^0(X)$ on $\tSob^{-1}_\comp(X)$ is by duality, namely $\langle Af, v \rangle = \langle f, A^*v \rangle$ for $v \in \tSob^1_\loc(X)$. This is just the restriction of the action of $A$ on the larger space $\CmI(X)$. The action of $A$ on $\tdotSob^{-1}_\comp(X)$ is also by duality. Recall that for $\nu \in (0,1)$ there is an identification 
\[
\tdotSob^{-1}_\comp(X) \subset x^{-\nu_-}\CdmI_\comp(X),
\]
and any $A \in \Psib^m(X)$ acts on $x^{-\nu_-}\CdmI_\comp(X)$ by the formula
\[
A (x^{-\nu_-}v) = x^{-\nu_-}(x^{\nu_-}A x^{-\nu_-}) v.
\]
In this sense the action of $A$ on $\tdotSob^{-1}_\comp(X)$ is given by restriction from the larger space $x^{-\nu_-}\CdmI(X)$. 

The proof of Lemma \ref{lem:boundedonsobolev} gives even more information: if $A \in \Psib^0(X)$ has compact support in $U \subset X$, then there exists $\chi \in \CcI(U)$ such that
\begin{equation} \label{eq:H1boundedbyseminorm}
\| Au \|_{\tSob^k(X)} \leq C\|\chi u \|_{\tSob^k(X)}
\end{equation}
for every $u \in \tSob^k_\loc(X)$ with $k = \pm1$, where the constant $C>0$ is bounded by a seminorm of $A$ in $\Psib^0(X)$. Similar estimates holds for $\tdotSob^k$ spaces, again with constants bounded by a seminorm of $A$ in $\Psib^0(X)$.

In exact analogy with \cite[Definition 3.5]{vasy2008propagation}, for $k =0,\pm 1$ we define subspaces of $\tSob^k(X)$ (or $\tdotSob^k(X)$) with additional regularity as measured by operators in $\Psib^m(X)$.

\begin{defi}
	Let $k = 0,\pm 1$ and $m\geq 0$. Given $u \in \tSob^k_\loc(X)$, we say that $u \in \tSob^{k,m}_\loc(X)$ if $Au \in \tSob^k_\loc(X)$ for all $A \in \Psib^m(X)$. When $m=\infty$, we define
	\[
	\tSob^{k,\infty}(X) = \bigcap_{m} \tSob^{k,m}(X).
	\]	
The spaces $\tdotSob^{k,m}(X)$ for $k= \pm 1$ are defined analogously.
\end{defi}

For finite $m$ it suffices to check that $u \in \tSob^k_\loc(X)$ and $Au \in \tSob^k_\loc(X)$ for a single elliptic $A$ (cf.~\cite[Remark 3.6]{vasy2008propagation}). The corresponding spaces with compact supports are defined in the obvious way; for $u \in \tSob^{k,m}_\comp(X)$ with $m \geq 0$ we set
\[
\| u \|_{\tSob^{k,m}(X)} = \| u \|_{\tSob^k(X)} + \| Au \|_{\tSob^k(X)},
\] 
for a choice of elliptic $A$, with the analogous definition for $u \in \tdotSob^{k,m}_\comp(X)$. 

It is also true that the weighted trace $\gamma_-$ maps $\tSob^{1,m}_\loc(X)$ into $H^{\nu+m}_\loc(X)$ (with continuity following by the closed graph theorem) for $m\geq 0$. This follows from
\begin{equation} \label{eq:gamma_-commute}
\gamma_- (Au) = \widehat{N}(A)(-i\nu_-)(\gamma_- u)
\end{equation}
and the fact that $\widehat{N}(A)(-\nu_-) \in \Psi^m( \pa X)$ is elliptic whenever $A \in \Psib^m(X)$ is elliptic. Indeed, \eqref{eq:gamma_-commute} holds when $m=0$ by the density of $x^{\nu_-}\CI(X)$ in $\tSob^{1}_\loc(X)$, and for $m>0$ by a regularization argument.

One can also give a definition of the spaces $\tSob^{k,m}(X)$ or $\tdotSob^{k,m}(X)$ for $m < 0$ as follows:

\begin{defi} \label{defi:negativeorderspaces}
	Let $k = \pm 1$ and $m < 0$; fix an elliptic $A \in \Psib^{-m}(X)$. We let $\tSob^{k,m}_\loc(X)$ denote the set of $u\in \CmI(X)$ of the form 
	\[
	u = u_1 + Au_2,
	\]
	where $u_1, \, u_2 \in \tdotSob^{k}_\loc(X)$. The same definition applies to $\tdotSob^{k,m}_\loc(X)$, where $\tdotSob^{-1,m}_\loc(X)$ is considered as a subspace of $x^{-\nu_-}\CdmI(X)$.
\end{defi} 

This definition is independent of the choice of elliptic $A$. Again the spaces with compact supports are defined in the obvious way; in that case one can choose $u_1, u_2$ with compact support (for a more detailed analysis of supports, see the discussion preceding \cite[Definition 3.17]{vasy2008propagation}). For $u \in \tSob^{k,m}_\comp(X)$ with $m < 0$ we set
\[
\| u \|_{\tSob^{k,m}(X)} = \inf\{ \| u_1\|_{\tSob^{k}(X)} + \| u_2\|_{\tSob^{k}(X)}: u = u_1 + Au_2 \},
\]
taking $u_1, u_2$ with compact supports, with the analogous definition for $u \in \tdotSob^{k,m}_\comp(X)$.

Extending Lemma \ref{lem:boundedonsobolev}, one can show that any $A\in \Psib^0(X)$ defines a continuous map between the spaces 
\[
\tSob^{k,m}_\loc(X) \rightarrow \tSob^{k,m}_\loc(X), \quad \tdotSob^{k,m}_\loc(X) \rightarrow \tdotSob^{k,m}_\loc(X)
\]
 for $k = 0,\pm1$ and any $m \in \RR$ (cf.~\cite[Lemma 5.8]{vasy2012wave}). Furthermore, the analogue of \eqref{eq:H1boundedbyseminorm} holds. For instance, if $A \in \Psib^0(X)$ has compact support in $U \subset X$, then there exists $\chi \in \CcI(U)$ such that
\[
\| Au \|_{\tSob^{k,m}(X)} \leq C\|\chi u \|_{\tSob^{k,m}(X)}
\]
for every $u \in \tSob^{k,m}_\loc(X)$, where the constant $C>0$ is bounded by a seminorm of $A$ in $\Psib^0(X)$. 

We will also need to extend the definition of $\gamma_-$ to $\tSob^{1,m}_\loc(X)$ for $m < 0$. It is easy to see that $x^{\nu_-}\CI(X)$ is dense in $\tSob^{1,m}_\loc(X)$ since $A$ as in Definition \ref{defi:negativeorderspaces} preserves the former space. Furthermore, the restriction of $\gamma_-$ to $x^{\nu_-}\CI(X)$ extends to a continuous map
\[
\tSob^{1,m}_\loc(X) \rightarrow H^{\nu+m}_\loc(X),
\]
which follows immediately from the definition of the $\tSob^{1,m}(X)$ norm and the fact that $\widehat{N}(A)(-i\nu_-) \in \Psi^{-m}(\pa X)$ is elliptic on $\pa X$. In analogy with \cite[Remark 3.16]{vasy2008propagation}, given $u = u_1 + Au_2 \in \tSob^{1,m}_\loc(X)$ with $A \in \Psib^{-m}(X)$ elliptic and $u_1,u_2 \in \tSob^{1}_\loc(X)$, we equivalently have
\begin{equation} \label{eq:extendedtracemap}
\gamma_- u = \gamma_-u_1 + \widehat{N}(A)(-i\nu_-)(\gamma_- u_2) \in H^{\nu+m}_\loc(\pa X).
\end{equation}
In particular, this definition is independent of the choice of $A$ and $u_1,u_2$.

With these definitions $\tdotSob^{-1,m}_\comp(X)$ is identified with the dual of $\tSob^{1,-m}_\loc(X)$ via the $\Ltwo(X)$ inner product for each $m \in \RR$; similarly $\tSob^{-1,m}_\comp(X)$ is the dual of $\tdotSob^{1,-m}_\loc(X)$. One also shows that $Q \in \tDiff^1(X)$ defines a bounded map 
\[
\tSob^{1,m}_\loc(X) \rightarrow \tSob^{0,m}_\loc(X)
\] 
for each $m \in \RR$. This follows from the comments following Lemma \ref{lem:Q0commutator}. Since this is also true for the spaces with compact support, by duality (replacing $m$ with $-m$) the transposition $Q^*$ defines a bounded map
\[
Q^* :\tSob^{0,m}_\loc(X) \rightarrow \tdotSob^{-1,m}_\loc(X).
\]
In particular, any operator of the form $L = \sum_i R^*_i Q_i$ with $Q_i, R_i \in \tDiff^{1}(X)$ defines a bounded map $L : \tSob^{1,m}_\loc(X) \rightarrow \tdotSob^{-1,m}_\loc(X)$ for each $m \in \RR$:
\[
\langle Lu , v \rangle := \sum_i \langle Q_i u, R_i v \rangle, 
\]
where $u \in \tSob^{1,m}_\loc(X)$ and $v \in \tSob^{1,-m}_\comp(X)$.

There is a natural wavefront set corresponding to $\tSob^{k,m}(X)$:

\begin{defi}
	Let $k = 0,\pm1$, and assume $u \in \tSob^{k,r}_\loc(X)$ for some $r\in \RR$. Given $q \in \bT^*X \setminus 0$, we say that $q \notin \WFb^{k,m}(X)$ if there exists $A \in \Psib^m(X)$ such that
	\[
	q \in \ellb(A), \quad Au \in \tSob^k_\loc(X).
	\]
	When $m=+\infty$, we say that $q \notin \WFb^{k,\infty}(X)$ if there exists $A \in \Psib^0(X)$ such that $q \in \ellb(A)$ and  $Au \in \tSob^{k,\infty}_\loc(X)$.
\end{defi}

This wavefront set is microlocal in the sense that
\[
\WFb^{k,m}(Au) \subset  \WFb^{k,m-s}(u) \cup \opWFb(A)
\]
for each $A \in \Psib^s(X)$; the proof is identical to that of \cite[Lemma 3.9]{vasy2008propagation}. From the construction of microlocal parametrices, one also obtains the following quantitative version:

\begin{lemm} \label{lem:ellipticestimate}
Let $\mathcal{A}$ be a bounded family in  $\Psib^s(X)$ and $G \in \Psib^s(X)$ be such that 
\[
\opWFb(\mathcal{A}) \subset \ellb(G).
\]
Suppose further that $\mathcal{A}$ and $G$ have compact support in $U \subset X$. Let $m \in \RR$. Then there exists $\chi \in \CcI(U)$ and $C>0$ such that
\[
\| Au \|_{\tSob^1(X)} \leq C(\|Gu\|_{\tSob^1(X)} + \|\chi u \|_{\tSob^{1,m}(X)})
\]
for every $u \in \tSob^{1,m}_\loc(X)$ with $\WFb^{1,s}(u) \cap \opWFb(G) = \emptyset$ and every $A\in \mathcal{A}$.
\end{lemm}

We also make the following useful observation:
\begin{lemm} \label{lem:L2intermsofH1}
Let $\mathcal{A}$ be a bounded family in  $\Psib^s(X)$ and $G \in \Psib^{s-1}(X)$ be such that 
\[
\opWFb(\mathcal{A}) \subset \ellb(G). 
\]
 Suppose further that $\mathcal{A}$ and $G$ have compact support in  $U \subset X$. Let $m \in \RR$. Then there exists $\chi \in \CcI(U)$ such that
\[
\| Au \|_{\Ltwo(X)} \leq C(\| Gu \|_{\tSob^1(X)} + \| \chi u \|_{\tSob^{1,m}(X)})
\]
for every $u \in \tSob^{1,m}_\loc(X)$ with $\WFb^{1,s-1}(u) \cap \opWFb(G) = \emptyset$ and every $A \in \mathcal{A}$.
\end{lemm}
\begin{proof}
	By a microlocal partition of unity, we can assume that $U$ is a coordinate patch with coordinates $(x,y^1,\ldots,y^{n-1})$, and that at least one of the vector fields $\{xD_x, \, D_{y^1},\ldots, D_{y^{n-1}}\}$ is elliptic on $\opWFb(\mathcal{A})$. Then we can write
	\[
	A = \sum_{i=0}^{n-1} Q_i G_i + R
	\]
	where $G_i \in \Psib^{s-1}(X)$ has $\opWFb(G_i) \subset \ellb(G)$, and $R \in \Psib^{-\infty}(X)$. Since $Q_i \in \tDiff^1(X)$, the proof is complete.
\end{proof}

The point of this lemma is that $G$ has one order lower than the family $\mathcal{A}$.

\subsection{Logarithmic twisting functions} \label{subsect:logarithmic}

We will also need to consider more general twisting functions with logarithmic corrections. It should be stressed, however, that the material in this section is not needed in the any of the following three situations:
\begin{enumerate} \itemsep6pt
	\item $\nu \in (0,1) \setminus \{1/2\}$,
	\item $\nu = 1/2$ and a certain mean curvature assumption along $\pa X$ is satisfied,
	\item $\nu > 0$ but one is only interested in propagation of singularities with Dirichlet boundary conditions.
\end{enumerate}
The first two conditions are explained in more detail in Section \ref{subsect:kg} below, and the third in Section \ref{subsect:dirichletform}. We say that 
\[
F \in x^{\nu_-}\CI(X) + (x^{\nu_+}\log x)\CI(X)
\]
is a \emph{logarithmic twisting function} if $x^{-\nu_-} F > 0$. This is in contrast to the notion of a smooth twisting function from the previous sections, where the logarithmic terms were absent.  Throughout this section we assume $\nu \in (0,1)$. First we work on $\RR^n_+$ with standard coordinates $(z^0,\ldots,z^{n-1})$, where $x = z^0 \in \RR_+$ and $y^\alpha = z^{\alpha}$ for $\alpha = 1,\ldots,n-1$.

\begin{lemm} \label{lem:H1multiplier}
	If $g \in L^\infty(\RR^n_+)$ satisfies $D_{y}^\alpha g \in L^\infty(\RR^n_+)$ for $|\alpha| \leq 1$ and
\[
	x^{\nu_-} \| D_x g(x,\cdot ) \|_{L^\infty(\RR^{n-1})}  \in x^{(n-2)/2} L^2(\RR_+),
\]
	then multiplication by $g$ defines a bounded map $\tSob^1(\RR^n_+) \rightarrow \tSob^1(\RR^n_+)$.
\end{lemm}
\begin{proof}
It is clear that $D^\alpha_y g : \tSob^1(\RR^n_+) \rightarrow \Ltwo(\RR^n_+)$ is continuous for $|\alpha| \leq 1$. Next, let $u \in \tSob^1(\RR^n_+)$ and let $u_n \in x^{\nu_-}\CcI(\RR^n_+)$ converge to $u$ in $\tSob^1(\RR^n_+)$. We can use the Sobolev embedding from \cite[Proposition 1.1']{grisvard1963espaces} to conclude that 
\[
\sup_x x^{-\nu_-}\| u_n (x,\cdot) \|_{L^2(\RR^{n-1})} < C\| u_n \|_{\tSob^1(\RR^n_+)},
\]
where $C>0$ is independent of $n$. Now write
\[
\| x^{\nu_-} D_x (x^{-\nu_-}g u_n)\|_{\Ltwo(\RR^n_+)} \leq C \| u_n \|_{\tSob^1(\RR^n_+)} + \| D_x (g) u_n \|_{\Ltwo(\RR^n_+)}.
\]
The last term on the right hand side is estimate by
\[
 \| D_x (g) u_n \|_{\Ltwo(\RR^n_+)}^2  \leq \int_{\RR} \left( x^{-2\nu_-} \| u_n(x,\cdot )\|_{L^2(\RR^{n-1})}^2 \right) \left( x^{2\nu_-} \| D_x g(x,\cdot) \|_{L^\infty(\RR^{n-1})}^2\right) x^{2-n}dx.
\]
By hypothesis, we conclude that 
\[
\| x^{\nu_-} D_x (x^{-\nu_-}g u_n)\|_{\Ltwo(\RR^n_+)} \leq C \| u_n \|_{\tSob^1(\RR^n_+)}.
\]
In particular, the left-hand side is uniformly bounded $n$, since $u_n \rightarrow u$ in $\tSob^1(\RR^n_+)$. Passing to a subsequence shows that $x^{\nu_-}D_x x^{-\nu_-}u \in \Ltwo(\RR^n_+)$, and continuity of 
\[
x^{\nu_-}D_x x^{\nu_-}g : \tSob^1(\RR^n_+) \rightarrow \Ltwo(\RR^n_+)
\]
then follows from the closed graph theorem. 
\end{proof}

Next, consider a logarithmic twisting function $F$ on $\RR^n_+$, and assume that $x^{-\nu_-}D_y^\alpha F$ is bounded for $|\alpha| \leq 1$. We show that the function $x^{\nu_-}F^{-1}$ satisfies the hypotheses of Lemma \ref{lem:H1multiplier}. Since $(x^{-\nu_-}F)^{-2}$ is bounded it suffices to show that 
\[
x^{\nu_-} \| D_x (x^{-\nu_-}F) \|_{L^\infty(\RR^n_+)} \in x^{(n-2)/2}L^2(\RR).
\]
We can write $x^{-\nu_-}F = F_0 + (x^{\nu_+ - \nu_-}\log x)F_1$ with $F_0,F_1 \in \CI(X)$, and the proof follows since $\nu \in (0,1)$. Thus multiplication by $x^{-\nu_-}F$ is a multiplier on $\tSob^1(\RR^n_+)$. The following result is an immediate corollary

\begin{lemm}
	If $F$ is a logarithmic twisting function and $B \in \Diff^1(X)$, then 
	\[
	FBF^{-1} : \tSob^1_\loc(X) \rightarrow \Ltwo_\loc(X)
	\]
	is bounded.
\end{lemm}

Next we consider what happens when we conjugate $A \in \Psib^m(X)$ by a logarithmic twisting function $F$. This necessitates the use of b-pseudodifferential operators with \emph{conormal coefficients}. A function $a \in L^\infty_\loc(\bT^*X)$ is said to be a symbol with  conormal coefficients if in local coordinates $(x,y,\sigma,\eta)$ it satisfies
\[
|(x\pa_x)^k \pa_y^\alpha \pa_\sigma^\ell \pa_\eta^\beta a(x,y,\sigma,\eta)| \leq C_{k\ell\alpha\beta} \langle (\sigma,\eta)\rangle^{m-|\beta|-\ell}.
\]
Roughly speaking, the conormal calculus consists of the quantizations of conormal symbols; we denote by $\Psibc^m(X) \supset \Psib^m(X)$ the corresponding operators of order $m$. These operators form a filtered $*$-algebra just as in the case of smooth coefficients. Furthermore, each $A\in\Psibc^0(X)$ defines a bounded operator on $\Ltwo_\loc(X)$.

 The only subtlety arises in the definition of the normal operator or the indicial family of $A \in \Psibc^m(X)$. This will not pose a problem here, since we only consider operators 
\[
A \in \Psib^m(X) + x^\delta\Psibc^m(X).
\]
for some $\delta > 0$.
If $A = B + C$ with $B \in \Psib^m(X)$ and $C \in x^{\delta}\Psibc^m(X)$, then we set $\widehat{N}(A) := \widehat{N}(B)$.  With this definition \eqref{eq:gamma_-commute} still holds.

\begin{lemm}
	If $A \in \Psib^m(X)$ and $F$ is a logarithmic twisting function, then
	\[
	FAF^{-1} \in \Psib^m(X) + x^{2\nu - \varepsilon}\Psibc^m(X)
	\]
	for each $\varepsilon > 0$.
\end{lemm}

Using Lemma \ref{lem:H1multiplier}, it is easy to see that $A \in \Psib^0(X) + x^{2\nu - \varepsilon}\Psibc^0(X)$ maps $\tSob^1_\loc(X)$ to itself. This is because given such an $A$, we can still find $A'  \in \Psib^0(X) + x^{2\nu - \varepsilon}\Psibc^0(X)$ such that
\[
D_x A = A'D_x + A'',
\]
the only difference being that now $A'' \in \Psib^0(X) + x^{2\nu-\varepsilon-1}\Psibc^0(X)$. This is because the commutator of $xD_x$ with an element of $x^{2\nu-\varepsilon}\Psibc^0(X)$ does not gain an extra order of vanishing. On the other hand, $x^{2\nu-1-\varepsilon}$ maps $\tSob^1_\loc(X)$ to $\Ltwo_\loc(X)$ by arguing as in the proof of Lemma \ref{lem:H1multiplier} since $\nu \in (0,1)$ and $\varepsilon> 0$ is arbitrary. 

Also note that the union of $\Psib^{m}(X) + x^{2\nu-\varepsilon}\Psibc^m(X)$ over $m \in \RR$ forms a filtered $*$-algebra, which is moreover closed under taking parametrices of invertible elements. Thus if $A$ as above is elliptic on $\WFb(B)$ for some $B \in \Psib^k(X)$, then the microlocal parametrix one obtains dividing $B$ by $A$ can also be taken in $\Psib^{k-m}(X) + x^{2\nu-\varepsilon}\Psibc^{k-m}(X)$, and similarly for the residual terms. This makes it possible to test for $\WFb^{k,m}$ using these operators for $k=0,\pm 1$.

Lastly, consider the analogue of Lemma \ref{lem:Q0commutator}. The simplest way to generalize the latter is to note that it holds for operators of the form $A = F BF^{-1}$ with $B \in \Psib^m(X)$. This means that when carrying out commutator arguments with a chosen commutant $A$, it should actually be applied with $FAF^{-1}$ (which has the same principal symbol).

\section{Asymptotically anti-de Sitter spacetimes}\label{sec:KG}

\subsection{Basic definitions}\label{sec:aAdS}

Let $X$ be an $n$-dimensional manifold with boundary $\partial X$. Suppose that $X^\circ$ is equipped with a smooth Lorentzian metric $g$ of signature $(1,n-1)$. Assume that the following conditions are satisfied near $\partial X$:

\begin{enumerate}[label=(\Alph*)]\itemsep6pt
	\item If $x\in \CI(X)$ is a bdf, then $\hat g = x^2 g$ extends smoothly to a Lorentzian metric on $X$.
	\item The pullback $\hat g|_{\pa X}$ of $\hat g$ to the boundary has Lorentzian signature.
	\item $\hat g^{-1}(dx,dx) = -1$ on $\partial X$.
\end{enumerate}

These properties are independent of the choice of bdf. When all three are satisfied we say that $(X,g)$ is an \emph{asymptotically anti-de Sitter} (aAdS) spacetime. The pullback $\hat g|_{\pa X}$ is only determined by $g$ up to a conformal multiple (corresponding to a change of bdf), hence $\partial X$ is referred to as a \emph{conformal boundary}.

\subsection{Special bdfs} \label{subsect:specialbdf}

At various points it will be convenient to use a \emph{special bdf} $x \in \CI(X)$ such that
\begin{equation} \label{eq:specialbdf}
\hat g(dx,dx) = -1
\end{equation}
in a neighborhood of $\pa X$.
In fact, if $k_0$ is a representative of the conformal class of boundary Lorentzian metrics, then $x$ is uniquely determined near $\partial X$ by specifying that $k_0 = \hat g|_{\pa X}$. The existence of such an $x$ is well known. For a proof in the Riemannian (conformally compact) setting, see \cite[Lemma 5.2]{graham1991einstein}; the proof applies verbatim in the Lorentzian case.

Given $U \subset \pa X$ with compact closure, we can always choose a collar neighborhood of $U$ diffeomorphic to $[0,\varepsilon)\times U$ in which the special bdf $x \in \CI(X)$ is identified with projection onto the first factor. Since $\pa X$ will typically be non-compact, we make no claim about the uniformity of $\varepsilon = \varepsilon(U) >0$ as $U$ varies. With this identification near $U$,
\[
g = \frac{-dx^2 + k}{x^2},
\]
where $x\mapsto k(x)$ is a family of Lorentzian metrics on $\partial X$ depending smoothly on $x\in [0,\varepsilon)$ such that $k(0) = k_0$. In particular, one can choose local coordinates $(x,y^1,\ldots,y^{n-1})$ such that $x$ is a special bdf, and 
\[
\hat g^{-1}(dx,dy^{\alpha}) = 0 \text{ near } \pa X,
\]
where $\alpha = 1,\ldots,n-1$.
 We call coordinates of this form $\emph{special coordinates}$.

We say that $g$ is even modulo $\mathcal{O}(x^{2k+1})$ (in the sense of Guillarmou \cite{guillarmou2005meromorphic}) if the Taylor expansion of $k$ at $x=0$ contains only even terms modulo $\mathcal{O}(x^{2k+1})$. This condition is independent of the choice of special bdf  (cf. \cite[Lemma 2.1]{guillarmou2005meromorphic}). Furthermore, as shown in \cite[Section 2]{guillarmou2005meromorphic}, any two special coordinate systems $(x,y)$ and $(\tilde x, \tilde y)$ on overlapping coordinates patches are related by
\begin{equation} \label{eq:changeofspecialcoords}
\tilde x = x \sum_{j=0}^{k+1} a_j(y) x^{2j} + x^{2k+4}\CI(X), \quad \tilde y = \sum_{j=0}^{k+1} b_j(y) x^{2j} + x^{2k+3}\CI(X).
\end{equation}
Evenness of $g$ modulo $\mathcal{O}(x^3)$ is implicit in the works of \cite{holzegel2012well,warnick2013massive,holzegel2014boundedness}, and is verified for solutions to the Einstein equations. In particular, it implies that $\partial X$ is totally geodesic with respect to the conformal metric $\hat g$. In this paper we \emph{do not} insist on evenness assumptions for the metric.

\subsection{The Klein--Gordon operator} \label{subsect:kg}

Let $(X,g)$ be an asymptotically AdS spacetime, and consider the Klein--Gordon operator 
\[
P = \Box_g - \lambda, \quad  \lambda  < \tfrac{(n-1)^2}{4}.
\]
It is convenient to parametrize $\lambda = \tfrac{(n-1)^2}{4}-\nu^2$, with $\nu > 0$. As in Section \ref{subsect:twistedderivatives}, set $\nu_\pm = \tfrac{n-1}{2} \pm \nu$. Fixing an arbitrary representative $k_0$ in the conformal class of boundary Lorentzian metrics, we use the special bdf $x$ as in the previous section.

Since $\det g = x^{-n}\det \hat g$, it follows that in special coordinates on a coordinate patch $[0,\varepsilon)\times U$, 
\[
\Box_g = -(x\partial_x)^2 + (n-1)x\partial_x + (xE)x\partial_x+ x^2 \tilde  P,
\]
where $\tilde P = \Box_k$ is a family of second order differential operators on $\pa X$ depending smoothly on $x \in [0,\varepsilon)$, and $E \in \CI(X)$ is given by 
\begin{equation} \label{eq:logdetderivative}
E = -\partial_x (\log |{\det \hat g}|).
\end{equation}
Observe that up to a scalar multiple, $E|_{\pa X}$ is the mean curvature of $\pa X$ with respect to $\hat g$.
Using the product structure near the boundary, we can identify the \emph{normal operator} ${N}(P)$ of $P$ with an operator on $X$. Thus
\[
{N}(P) = (xD_x)^2 + i(n-1)xD_x - \lambda.
\]
In this case the indicial family $\widehat{N}(P)(s) = s^2 + i(n-1)s - \lambda$ can be identified with a scalar multiplication operator.

In general, the difference between $P$ and ${N}(P)$ in an operator in $x\Diffb^1(X)$. However, 
\begin{equation} \label{eq:normallowerordeterms}
P - {N}(P) = i(xE)xD_x + x^2\Diffb^2(X),
\end{equation}
so if $\pa X$ has zero mean curvature with respect to $\hat g$, then actually $P- N(P) \in x^2 \Diffb^2(X)$.

There are  two important $L^2$ spaces on $X$. The first is $L^2(X,dg) = x^{n/2}L^2(X,d\hat g)$, and the second is $L^2(X, x^{2}dg)$. It is this second space that is compatible with the normalization in Section \ref{subsect:twistedderivatives}, and unless specified explicitly, we use 
\[
\Ltwo(X) = L^2(X,x^{2}dg) = x^{-1}L^2(X,dg).
\]

The inner product of $u,v \in \Ltwo(X)$ will be denoted by $\langle u, v \rangle$, whereas the induced inner product of $f,g \in L^2(\pa X)$ coming from the volume density $dk_0$ will be denoted by $\langle f, g \rangle_{\pa X}$.

Recall the notion of a twisting function, either smooth or logarithmic, as in Section \ref{subsect:interaction}. For the next definition we fix such a function $F$.

\begin{defi} 
 Given $u \in \CdI(X)$, define the twisted differential $d_F u \in \CdI(X; T^* X)$ by
\[
d_F u = F d(F^{-1}u) = du + uF^{-1}\cdot  dF.
\]

\end{defi} 

The one-form $F^{-1} \cdot dF \in x^{-1}\CI(X;T^*X)$ can be thought of as a singular magnetic potential. Note that $d_Fu$ continues to make sense as a $T^*X$-valued distribution on $X$ when $u \in \CmI(X)$. In local coordinates  $(z^0,\ldots,z^{n-1})$, set
\[
Q_i = F D_{z^i} F^{-1},
\] 
so that $-i \cdot d_F u = (Q_i u) dz^i$. We define the twisted Dirichlet form for the metric by setting
\[
\mathcal{E}_0(u,v) = -\int G(d_F u, d_F \bar v)\,  dg = -\int \hat G(d_F u, d_F \bar v) \, x^2 dg,
\]
where $G$ is the induced inner product on the fibers of $T^*X$. Thus in coordinates, we have
\[
G(d_F u, d_F \bar v) =  - g^{ij} Q_i u \cdot Q_j \bar v.
\]
If $u,v \in \tSob^1_\loc(X)$, at least one of which has compact support, then $\mathcal{E}_0(u,v)$ is finite.
We now come to the main reason for introducing the twisting function: as observed in this context by \cite{warnick2013massive} and \cite{holzegel2014boundedness},
\[
P = -(d_F)^\dagger d_F + F^{-1}P(F),
\]
where $(d_F)^\dagger : \CdI(X; T^* X) \rightarrow \CdI(X)$ is the $dg$ adjoint of $d_F$. This is useful provided $F$ is chosen so that multiplication  by the function 
\[
S_F = F^{-1}P(F) \in \CI(X^\circ)
\]
is bounded $\Ltwo(X) \rightarrow x^2 \Ltwo(X)$. For instance, $S_F = \mathcal{O}(x^2)$ is sufficient. We thus make the following definition:

\begin{defi} \label{defi:admissible}
	We say that a twisting function, either smooth or logarithmic, $F$ is \emph{admissible} if $S_F \in x^2L^\infty_\loc(X)$. 
\end{defi}

We now discuss conditions under which it is possible to choose an admissible twisting function, which plays an important role when considering Neumann or Robin boundary conditions; for Dirichlet boundary conditions any smooth twisting function suffices, essentially due to Lemma \ref{lem:1/xbounded} (for more details, see Section \ref{subsect:dirichletform}). In particular, we only need to consider the case $\nu \in (0,1)$.

 First suppose that $\nu \neq 1/2$, so the difference between the indicial roots $\nu_+ - \nu_- = 2\nu$ of $P$ is not an integer. Then there always exists an admissible smooth twisting function, which can actually be chosen to satisfy 
 \[
 S_F \in \CdI(X).
  \] (Cf. \cite[Lemma 4.13]{vasy2012wave}). The terms in the polyhomogeneous expansion of $F$ are uniquely determined by specifying $x^{-\nu_-}F|_{\pa X}$ (which should be non-vanishing) for some choice of bdf $x$. More explicitly, let $x$ be a special bdf. Then in special coordinates, we can take
 \[
 F(x,y) = x^{\nu_-}(1 + xf(y)), \quad f(y) = \frac{\nu_- E(0,y)}{\widehat{N}(P)(-i(1+\nu_-))}.
 \]
 We can simply multiply this $F$ by a function of $y$ to specify the restriction $x^{-\nu_-}F|_{\pa X}$.
 
 On the other hand, if $\nu = 1/2$, then $\nu_+ - \nu_- = 1$ differ by an integer. We can still find a logarithmic twisting function, but not necessarily a smooth one.
 Following the arguments in \cite[Lemma 4.12]{vasy2012wave} and referring to \eqref{eq:normallowerordeterms}, the existence of an admissible smooth twisting function when $\nu = 1/2$ is equivalent to the vanishing of the mean curvature of $\pa X$ with respect to $\hat g$. This condition is independent of the choice of special bdf.  
 
 For simplicity, we henceforth focus on the case where there exists an admissible smooth twisting function. This is just to avoid some of the technicalities associated with operators with conormal coefficients as discussed Section \ref{subsect:logarithmic}.
 \begin{hypo} \label{hypo:twisting}
 	If $\nu =1/2$, then the mean curvature of $\pa X$ with respect to $\hat g$ vanishes. In other words, there exists a function $F \in x^{\nu_-}\CI(X)$ satisfying $x^{-\nu_-}F>0$, such that
 	\[
 S_F  \in x^2 \CI(X).
 	\]
 \end{hypo}

%The correct choice of twisting function is intimately related to asymptotic expansions of solutions to $Pu =0$ and is discussed in the next section.

It will also be convenient to define a sesquilinear pairing on one-forms relative to a fixed $\CI$ positive-definite inner product $H$. We use this to measure the inner product of differentials twisted by a fixed twisting function $F$ (to be chosen). Introduce the Dirichlet form associated to $H$ by
\[
\mathcal{Q}(u,v) = \int H(d_F u, d_F \bar v) \, x^{2} dg.
\] 
Thus the $\tSob^1$ norm-squared of $u \in \tSob^1_\comp(X)$ can be taken to be $\mathcal{Q}(u,u) + \| u \|^2_{\Ltwo(X)}$.

\subsection{Asymptotic expansions} \label{subsect:asymptoticexpansion}

Next, we discuss asymptotic expansions of solutions to the Klein--Gordon and related equations. Given an aAdS spacetime $(X,g)$ and $k \in \RR \cup \{\pm \infty\}$, define the spaces
\[
\maxdom^k = \{u \in \tSob^{1,k}_\loc (X): Pu \in x^2\tSob_\loc^{0,k}(X) \}. 
\]
We also abbreviate $\maxdom = \maxdom^0$. For finite $k$ we equip these spaces with the seminorms
\[
u \mapsto \| \phi u \|_{\tSob^{1,k}(X)} + \| x^{-2} \phi Pu \|_{\tSob^{0,k}(X)},
\]
and when $k=\infty$ we use the collection of all seminorms for all finite $k$. The main technical difficulty is that $\maxdom^k$ is \emph{not} closed under applications of $A \in \Psib^{0}(X)$. In fact, it is not closed under multiplication by arbitrary $\CcI(X)$ functions either. However, it turns out that closedness does  hold if we work with b-pseudodifferential operators satisfying certain evenness properties with respect to a special bdf.

Given an aAdS spacetime that is even modulo $\mathcal{O}(x^{2k+1})$, we introduce in Appendix \ref{app:even} a class of b-pseudodifferential operators $\Psibeven^m(X) \subset \Psib^m(X)$ whose coefficients are even modulo $\mathcal{O}(x^{2k+3})$. Note that every aAdS spacetimes is even modulo $\mathcal{O}(x)$, so this construction is always non-trivial. The definition mirrors a similar construction by Albin  within the 0-calculus \cite{albin2007renormalized}. The key property of $\Psibeven(X)$ is the following:

\begin{lemm} \label{lem:evenQcommutator}
	Let $B \in \Psibeven^m(X)$ have compact support in a boundary coordinate patch with special coordinates $(x,y^1,\ldots,y^{n-1})$. There exists $B' \in \Psib^{m}(X)$  and $B'' \in x \Psib^m(X)$ such that
	\[
	Q_0 FBF^{-1} = B' Q_0 + B''
	\]
\end{lemm}
\begin{proof}
	Recall that $Q_0 = FD_x F^{-1}$. It suffices to show that \eqref{eq:Dxswap} holds with $A'' \in x\Psib^m(X)$, since the result follows by conjugating \eqref{eq:Dxswap} by $F$. Tracing through the proof of \eqref{eq:Dxswap} in \cite[Section 2]{vasy2008propagation}, one can take 
	\[
	A'' = x^{-1}[xD_x,B].
	\] 
	The key here is that $[xD_x,B]$ has a vanishing indicial family, and hence lies in $x\Psib^m(X)$ in general. However, since $xD_x \in \Psibeven^1(X)$ and we are assuming that $B \in \Psibeven^m(X)$, it follows that 
	\[
	[xD_x,B] \in x^2 \Psib^m(X)
	\]
	according to Lemma  \ref{lem:quadraticvanishing}.
	This shows that $A'' \in x \Psib^m(X)$.
\end{proof}

 The additional order of vanishing when $A \in \Psibeven^m(X)$ is crucial when considering the action of $A$  on $\maxdom^k$.

\begin{lemm} \label{lem:preservesmaxdomain}
If $B \in \Psibeven^{m}(X)$, then $A = FBF^{-1}$ maps $\maxdom^k \rightarrow \maxdom^{k-m}$ for each $k \in \RR$. 
\end{lemm}
\begin{proof}
Since $A$ maps $\tSob^{1,k}_\loc(X) \rightarrow \tSob^{1,k-m}_\loc(X)$, the only additional point to verify is that 
	\[
	Pu \in x^2\tSob^{0,k}_\loc(X) \Longrightarrow PAu \in x^2\tSob^{0,k-m}_\loc(X).
	\]
	Now $PAu = APu + [P,A]u$, and since $APu \in x^2\tSob^{0,k-m}_\loc(X)$, it suffices to consider the commutator. Also, 
	\[
	x^{-2}[P,A] = [x^{-2}P,A] - [x^{-2},A]P
	\]
	and $[x^{-2},A] \in x^{-2}\Psib^{m-1}(X)$, so $[x^{-2},A]Pu \in \tSob^{0,k-m}_\loc(X)$. Recall that the distributional action of $x^{-2}P$ is given by
	\[
	x^{-2}P = Q_0^* Q_0 + K,
	\]
	where $K = Q_{\alpha}^*\hat g^{\alpha\beta} Q_\beta + x^{-2}S_F \in \Psib^2(X)$. Thus $[K,A] \in \Psib^{m+1}(X)$, so $[K,A]u \in \tSob^{0,k-m}_\loc(X)$. As for the final term,
	\begin{align*}
	[Q_0^*Q_0,A] &= Q_0^*[Q_0,A] + [Q_0^*,A]Q_0 \\
	& = Q_0^*[Q_0,A] - [Q_0,A^*]^*Q_0.
	\end{align*}
	Now use Lemma \ref{lem:evenQcommutator} to write $[Q_0,A] = A_1 Q_0 + A_0$, where $A_0 \in x\Psib^{m}(X)$ and $A_1 \in \Psib^{m-1}(X)$. As a consequence, we can write
	\[
	Q_0^*[Q_0,A] = Q_0^* (A_1Q_0+A_0) = (A'Q_0^* + A'')Q_0 + Q_0^* A_0,
	\]
	where now $A',A'' \in \Psib^{m-1}(X)$. For the term $Q_0^*A_0$, we use the fact that $Q_0^*x = Q_0x + \Psib^0(X)$ to see that $Q_0^* A_0 u \in \tSob^{0,k-m}_\loc(X)$. We can also write
	\[
	 (A'Q_0^* + A'')Q_0  = A'x^{-2} P - A'K + A''Q_0,
	\]
	so when applied to $u$ this is also in $\tSob^{0,k-m}_\loc(X)$. In conclusion, $Q_0^*[Q_0,A]u \in \tSob^{0,k-m}_\loc(X)$. The term $[Q_0,A^*]^*Q_0$ is handled analogously.
\end{proof}

One application of Lemma \ref{lem:preservesmaxdomain} is when $A$ is multiplication by a cutoff function $\chi \in \CI_{\comp,\mathrm{even}}(X)$ (see Appendix \ref{app:even} for notation). We can always find a partition of unity $\{\chi_i\}$ subordinate to a covering of  $X$ by either interior or special boundary coordinates patches, such that $\chi_i \in \CI_{\comp,\mathrm{even}}(X)$ in the latter case. This allows us to reduce the study of $\maxdom^k$ to a local one.

First we work locally, assuming that $X = \RR^n_+$. Assume that $g$ is an aAdS metric given in standard coordinates $(x,y)$ by
\begin{equation} \label{eq:aAdSRn}
g = \frac{-dx^2 + k^{\alpha\beta}(x,y)dy^\alpha dy^\beta}{x^2}.
\end{equation}
 We show that any $u \in \tSob^1(\RR^n_+)$ with compact support and satisfying $Pu \in x^2\Ltwo(\RR^n_+)$ admits a partial asymptotic expansion. Fix an admissible twisting function on $\RR^n_+$ satisfying $x^{-\nu_-}F = 1$ when $x=0$.

%\begin{lemm}  If $u \in \tSob^1(\RR^n_+)$ has compact support and $Pu \in x^{2}\Ltwo(\RR^n_+)$, then
%	\[
%	u \in x^{r}H^2_\be(\RR_+;H^{-1}(\RR^{n-1})),
%	\]
%	where $r = (n-2)/2$.
%\end{lemm} 
%\begin{proof}
%We have $u \in x^{r}H^1_\be(\RR_+;L^2(\RR^{n-1})) \cap x^r L^2(\RR_+;H^1(\RR^{n-1})) $, so it suffices to show that 
%\[
%(xD_x)^2 \in  x^{r}L^2(\RR_+; H^{-1}(\RR^{n-1})).
%\]
%But this follows from the equation $Pu \in x^2\Ltwo(X)$ since
%	\[
%	(xD_x)^2 u = Pu - (i(n-1)xD_x - \lambda + i(xE)xD_x + x^2 R)u,
%	\]
%	which certainly lies in $x^{r} L^2(\RR_+; H^{-1}(\RR^{n-1}))$.
%\end{proof}

\begin{lemm} \label{lem:graphnormexpansion}
Let $g$ be an asymptotically AdS metric on $\RR^n_+$ of the form \eqref{eq:aAdSRn}, and set $r = (n-2)/2$. If 
\[
u \in \tSob^{1,k}_\comp(\RR^n_+), \quad Pu \in x^2\tSob^{0,k}_\comp(\RR^n_+)
\] 
for $k \geq 0$, then the restriction of $u$ to any half-plane $\{ x < \varepsilon\}$ admits an asymptotic expansion
	\begin{equation} \label{eq:graphnormexpansion}
	u = Fu_- + x^{\nu_+}u_+ + x^{r+2}H^{k+2}_\be ([0,\varepsilon);H^{k-3}(\RR^{n-1})),
	\end{equation}
where $u_- \in H^{\nu+k}(\RR^{n-1})$ and $u_+ \in H^{-1-2\nu+k}(\RR^{n-1})$. Furthermore $u_- = \gamma_- u$.
%
%Moreover, if  $u \in \tSob^{1,\infty}_\comp(\RR^n _+)$ and $Pu \in x^2 \tSob^{0,\infty}_\comp(\RR^n_+)$, then $u_\pm \in \CI(\RR^{n-1})$ and the remainder lies in $x^{r+2}H^N_\be([0,\varepsilon); H^{N}(\RR^{n-1})$ for each $N$.
\end{lemm}
\begin{proof}
Assume that $k=0$; the general case is proved analogously. We have that $P$ is a differential operator of the form
\[
P = (xD_x)^2 + i(n-1)xD_x + i(xE)xD_x + x^2 \tilde P,
\]
with $E \in \CI(\RR^n_+)$ and $\tilde P \in \Diffb^2(\RR^n_+)$. Since $P - \widehat{N}(P) \in x\Diffb^2(X)$, it follows that
\[
\widehat{N}(P)u \in x^{r+1}L^2(\RR_+;H^{-1}(\RR^{n-1})).
\]
Applying the Mellin transform (which is uniformly square integrable on horizontal lines with $\Im s > r+\tfrac 1 2$) and deforming the contour to any horizontal line with $\Im s > r - \tfrac 1 2 $, we obtain 
\begin{equation} \label{eq:graphpartialexpansion1}
u = x^{\nu_-}u_- + u_1
\end{equation}
when $x < \varepsilon$, where $u_1 \in x^{r+1} H^2_\be([0,\varepsilon); H^{-1}(\RR^{n-1}))$ and $u_- \in H^{-1}(\RR^{n-1})$. As in Lemma \ref{lem:H1expansion}, by interpolation $u_- \in H^{2\nu-1}(\RR^{n-1})$. On the other hand, $u_- = \gamma_- u$, where $\gamma_-$ is the trace from Lemma \ref{lem:H1expansion}, since $u_-$ in a partial expansion of the form \eqref{eq:graphpartialexpansion1} is unique. This shows that in fact $u_- \in H^{\nu}(\RR^{n-1})$, which is a stronger statement since $\nu \in (0,1)$.

Let $\phi = \phi(x) \in \CcI(\RR_+)$ be identically one on $\supp u$. Replacing $u_1$ with $\phi u_1$, we can write $u = x^{\nu_-}\phi u_- + u_1$ on $\RR_+$, with $u_1$ of compact support. Then
\[
\widehat{N}(P)u_1 = Pu_1 - i(xE) xD_x u_1 - x^2 \tilde P u_1,
\]
and the latter two terms lie in $x^{r+2} L^2(\RR_+; H^{-3}(\RR^{n-1}))$ by the properties of $u_1$. On the other hand, 
\[
Pu_1 = Pu - P(x^{\nu_-} \phi u_-) = -P(x^{\nu_-}\phi u_-) + x^2 \Ltwo(\RR^n_+).
\]
Thus the only obstruction to having $ \widehat{N}(P) u_1 \in x^{r+2}L^2(\RR_+;H^{-3}(\RR^{n-1}))$ is the term $P(x^{\nu_-} \phi u_-)$, which a priori is merely in $x^{r+1}L^2(\RR_+; H^{-1}(\RR^{n-1}))$.

 This is remedied by replacing $u_-$ with $Fu_-$, which corresponds to replacing $u_1$ with $u_1 + (x^{\nu_-}-F)\phi u_-$. Thus
 \[
 u = F\phi u_- + u_1,
 \]
 where we still have $u_1 \in x^{r+1} H^2_\be(\RR_+; H^{-1}(\RR^{n-1}))$. Now we repeat the same argument to obtain an asymptotic expansion for $u_1$. When $x< \varepsilon$, this yields
\[
u = Fu_- + x^{\nu_+}u_+ + u_2,
\]
where $u_2 \in x^{r+2}H^2_\be([0,\varepsilon); H^{-3}(\RR^{n-1}))$ and $u_+ \in H^{-3}(\RR^{n-1})$. The a priori regularity of $u_+$ can be improved by interpolation to give $u_+ \in H^{-1-2\nu}(\RR^{n-1})$.
\end{proof}

Suppose that in Lemma \ref{lem:graphnormexpansion} we can take $k = \infty$. Then the residual term on the right-hand side of \eqref{eq:graphnormexpansion}, denoted in the proof by $u_2$, is in $x^2 \tSob^{0,\infty}_\loc([0,\varepsilon) \times \RR^{n-1})$. By Sobolev embedding, this implies that
\begin{equation} \label{eq:sobolevembedding}
Lu_2 \in x^{(n+1)/2}L^\infty_\loc([0,\varepsilon) \times \RR^{n-1})
\end{equation}
for each b-differential operator $L$. Next, we show that $\mathcal{X}^\infty$ is dense in $\mathcal{X}^k$ for each $k \in \RR$.

\begin{lemm}
	Let $(X,g)$ be an aAdS spacetime. If $\nu \in (0,1)$ and $k \in \RR$, then $\maxdom^{\infty}$ is dense in $\maxdom^k$.
\end{lemm}
\begin{proof}
The proof is standard, making sure to use regularizers in $\Psibeven^{-\infty}(X)$. Let 
\[
\{A_r \in \Psib^{-\infty}(X): r \in (0,1)\}
\]
 be a compactly supported bounded family in $\Psib^{0}(X)$ such that $A_r \rightarrow 1$ in $\Psib^\delta(X)$ for each $\delta > 0$. We may furthermore arrange that
\[
A_r = FB_rF^{-1}, \quad B_r \in \Psibeven^{-\infty}(X).
\]
According to Lemma \ref{lem:preservesmaxdomain}, $A_r u \in \maxdom^{\infty}$. Tracing through the proof of Lemma \ref{lem:preservesmaxdomain} also shows that $[P,A_r]$ converges to zero strongly on $\maxdom^k$, and so $A_r u \rightarrow u$ in the graph norm of $\maxdom^k$.	
\end{proof}

For simplicity, to define $\gamma_-$ on $\tSob^1_\loc(X)$ (which recall depends on the choice of bdf) let us fix once and for all a reference special bdf $x$. Given this choice of $x$, we also fix an admissible twisting function $F$ normalized by $x^{-\nu_-}F|_{\pa X} = 1$. We then define the second trace $\gamma_+$ on $\maxdom^\infty$ by
\[
\gamma_+ u = x^{1-2\nu}\pa_x (F^{-1}u) |_{\pa X}.
\]
This is well defined in light of Lemma \ref{lem:graphnormexpansion} which shows that the restriction of $u$ to a special coordinate patch can be written in the form
\[
u = x^{\nu_-}F u_- + x^{\nu_+}u_+ + u_2, \quad u_2 \in x^2\tSob^{0,\infty}_\loc([0,\varepsilon)\times \RR^{n-1}).
\]
Appealing to \eqref{eq:sobolevembedding} shows that 
\[
x^{1-2\nu}\pa_x(F^{-1}u_2) \in x^{1-\nu}L^\infty_\loc([0,\varepsilon)\times \RR^{n-1}),
\]
which therefore vanishes when $x=0$. Thus, in these coordinates, $\gamma_+ u = 2\nu u_+$. In the next section we extend $\gamma_+$ to $\mathcal{X}^k$.

\subsection{Green's formula} \label{subsect:greens}

Let $\nu \in (0,1)$, and suppose that $u \in \maxdom^{\infty}$
and $v \in x^{\nu_-}\CcI(X)$. If either $\gamma_+ u =0$ or $\gamma_- v = 0$, then
\begin{equation} \label{eq:pregreensformula}
\int Pu \cdot \bar v \, dg = \mathcal{E}_0(u,v) + \int S_F u \cdot \bar v \, dg.
\end{equation}
Because $\maxdom^{\infty}$ is dense in the graph space $\maxdom$, if $v \in \tdotSob^1_\comp(X)$ is fixed, then \eqref{eq:pregreensformula} is also valid for arbitrary $u \in \tSob^1_\loc(X)$ satisfying $Pu \in x^2\Ltwo_\loc(X)$. More generally, however, there are boundary terms, and the correct Green's formula is
\begin{equation} \label{eq:greensformula}
\int Pu \cdot \bar v \, dg =  \mathcal{E}_0(u,v) + \int S_F u \cdot \bar v \, dg + \int \gamma_+ u \cdot \gamma_- \bar v \, dk_0
\end{equation}
for $u \in \maxdom^\infty$ and $v \in \tSob^1_\comp(X)$.
This formula can be extended to more general $u$ as follows.
\begin{lemm}\label{lem:gammaplus}
Given $k \in \RR$, the map $\gamma_+$ extends to a bounded map $\maxdom^k \rightarrow H^{k-\nu}_\loc(X)$. Moreover, if $u \in \maxdom^k$, then Green's formula \eqref{eq:greensformula} holds for every $v \in \tSob^{1,-k}_\comp(X)$.
\end{lemm} 
\begin{proof}
We focus on the case $k=0$. Notice that the proof of Green's formula would be trivial if we a priori knew that $\gamma_+$ maps continuously into $H^{-\nu}_\loc(X)$. As it stands, however, we must proceed differently. Given $u \in \tSob^1_\loc(X)$ with $Pu \in x^2\Ltwo_\loc(X)$, consider the linear functional $\ell$ defined on $v_- \in \CcI(\pa X)$ given by
\[
\ell(v_-) = \int Pu \cdot \bar v \, dg -  \mathcal{E}_0(u,v) - \int S_F u \cdot \bar v \, dg,
\]
where $v \in x^{\nu_-}\CcI(X)$ is any element such that $v_- = \gamma_- v$. This is well defined in view of \eqref{eq:pregreensformula}. In particular, it is possible to take $v = \rho v_-$, where 
\[
\rho : H^\nu_\comp(\pa X) \rightarrow \tSob^1_\comp(X)
\] 
is a continuous extension that maps $\CcI(\pa X)$ onto $x^{\nu_-}\CcI(X)$ (as in Lemma \ref{lem:traceinverse}). If $\phi \in \CcI(X)$ is such that $\phi = 1$ on $\supp \rho v_-$, then
\[
|\ell(v_-)| \leq C\big (\| x^{-2} \phi Pu \|_{\Ltwo(X)} + \| \phi u \|_{\tSob^1(X)} \big )\| \rho v_- \|_{\tSob^1(X)} \leq C\| v_- \|_{H^\nu(\pa X)}.
\]
By Hahn--Banach, there exists an element of $H^{-\nu}_\loc(X)$, which we denote by $\tilde u_+$, such that $\ell(v_-) = \langle \tilde u_+, v_- \rangle_{\pa X}$, and furthermore
\[
\| \phi \tilde u_+ \|_{H^{-\nu}(\pa X)} \leq C\big (\| x^{-2} \phi Pu \|_{\Ltwo(X)} + \| \phi u \|_{\tSob^1(X)} \big ).
\]
This shows that $u \mapsto \tilde u_+$ is a continuous map from the graph space to $H^{-\nu}_\loc(X)$, and that \eqref{eq:greensformula} holds for arbitrary $u$ in the graph space, provided $\gamma_+ u$ is replaced with $\tilde u_+$. It therefore remains to show that $\gamma_+ u = \tilde u_+$. This is true if $u \in \maxdom^\infty$ by \eqref{eq:greensformula}, and thus also holds for $u \in \maxdom$ by the density of $\maxdom^\infty$.
\end{proof}

\section{The boundary value problem}\label{sec:bvp}
\subsection{The Dirichlet form} \label{subsect:dirichletform}

The Dirichlet problem is given a weak formulation in the usual way. For a fixed smooth twisting function $F$, define
\[
\mathcal{E}_D(u,v) = \mathcal{E}_0(u,v) + \int S_F u \cdot \bar v \, dg.
\]
For an arbitrary $F$, in general one only has $S_F \in xC^\infty(X)$. By Hardy's inequality in one dimension, this is enough to guarantee that multiplication by $S_F$ is bounded $\tdotSob^1_\loc(X) \rightarrow x^2 \Ltwo_\loc(X)$. We then define the map $P_D : \tdotSob^1_\loc(X) \rightarrow \tSob^{-1}_\loc(X)$ by
\[
\langle P_D u, v \rangle = \mathcal{E}_D(u,v).
\]
This agrees with the distributional action of $P$ on $\tdotSob^1_\loc(X)$. As in the discussion following \ref{defi:negativeorderspaces}, we can also extend $P_D$ to a map $\tdotSob^{1,m}_\loc(X) \rightarrow \tSob^{-1,m}_\loc(X)$ for each $m \in \RR$. The only additional ingredient needed is an extension of Lemma \ref{lem:1/xbounded} to spaces with general conormal regularity. This latter statement is deduced from Lemma \ref{lem:1/xbounded} by writing $xA = A'x$, where $A' = xAx^{-1} \in \Psib^{m}(X)$ whenever $A \in \Psib^{m}(X)$.

Next we consider the Robin problem under the assumption that an \emph{admissible} twisting function exists. Formally, we consider $u \in \tSob^1_\loc(X)$ satisfying the boundary conditions
\begin{equation} \label{eq:robinbc}
\gamma_+ u - \beta\gamma_- u = 0,
\end{equation}
where $\beta \in \CI(\pa X)$. Eventually we will restrict to the case where $\beta$ is real-valued. These boundary conditions are well-defined in the strong sense provided $Pu \in x^2 \Ltwo_\loc(X)$. For the weak formulation, set
\[
\mathcal{E}_{\Robin}(u,v) = \mathcal{E}_0(u,v) + \int S_F u \cdot \bar v \, dg + \int \beta \gamma_- u \cdot \gamma_- \bar v \, dk_0.
\]
Note the importance of the requirement that $S_F \in x^2L^\infty_\loc(X)$: the weaker condition $S_F \in x\CI(X)$ does not guarantee that $S_F$ is bounded $\tdotSob^1_\loc(X) \rightarrow x^2 \Ltwo_\loc(X)$, unlike in the Dirichlet case.
We then define $P_{\Robin} : \tSob^1_\loc(X) \rightarrow \tdotSob^{-1}_\loc(X)$ by
\[
\langle P_{\Robin} u, v \rangle = \mathcal{E}_{\Robin}(u,v).
\]
If $f \in \Ltwo_\loc(X) \subset \tdotSob^{-1}_\loc(X)$, then $P_{\Robin} u = f$ implies that $x^{-2} Pu = f$ in distributions and that the Robin boundary conditions \eqref{eq:robinbc} is satisfied in the strong sense. \begin{rema}
	We like to think of the distributional action of $P$ as extending to a map $\maxdom \rightarrow x^2\Ltwo_\loc(X)$. Note, however, that $\maxdom \subset \Ltwo_\loc(X)$, and in this sense $P$ acts on $L^2$-based spaces with different weights. On the other hand, when defining the Robin (or Dirichlet) realization of $P$ via a sesquilinear form, we have that $P_R$ maps $\tSob^1_\loc(X)$ to a dual space relative to the $x^2dg$ inner product. Since $\Ltwo_\loc(X)$ embeds in $\tdotSob^{-1}_\loc(X)$, it makes sense to solve equations of the form $P_Ru = f$ with $f \in \Ltwo_\loc(X)$. This inevitably leads to statements like
	\[
	P_R u = f \Longrightarrow x^{-2}Pu = f \text{ in } \CmI(X)
	\]
	above with weights that at first glance appear contradictory.
\end{rema}

 Just as for the Dirichlet problem, we can also extend $P_R : \tSob^{1,m}_\loc(X) \rightarrow \tdotSob^{-1,m}_\loc(X)$. In this case we must also show that the boundary pairing makes sense; this follows from the regularity $\gamma_- u \in H^{\nu+m}_\loc(X)$, valid with any $m \in \RR$ for which $u \in \tSob^{1,m}_\loc(X)$.

%Whenever $\chi, \tilde \chi \in \CcI(X)$ satisfy $\chi =1$ on $\supp \tilde \chi$, then there is an estimate of the form
%\[
%\| \tilde \chi P_{\Robin} u\|_{\tdotSob^{-1}(X)} \leq C \|  \chi u \|_{\tSob^1(X)}.
%\]
%A similar estimate holds for $P_D$.
%\todo{You might not need this anymore}

\subsection{Microlocal estimates} \label{subsect:dirichletmicrolocal}

We give some microlocal estimates for the Dirichlet form. We always work with b-pseudodifferential operators having compact support in a fixed coordinate patch. Near the boundary, it is convenient to use special coordinates $(x,y^1,\ldots,y^{n-1})$, where $x$ is our fixed special bdf.

\begin{lemm} \label{lem:dirichletform1}
Let $U \subset X$ be a boundary coordinate patch and $m\leq 0$. Let $\mathcal{A} = \{A_r: r \in (0,1)\}$ be a bounded subset of $\Psib^s(X)$ with compact support in $U$, such that 
\[
A_r \in \Psib^{m}(X) \text{ for each } r \in (0,1).
\]
Let $G_1 \in \Psib^{s-1/2}(X)$ be elliptic on $\opWFb(\mathcal{A})$, with compact support in $U$. Then there exists $C_0 > 0$ and $\chi \in \CcI(U)$ such that
\[
\mathcal{E}_0(A_ru, A_r u) \leq \mathcal{E}_0(u, A_r^*A_r u) + C_0(\|G_1 u \|_{\tSob^1(X)}^2 + \| \chi u \|_{\tSob^{1,m}(X)}^2)
\]
for every $r \in (0,1)$ and every $u \in \tSob^{1,m}_{\loc}(X)$, provided $\WFb^{1,s-1/2}(u) \cap \opWFb(G_1) = \emptyset$.
\end{lemm}
\begin{proof}
Note that $A^*_r A_r u \in \tSob^{1,-m}_\comp(X)$ for $r\in (0,1)$, hence the pairing $\mathcal{E}_0(u,A_r^*A_ru)$ is well-defined (the precise order of each individual $A_r$ does not play an important role, and is simply chosen to justify the pairings; the order of the \emph{family} $\mathcal{A}$ is the crucial point).	Commuting twice,
	\begin{align*}
	\mathcal{E}_0(u,A_r^*A_ru) &= \langle A_r \hat g^{jk} Q_j u, Q_k A_ru \rangle + \langle \hat g^{jk} Q_j u, [Q_k,A_r^*]A_ru \rangle 
	\\&= \mathcal{E}_0(A_ru,A_ru) + \langle [A_r, \hat g^{jk} Q_j ]u, Q_k A_r u \rangle  + \langle \hat g^{jk} Q_j u, [Q_k,A_r^*]A_ru \rangle. 
	\end{align*}
	The two commutator terms are of lower order; since both can be treated in the same way we focus on the first of them. Here it is convenient to distinguish the commutators with $Q_j$ in the cases $j=0$ or $j\neq0$. Thus we write
	\begin{equation} \label{eq:dirichletformcommutatorsplitting}
	\langle [A_r, \hat g^{jk} Q_j ]u, Q_k A_r u \rangle  =	\langle [A_r, Q_0]u, Q_0 A_r u \rangle + \langle [A_r, \hat g^{\alpha \beta} Q_\alpha ]u, Q_\beta A_r u \rangle
	\end{equation}
	where $\alpha,\beta$ range only over $1,\ldots,n-1$. The second commutator term in \eqref{eq:dirichletformcommutatorsplitting} can be handled entirely within the b-calculus, so we consider only the first term.
	
	Let $\Lambda_{1/2} \in \Psib^{1/2}(X)$ be everywhere elliptic, and let $\Lambda_{-1/2} \in \Psib^{-1/2}(X)$ be a para\-metrix,	so that $1 = \Lambda_{1/2} \Lambda_{-1/2} +R$ where $R \in \Psib^{-\infty}(X)$; in that case,
	\begin{equation} \label{eq:dirichletformQ0commutator}
	\langle [A_r, Q_0]u, Q_0 A_r u \rangle = \langle \Lambda^*_{1/2} [A_r,Q_0]u, \Lambda_{-1/2} Q_0 A_r u \rangle + \langle \langle [A_r, Q_0]u, RQ_0 A_r u \rangle . 
	\end{equation}
	Now for any $A \in \Psib^s(X)$ with compact support in $U$, repeated applications of Lemma \ref{lem:Q0commutator} show that the first term on the hand side of \eqref{eq:dirichletformQ0commutator} can be written as
	\[
	\langle \Lambda^*_{1/2} [A,Q_0]u, \Lambda_{-1/2} Q_0 A u \rangle = \langle (Q_0 A' + A'') u, (Q_0 B' + B'' )u \rangle, 
	\]
	for some $A', B' \in \Psib^{s-2}(X)$ and $A'', B'' \in \Psib^{s-1}(X)$, where the maps $A \mapsto (A',A'', B',B'')$ are micro\-local. Applying this in particular to $A = A_r \in \mathcal{A}$, it follows that the corresponding families 
	\[
	\mathcal{A}' = \{ A'_r: r\in (0,1) \}, \quad \mathcal{A}'' = \{A''_r: r\in (0,1)\}
	\]
	are bounded in $\Psib^{s-1}(X)$ and $\Psib^{s}(X)$ respectively, and 
	\[
	\WFb'(\mathcal{A}') \cup \WFb'(\mathcal{A}'') \subset \WFb'(\mathcal{A}) \subset \ellb(G_1).
	\]
	Since the same is true of the families formed by $B'_r$ and $B_r''$, we can bound
	\[
	|\langle \Lambda^*_{1/2} [A_r,Q_0]u, \Lambda_{-1/2} Q_0 A_r u \rangle| \leq  C_0 (\| G_1u \|^2_{\tSob^1(X)} + \| \chi u \|^2_{\tSob^{1,m}(X)})
	\]
	with $C_0 > 0$ independent of $r$, for a suitable cutoff $\chi \in \CcI(U)$. The second term in \eqref{eq:dirichletformQ0commutator} involving $R$ can be bounded by $\| \chi u \|^2_{\tSob^{1,m}(X)}$ itself.
\end{proof}

Now we consider the Robin problem, always assuming the existence of an admissible twisting function $F$ with $S_F \in x^2\CI(X)$.

\begin{lemm} \label{lem:dirichletform2}
Let $U \subset X$ be a boundary coordinate patch and $m\leq 0$. Let $\mathcal{A} = \{A_r: r \in (0,1)\}$ be a bounded subset of $\Psib^s(X)$ with compact support in $U$, such that 
\[
A_r \in \Psib^{m}(X) \text{ for each } r \in (0,1).
\]
	Let $G_0\in \Psib^{s}(X)$ and $ G_1 \in \Psib^{s-1/2}(X)$ both be elliptic on $\opWFb(\mathcal{A})$, with compact support in $U$. Then there exists $C_\varepsilon>0$ and $\chi \in \CcI(U)$ such that
	\begin{multline*}
\mathcal{E}_0(A_r u,A_r u) - \varepsilon\mathcal{Q}(A_r u, A_r u) \\
\leq  C_\varepsilon ( \| \chi u\|^2_{\tSob^{1,m}(X)} + \| \chi P_{\Robin} u\|^2_{\tdotSob^{-1,m}(X)} +  \| G_1 u \|^2_{\tSob^1(X)} + \| G_0 P_{\Robin} u \|^2_{\tdotSob^{-1}(X)})
	\end{multline*}
	for every $r \in (0,1)$ and every $u \in \tSob^{1,m}_{\loc}(X)$, provided
	\[
	\WFb^{1,s-1/2}(u) \cap \opWFb(G_1) = \emptyset, \quad \WFb^{-1,s}(P_{\Robin} u ) \cap \WFb'(G_0) = \emptyset. 
	\]
\end{lemm}
\begin{proof}
Let $f = P_{\Robin} u $. By definition $\mathcal{E}_{\Robin}(u,A_r^*A_ru) = \langle A_rf, A_r u \rangle $, and according to Lemma \ref{lem:L2intermsofH1},
		\begin{align*}
	|\langle A_r f, A_r u \rangle| &\leq \| A_r f \|_{\tdotSob^{-1}(X)} \| A_r u \|_{\tSob^1(X)} 
	\\&\leq \varepsilon \| A_r u \|^2_{\tSob^1(X)} + C_\varepsilon ( \| G_0 f\|^2_{\tdotSob^{-1}(X)} + \| \chi u \|^2_{\tSob^{1,m}(X)}) 
	\\ &\leq \varepsilon \mathcal{Q}(A_r u, A_r u) +  C_\varepsilon(\| G_0 f\|^2_{\tdotSob^{-1}(X)} + \| G_1 u \|^2_{\tSob^1(X)} +\| \chi u \|^2_{\tSob^{1,m}(X)}). 
	\end{align*}
	It remains to bound the difference between $\mathcal{E}_{\Robin}(u,A_r^*A_r u)$ and $\mathcal{E}_0(A_r u,A_r u)$. First consider 
	\[
	\mathcal{E}_{\Robin}(u,A_r^*A_r u) - \mathcal{E}_0(u,A_r^*A_r u) = \langle x^{-2}S_F u ,A_r^* A_r u\rangle + \langle \beta \gamma_- u, \gamma_- (A_r^* A_r u) \rangle_{\pa X}.
	\]
By Lemma \ref{lem:L2intermsofH1}, we have
	\[
	|\langle x^{-2} S_F u, A_r^* A_r u \rangle | \leq C_0 ( \| \chi u\|^2_{\tSob^{1,m}(X)} +  \| G_1 u \|^2_{\tSob^1(X)}).
	\]
	This estimate would be true even with $G_1 \in \Psib^{s-1}(X)$.
	Next, we must handle the boundary terms. Recall that for any $B \in \Psib^m(X)$,
	\[
	\gamma_-(Bu) = (x^{-\nu_-}Bu)|_{\pX} = \widehat{N}(B)(-i\nu_-)(\gamma_- u).
	\]
	As in Section \ref{subsect:b-microlocalization}, we have
	\[
	\widehat{N}(A_r^*A_r)(-i\nu_-) = \widehat{N}(\tilde A_r)(-i\nu)^*\widehat{N}(A_r)(-i\nu),
	\] 
	where $\tilde A_r  = x^{2\nu_-}A_rx^{-2\nu_-}$, and the adjoint on the right is with respect to the induced density $dk_0$. Extend $\beta$ arbitrarily to a $\CI$ function on $X$, so multiplication by $\beta$ on the boundary can be written as $\widehat{N}(\beta)(-i\nu_-)$. Then
	\begin{align*}
	\langle \beta\gamma_-u_, \gamma_-(A_r^*A_ru) \rangle_{\pa X} &= \langle \widehat{N}(\tilde A_r\beta)(-i\nu_-)\gamma_- u, \widehat{N}(A_r)(-i\nu_-)\gamma_- u\rangle_{\pa X}\\
	&= \langle \gamma_-(\tilde A_r \beta u), \gamma_- (A_ru)\rangle_{\pa X}
	\\ &= \langle \gamma_-(\beta \tilde  A_r u + [\tilde A_r, \beta]u), \gamma_- (A_ru)\rangle_{\pa X}.
	\end{align*}
	Note that $\tilde A_r$ has the same principal symbol as $A_r$, so by \eqref{eq:traceinterpolation2},
	\[
	|\langle \beta\gamma_-u_, \gamma_-(A_r^*A_ru) \rangle|\leq \varepsilon\mathcal{Q}(A_r u, A_r u) + C_\varepsilon (\| G_1 u \|^2_{\tSob^1(X)} + \| \chi u \|^2_{\tSob^{1,m}(X)})
	\]
	 for every  $\varepsilon>0$,
Again, this would even be true with $G_1 \in \Psib^{s-1}(X)$. Combining these facts with Lemma \ref{lem:dirichletform1} finishes the proof.
\end{proof}

Lemma \ref{lem:dirichletform2} also holds for the Dirichlet problem, this time taking $u \in \tdotSob^{1,m}_\loc(X)$ with $m\leq 0$. Here it suffices to work with an arbitrary twisting function. There are two differences: firstly, there is no boundary integral to estimate, and secondly, because $S_F \in x \CI(X)$,
\begin{align*}
| \langle  A_r x^{-2}S_F , A_r u\rangle| &= |\langle ( x A_r x^{-1}) x^{-1} S_F u, x^{-1}  A_r u \rangle |
\\ &\leq C (\| G_1 u \|_{\tSob^1(X)} + \| \chi u \|_{\tSob^{1,m}(X)})\| A_r u \|_{\tSob^1(X)}
\end{align*}
by Hardy's inequality in one dimension. Therefore
\[
| \langle  A_r x^{-2}S_F , A_r u\rangle| \leq \varepsilon\mathcal{Q}(A_r u, A_r u)  + C_\varepsilon (\| G_1 u \|^2_{\tSob^1(X)} + \| \chi u \|^2_{\tSob^1(X)}) 
\]
by Cauchy--Schwarz and Lemma \ref{lem:L2intermsofH1}.

Next, we consider properties of $\Im \mathcal{E}_0(u,A^*Au)$. This will be used in the positive commutator arguments. If $u \in \tSob^1_\loc(X)$ and $A \in \Psib^0(X)$ has compact support, then
\begin{align} \label{eq:twistedcommutator}
\begin{split} 
\mathcal{E}_0(u,Au) - \mathcal{E}_0(Au,u) &= \langle \hat g^{ij}Q_j u, Q_i Au \rangle - \langle \hat g^{ij}Q_j Au, Q_i u \rangle \\
&= \langle \hat g^{ij} Q_j u, [Q_i, A]u \rangle - \langle [\hat g^{ij} Q_j, A]u, Q_i u \rangle 
\\ &+ \langle (A^*-A) \hat g^{ij}Q_j u, Q_i u \rangle.
\end{split}
\end{align}
If $A$ is replaced with $A^*A$, then the third term vanishes. Therefore we have
\begin{align*}
2i \Im \mathcal{E}_0(u, A^*Au) &= \langle \hat g^{ij} Q_j u, [Q_i, A^*A]u \rangle - \langle  [\hat g^{ij}Q_j, A^*A]u, Q_i u \rangle 
\\ &= \langle Q_0 u, [Q_0, A^*A]u \rangle - \langle [Q_0, A^*A], Q_0 u \rangle + \langle [Q_\alpha \hat g^{\alpha\beta}Q_\beta, A^*A]u,u \rangle.
\end{align*}
With $a = \bsymbol{0}(A)$,
\begin{align*}
\langle Q_0 u, [Q_0, A^*A]u \rangle - \langle [Q_0, A^*A], Q_0 u \rangle &= \langle Q_0 u, Q_0 A_1 u \rangle - \langle Q_0 A_1u, Q_0 u \rangle \\ &+ \langle Q_0 u, A_0 u \rangle - \langle A_0 u, Q_0 u \rangle,
\end{align*}
where $\bsymbol{-1}(A_1) = (1/i)\pa_\sigma (a^2)$ and $\bsymbol{0}(A_0) = (1/i)\pa_x (a^2)$.

\section{Propagation of singularities}
\label{sec:propagation}
\subsection{The characteristic variety and bicharacteristics}
The principal symbol of $x^{-2}P$, as a function on $T^* X^\circ \setminus 0$, is given by $\hat p = -\hat g^{ij}\zeta_i \zeta_j$, where we have written covectors in coordinates $(z^0,\ldots,z^{n-1})$ as $\zeta_i\, dz^i$. Note that the mass parameter $\lambda$ plays no role in this expression. Furthermore, $\hat p$ extends smoothly to a function on $T^* X \setminus 0$.
Let 
\[
\chare = \{\hat p = 0\} \subset T^*X \setminus 0
\]
denote the characteristic set of $\hat p$. The compressed characteristic set $\cchare$ is the image of $\chare$ in $\bdotT^*X \setminus 0$ under the quotient map $\pi : T^*X \rightarrow \bdotT^*X$. We equip $\cchare$ with the subspace topology inherited from $\bT^*X$. This is the same as the quotient topology; cf.~\cite[Lemma 5.1]{vasy2008propagation}. 

There is a natural decomposition of $\bdotT^*X\setminus 0$ into its elliptic, hyperbolic, and glancing components. Thus,
\[
\begin{aligned}
\ellip &= \{q \in \bdotT^*X \setminus 0: \pi^{-1}(q) \cap \chare = \emptyset\}, \\
\gl &= \{q \in \bdotT^*X \setminus 0: |\pi^{-1}(q) \cap \chare| = 1 \}, \\
\hyp &= \{q \in \bdotT^*X \setminus 0: |\pi^{-1}(q) \cap \chare| = 2 \}.
\end{aligned}
\]
Let $U$ be a boundary coordinate patch with coordinates $(x,y^1,\ldots,y^{n-1})$, where $x$ is a special bdf and $dx$ is orthogonal to each $dy^i$. If $(x,y,\sigma,\eta)$ are the corresponding canonical coordinates on $\bT^*_U X$, then a point 
\[
q_0 = (x_0,y_0,\sigma_0,\eta_0) \in \bdotT^*_U X
\]
is in $\hyp$ precisely if $q_0 \in T^*\pa X$ (namely $x_0 = \sigma_0 = 0$) and
\[
\hat g^{\alpha \beta}(0,y_0)(\eta_0)_{\alpha}(\eta_0)_{\beta} > 0.
\]
Similarly, a point $q_0 \in T^*\pa X$ is in $\gl$ when $\hat g^{\alpha \beta}(0,y_0)(\eta_0)_{\alpha}(\eta_0)_{\beta} =0$. There are several equivalent definitions of $\GBB$s in this setting, but we choose the following:

\begin{defi}\label{def:gbb}
	If $I \subset \RR$ is an interval, we say that a \emph{continuous} map $\gamma: I \rightarrow \cchare$ is a $\GBB$ if the following two conditions are satisfied for each $s_0 \in I$:
	\begin{enumerate} \itemsep6pt 
		\item If $q_0 = \gamma(s_0) \in \gl$, then for every $f\in \CI(\bT^*X)$, 
		\[
		\frac{d}{ds} (f\circ \gamma)(s_0) = \{ \hat p, \pi^*f \}(\eta_0),
		\]
		where $\eta_0 \in \chare$ is the unique point for which $\pi(\eta_0) = q_0$.
		\item If $q_0 = \gamma(s_0) \in \hyp$, then there exists $\varepsilon > 0$ such that $0 < |s-s_0| < \varepsilon$ implies that $x(\gamma(s)) \neq 0$.
	\end{enumerate}
\end{defi}

Continuity of $\gamma$ implies that tangential momentum is conserved upon interaction with the boundary. The second condition, at hyperbolic points, says that $\GBB$s reflect instantaneously.

Since $\hat p$ is a smooth function on $T^*X$, various properties of $\GBB$s that are true in the setting of smooth boundary value problems are also valid here. In particular, the entire discussion in \cite[Section 5]{vasy2008propagation} applies verbatim.

\subsection{Elliptic estimates} \label{subsect:ellipticestimate}
It is convenient to introduce a normal coordinate system as follows: given a point $p \in \pa X$, fix a spacelike surface in $\pa X$ (with respect to $k_0$) passing through $p$. We can then choose coordinates $(y^1,\ldots,y^{n-1})$ such that $k_0^{-1}$ is of the form
\[
k_0^{-1} = \pa_{y^{n-1}}^2 - \sum h^{ab} \pa_{y^a} \pa_{y^b},
\]
where $a,b$ range over $1,\ldots,n-2$ and $h^{ab}$ is positive definite. Transporting these coordinates to a collar neighborhood of $X$ using the product structure, we see that
\[
g^{-1} = x^2(-\pa_x^2 - h^{ab}\pa_{y^a} \pa_{y^b} + \pa^2_{y^{n-1}} + \mathcal{O}(x)) .
\]
This is useful for the following reason: if $v$ has support in $\{|x| < \delta\}$, then
\begin{equation} \label{eq:Qn-1control}
\begin{split}
	\mathcal{E}_0(v,v) &= \| Q_0 v \|^2_{\Ltwo(X)} + \langle \hat g^{\alpha \beta}Q_\alpha v, Q_\beta v \rangle 
\\ &\geq  \| Q_0  v \|^2_{\Ltwo(X)} - (1+C\delta) \|Q_{n-1}v\|^2_{\Ltwo(X)} + (1-C\delta)\langle h^{\alpha \beta} Q_\alpha v, Q_\beta v \rangle,
\end{split}
\end{equation}
where $C >0$ is independent of $\delta>0$. Furthermore, the principal symbol $\hat p$ restricted to $T^*_Y X$ is just
\[
\hat p|_{T^*_Y X} = \xi^2 + h^{ab}\eta_a \eta_b - \eta_{n-1}^2.
\]
Let $q_0 \in \bT^*_YX \setminus 0$ be such that $q_0 \notin \cchare$. Thus there are two possibilities: either $q_0$ is not in the compressed b-cotangent bundle, or if it is, then $h^{ab}\eta_a \eta_b > \eta_{n-1}^2$ at $q_0$. This observation is rephrased as follows:
\begin{lemm} \label{lem:ellipticpossibilities}
If $q_0 \in \bT^*_YX \setminus 0$, then there is a conic neighborhood $V$ of $q_0$ in which one of the following is true:
\begin{enumerate} \itemsep6pt
		\item There is $\varepsilon > 0$ such that $\sigma^2 < \varepsilon^2 (\eta_{n-1}^2 + h^{ab}\eta_a \eta_b)$ and $h^{ab}\eta_a \eta_b > (1+\varepsilon)\eta_{n-1}^2$.
	\item There is $C>0$ such that $|\eta_{n-1}| < C|\sigma|$.
\end{enumerate}
\end{lemm}

Using this observation, it is simple to prove the following elliptic regularity for either the Dirichlet or Robin (Neumann) problem. Since the proofs are identical upon substituting the appropriate function spaces, we focus on the Robin case.

\begin{theo} \label{theo:ellipticregion}
Let $u \in \tSob^{1,m}_\loc(X)$ for some $m \leq 0$, and $q_0 \in \bT^*_YX \setminus 0$. If $s \in \RR \cup \{+\infty\}$ and 
\[
q_0 \in \WFb^{1,s}(u) \setminus \WFb^{-1,s}(P_{\Robin} u),
\] 
then $q_0 \in \cchare$.
\end{theo}
\begin{proof}
Assume that $q_0 \notin \cchare$. We show that if $q_0 \notin \WFb^{-1,s}(P_{\Robin}u)$ and $q_0 \notin \WFb^{1,s-1/2}(u)$, then
\[
q_0 \notin \WFb^{1,s}(u).
\]
The proof is then finished by induction; the inductive hypothesis is satisfied for any $s \leq 1/2 +m$ by the assumption that $u \in \tSob^{1,m}_\loc(X)$. Let $\mathcal{A} = \{A_r : r \in (0,1)\}$ be a bounded subset of $\Psib^s(X)$ such that $A_r \in \Psib^{m-1}(X)$ for each $r \in (0,1)$. We assume that $\opWFb(\mathcal{A})$ is contained in a sufficiently small neighborhood of $V$ so that
\[
\WFb^{1,s-1/2}(u) \cap \opWFb(\mathcal A) = \emptyset.
\]
Fixing a boundary coordinate patch $U$ containing $q_0$, assume in addition that $\mathcal{A}$ has compact support  in this patch. We consider two cases corresponding to those of Lemma \ref{lem:ellipticpossibilities}, letting $V$ be a conic neighborhood of $q_0$ as in the lemma.

\begin{inparaenum}
	\item If $C>0$ and $\delta > 0$ is sufficiently small, then 
\begin{align*}
(1-C\delta)h^{ab}\eta_a \eta_b - (1+C\delta)\eta_{n-1}^2 &= (1-C\delta)(h^{ab}\eta_a \eta_b -\eta_{n-1}^2) - 2C\delta \eta_{n-1}^2 \\
& > \left(\varepsilon(1 - C\delta) - 2C\delta \right)\eta_{n-1}^2
\\ & > (\varepsilon/2)\eta_{n-1}^2
\end{align*}
on $V$. This implies that $ (1-C\delta)h^{ab}\eta_a \eta_b - (1+C\delta)\eta_{n-1}^2$ is elliptic near $V$. If the family $\mathcal{A}$ has support in $\{|x| < \delta\}$, we can choose $B \in \Psib^1(X)$  with $\opWFb(\mathcal{A}) \subset \ellb(B)$ such that 
\[
\opWFb( (1-C\delta) h^{ab}Q_a^* Q_b - (1+C\delta)Q^*_{n-1}Q_{n-1} - B^*B + T) \cap U = \emptyset,
\]
where $T \in \Psib^{1}(X)$.
Therefore, for each $r \in (0,1)$,
\begin{align*}
\mathcal{E}_0(A_ru,A_ru) &\geq  \| Q_0 A_r u \|^2_{\Ltwo(X)}+ \| BA_ru\|^2_{\Ltwo(X)} - \langle TA_ru,A_ru \rangle
\\ &\geq C^{-1} \mathcal{Q}(A_ru, A_ru) - \langle TA_ru,A_r u\rangle.
\end{align*}
Note that this pairing makes sense since $A_r \in \Psib^{m}(X)$ and $T \in \Psib^{1}(X)$, using that $u \in \tSob^{1,m}_\loc(X)$.

\item The second case is similar, noting that for $v$ with support in $\{ |x| < \delta\}$,
\[
\| Q_0 v \|^2_{\Ltwo(X)} \geq \delta^{-2} \| xQ_0 v\|^2_{\Ltwo(X)}.
\]
Consider the operator $(2\delta^2)^{-1}(xQ_0)^*(xQ_0) - (1+C\delta)Q_{n-1}^*Q_{n-1}$, which is in $\Psib^2(X)$ with principal symbol
\[
\frac{\sigma^2}{2\delta^2}- (1+C\delta)\eta_{n-1}^2 > c \eta_{n-1}^2
\]
on $V$. In particular, $(2\delta^2)^{-1}(xQ_0)^*(xQ_0) - (1+C\delta)Q_{n-1}^*Q_{n-1}$ is elliptic on $V$, so again we can find  $B \in \Psib^1(X)$ with $V \subset \ellb(B)$ such that 
\[
\WFb((2\delta^2)^{-1}(xQ_0)^*(xQ_0) - (1+C\delta)Q_{n-1}^*Q_{n-1} - B^*B + F) \cap U = \emptyset,
\]
where $T \in \Psib^1(X)$. Again we find that for each $r \in (0,1)$,
\begin{align*}
\mathcal{E}_0(A_ru,A_ru) &\geq  \tfrac{1}{2}\| Q_0 A u \|^2_{\Ltwo(X)} + (1-C\delta)\langle h^{\alpha \beta} Q_\alpha v, Q_\beta v \rangle + \| BA_ru\|^2_{\Ltwo(X)} - \langle TA_ru,A_ru \rangle \\
&\geq C^{-1} \mathcal{Q}(A_ru, A_ru) - \langle TA_ru,A_ru\rangle.
\end{align*}
\end{inparaenum}
Let $\Lambda_{1/2} \in \Psib^{1/2}(X)$ be elliptic with $\Lambda_{-1/2} \in \Psib^{-1/2}(X)$ a parametrix. Then in both cases we can write
\[
\langle TA_r u , A_r u \rangle = \langle \Lambda_{-1/2}^* TA_r u, \Lambda_{ 1/2} A_r u \rangle + \langle TA_r u, RA_r u \rangle
\]
with $R \in \Psib^{-\infty}(X)$, which shows that $|\langle TA_r u , A_r u \rangle|$ is uniformly bounded in $r \in (0,1)$ by Lemma \ref{lem:L2intermsofH1} and the a priori hypothesis on $u$.

Finally, we choose the family $\mathcal{A}$. Let $A \in \Psib^s(X)$ be elliptic at $q_0$, with compact support in $U \cap \{x < \delta\}$. Let $\{ J_r: r \in (0,1)\}$ be a bounded family in $\Psib^{0}(X)$ such that $J_r \in \Psib^{m-s-1}(X)$ for each $r \in (0,1)$, converging to the identity in $\Psib^{1}(X)$ as $r\rightarrow 0$. We then let 
\[
A_r = J_r A,
\]
so that in particular $A_r u \rightarrow Au$ in $\CmI_\comp(X)$. Taking $\varepsilon >0$ sufficiently small in Lemma \ref{lem:dirichletform2}, combined with Lemma \ref{lem:L2intermsofH1}, shows that $A_r u$ is uniformly bounded in $\tSob^1(K)$ for a suitable compact set $K \subset X$. Extracting a weakly convergent subsequence in $\tSob^1(K)$ shows that $Au \in \tSob^{1}(K)$.
\end{proof}

Combined with standard elliptic regularity away from $\pa X$, we have shown the following:

\begin{corr} \label{corr:ellipticregularity}
	If $u \in \tSob^{1,m}_\loc(X)$ for some $m \leq 0$ and $s \in \RR \cup \{+\infty\}$, then 
	\[
	\WFb^{1,s}(u) \setminus \WFb^{-1,s}(P_{\Robin} u) \subset \cchare
	\]
	
\end{corr}
The same result holds for the Dirichlet problem, now taking $u \in \tdotSob^{1,m}_\loc(X)$.

\subsection{The hyperbolic region} \label{subsect:hyperbolicregion}

Now we focus on the hyperbolic region. Fix a boundary coordinate patch $U$ with coordinates $(x,y^1,\ldots,y^{n-1})$, and let $q_0 \in \hyp \cap \bT^*_U X$. In coordinates $(x,y,\sigma,\eta)$, this means that
\[
q_0 = (0,y_0,0,\zeta_0), \quad \hat g^{\alpha\beta}(0,y_0)(\eta_0)_\alpha (\eta_0)_\beta > 0.
\]
Let $u \in \tSob^{1,m}_\loc(X)$ for some $m \leq 0$, and suppose that $q_0 \notin \WFb^{-1,s+1}(P_{\Robin} u)$. In order to prove propagation of singularities through hyperbolic points, it suffices to show that $q_0 \in \WFb^{1,s}(u)$ implies $q_0$ is an accumulation point of $\WFb^{1,s}(u) \cap \{\sigma < 0\}$.

\begin{prop} \label{prop:hyperbolicregion}
Let $u \in \tSob^{1,m}_\loc(X)$ for some $m \leq 0$, and suppose that 
\[
q_0 \notin \WFb^{-1,s+1}(P_{\Robin} u).
\]
 If there exists a conic neighborhood $W \subset T^*X\setminus 0$ of $q_0$ such that
\[
W \cap \{\sigma < 0\} \cap \WFb^{1,s}(u) = \emptyset,
\]
then $q_0 \notin \WFb^{1,s}(u)$.
\end{prop}

It  suffices to prove the proposition with any conic subset $W_0 \subset W$ containing $q_0$. If $W_0$ is sufficiently small, we can assume $W_0 \cap \WFb^{-1,s+1}(P_{\Robin} u) = \emptyset$. In particular, 
\[
W_0 \cap \WFb^{1,s}(u) \subset \cchare 
\]
by elliptic regularity, so $x \neq 0$ on $W_0 \cap \{\sigma < 0\} \cap \WFb^{1,s}(u)$. Thus, if $q_0 \in \WFb^{1,s}(u)$, then $q_0$ is the limit of points in the wavefront set intersected with the interior.

The rest of this section is dedicated to the proof of Proposition \ref{prop:hyperbolicregion}. The proof proceeds iteratively, increasing the regularity by $1/2$ at each step. Thus we assume that the hypotheses of the proposition hold, and that $q_0 \notin \WFb^{1,s-1/2}(u)$. We then show that $q_0 \notin \WFb^{1,s}(u)$. Note that the inductive hypothesis is always satisfied for $s\leq 1/2+m$. 

As in Section \ref{subsect:ellipticestimate}, we can further choose coordinates $(y^1,\ldots,y^{n-1})$ such that
\[
\hat g^{\alpha\beta}(0,y)\eta_\alpha \eta_\beta = \eta_{n-1}^2 - h^{ab}(y)\eta_a \eta_b.
\]
In particular, $\eta_{n-1}$ is non-vanishing at $q_0$. If $\kappa : \bT^*X \setminus 0 \rightarrow \bS^*X$ is the canonical projection\footnote{Recall that $\bS^*X$ is the \emph{$\b$-cosphere bundle}, obtained by quotienting $\bT^*X \setminus 0$ by the fiberwise $\rr_+$ action of dilations.}, then in a sufficiently small neighborhood of $\kappa(q_0)$ in $\bS^*X$ we may use projective coordinates 
\[
x, \ y, \ \hat \sigma = \rho^{-1}\sigma, \ \hat \eta_a = \rho^{-1} \eta_a,
\]
where we define $\rho = |\eta_{n-1}|$.
Closely following \cite[Section 6]{vasy2008propagation}, define the functions
\[
\omega = |x|^2 + |y-y_0|^2 + \sum |\hat \eta_a - (\hat \eta_0)_a|^2, \quad \phi = \hat \sigma+\frac{1}{\beta^2\delta}\omega.
\]
The parameters $\delta,\beta$ will be chosen
later; $\beta >0$ will be chosen as an overall large parameter, whereas $\delta > 0$ will be chosen small to localize near $q_0$. Choose cutoff functions $\chi_0, \chi_1$ with the following properties:
\begin{itemize} \itemsep6pt
	\item $\chi_0$ is supported in $[0, \infty)$, with $\chi_0(s)
	=\exp(-1/s)$ for $s>0$.
	\item $\chi_1$ is supported in $[0,\infty)$, with $\chi_1(s) =1$ for
	$s\geq 1$, and $\chi_1'\geq 0$.
\end{itemize}
In a small neighborhood of $\kappa(q_0)$ define the functions
\begin{equation} \label{eq:hypcommutant}
a=\chi_0(2-\phi/\delta) \chi_1(2+\hat \sigma/\delta).
\end{equation}
For each fixed $\beta > 0$, the support of $a$ is controlled by the
parameter $\delta > 0$ as in the following lemma; in particular, we can extend $a$ to a globally defined symbol in $S^0(\bT^*X)$. 
\begin{lemm} \label{lem:asupport}
	Given a neighborhood $V\subset \bS^* X$ of $\kappa(q_0)$ and $\beta>0$, there exists $\delta_0>0$ such that $\supp a \subset V$ for each $\delta \in (0,\delta_0)$.
\end{lemm}
\begin{proof}
	Necessary conditions to lie in the support of $a$ are $\phi \leq 2\delta$ and  $-2\delta  \leq \hat \sigma$. From the definition of $\phi$, 
	\[
	|\hat \sigma| \leq 2\delta, \quad 0 \leq \omega \leq \beta^2\delta(2 \delta-\hat \sigma) \leq 4\beta^2 \delta^2
	\] 
	on $\supp a$, i.e.,\
	\begin{equation}
	\label{suppa}
	\supp a \subset \{|\hat\sigma| \leq 2\delta, \, \omega^{1/2} \leq 2\beta \delta \}.
	\end{equation}
	Finally, observe that any neighborhood of $V$ of $\kappa(q_0)$ contains a set of the form $\{|\hat\sigma| \leq 2\delta, \, \omega^{1/2} \leq 2\beta \delta \}$ provided $\delta$ is sufficiently small.
\end{proof}

%If $\beta >0$ is fixed and we make the dependence on $\delta$ explicit by writing $a = a_\delta$, then
%\[
%\supp a_{\delta} \subset \{ a_{\delta'}  \neq 0 \}
%\]
%whenever $\delta' > \delta > 0$.
%Given compactly supported quantizations $A_\delta \in \Psib^0(X)$, it follows that $A_{\delta'}$ is elliptic on $\opWFb(A_\delta)$.

%\begin{rema}
%In the higher codimension corners setting of \cite{vasy2008propagation}, an additional cutoff to the characteristic set is needed to produce a valid symbol; this is not an issue for the codimension one boundary case. In \cite{vasy2008propagation} it is only the intersection of the support of $a$ with the image of $\cchare$ in $\bS^*X$ which shrinks as $\delta \rightarrow 0$, which by elliptic regularity is all that is needed.
%\end{rema}

Fix a conic neighborhood $V_0$ of $q_0$ having compact closure in which $\hat g^{\alpha\beta}\eta_{\alpha}\eta_{\beta} > 0$. If $\psi_0 \in \CcI(\bS^*X)$ is identically one $V_0$ with support in a sufficiently small neighborhood of $V_0$, then 
\[
\eta_{n-1} \neq 0 \text{ on } \supp \psi_0.
\] 
 To facilitate the regularization argument, fix a bounded family of operators $\{J_r: r\in (0,1)\}$ in $\Psib^{s+1/2}(X)$ such that  $J_r \in \Psib^m(X) \text{ for } r \in (0,1)$.
We choose $J_r$ so that its principal symbol $j_r = \bsymbol{0}(J_r)$ is given by
\[
j_r =  \psi_0 \cdot \rho^{s+1/2}\langle r\rho\rangle ^{-s - 1/2 + m}.
\]
In particular, $J_r$ is elliptic near $V_0$. We then define $A_r = AJ_r $, which is bounded in $\Psib^{s+1/2}(X)$.

We also need some additional auxiliary operators. Let $B \in \Psib^{-1/2}(X)$ have principal symbol
\[
b = \rho^{-1/2}\delta^{-1/2} (\chi'_0 \chi_0)^{1/2} \chi_1.
\]
Here the arguments of $\chi_0,\chi_1$ are as in \eqref{eq:hypcommutant}. Let $B_r \in \Psib^{s+1}(X)$ have principal symbol $\bsymbol{s+1}(B_r) = \rho j_r b$, and let $\tilde B_r  \in \Psib^{s}(X)$ have principal symbol $\bsymbol{s}(\tilde B_r)= j_r b$. Finally, let $C \in \Psib^0(X)$ have principal symbol 
\begin{equation} \label{eq:Csymbol}
\bsymbol{0}(C) = \rho^{-2}(\hat g^{\alpha\beta}\eta_{\alpha}\eta_{\beta})^{1/2}\psi_0.
\end{equation}
This makes sense, since  $\hat g^{\alpha \beta}\eta_\alpha \eta_{\beta} >0$ near $\supp \psi_0$.
In order to handle terms bounded by the inductive hypothesis, fix $G_1 \in \Psib^{s-1/2}(X)$ and $G_0 \in \Psib^{s+1}(X)$ such that 
\[
\opWFb(G_1) \cap \WFb^{1,s-1/2}(u) = \emptyset, \quad \opWFb(G_0) \cap \WFb^{-1,s+1}(P_{\Robin} u) = \emptyset
\]
and $q_0 \in \ellb(G_1) \cap \ellb(G_0)$.

Throughout the rest of this section there appear various operators with wavefront sets contained in $\{\hat \sigma < 0\}$. Given $\beta, \delta$, we use the notation $E_r$ to denote a generic operator such that $\mathcal{E} = \{E_r: r\in (0,1)\}$ forms a bounded family (in a space of operators of fixed order, to be specified) and
\begin{equation} \label{eq:genericapriori}
\opWFb(\mathcal{E}) \subset \{-2\delta \leq \hat \sigma \leq -\delta, \, \omega^{1/2} \leq 2\beta \delta\}.
\end{equation}
Similarly, we denote by $T_r$ a generic operator such that
\begin{equation} \label{eq:genericlowerorder}
\opWFb(\mathcal{T}) \subset \{|\hat \sigma| \leq 2\delta, \, \omega^{1/2} \leq 2\beta\delta\},
\end{equation}
where we set $\mathcal{T}_r = \{T_r : r\in (0,1)\}$. Terms of the form $T_r$ will arise as lower order errors bounded by the inductive hypothesis. For notational flexibility we allow the operators $E_r, \,T_r$ to change from line to line, but their orders will always be made explicit.

\begin{lemm} \label{lem:hyperbolicHpphi}
	Let $g \in S^k(\bT^*X)$. Given $\beta >0$ and $c > 0$, there exists $\delta_0, C_0 > 0$ such that for each $\delta \in (0,\delta_0)$,
\[
|\pa_x \phi| + \rho^{1-k}|\{\phi,g\}| \leq C_0(\delta+ \beta \delta + \beta^{-1})
\]
whenever $|\hat \sigma| < c\delta$ and $\omega^{1/2} \leq c\beta\delta$.
\end{lemm}
\begin{proof}
	Note that we must take $\delta_0 > 0$ sufficiently small even to make sense of $\phi$, since it is only locally defined near $q_0$. First consider the Poisson bracket $\{\phi,g\}$, consisting of three terms
	\[
	\{\phi,g\} = \{\rho^{-1},g\}\sigma + \{\sigma,g\}\rho^{-1} + (\beta^2\delta)^{-1}\{\omega,g\}.
	\]
	The term $\rho^{1-k}|\{\rho^{-1},g\}\sigma|$ is bounded by a constant times $\delta$. Furthermore, because
	\[
	\rho^{1-k}|\{\omega,g\}| + \rho^{-k}|x\pa_x g| \leq C \omega^{1/2}
	\] 
	locally uniformly, the desired bound holds. Similarly, $|\pa_x \omega| \leq C \omega^{1/2}$, which completes the proof.
\end{proof}

Lemma \ref{lem:hyperbolicHpphi} can be applied to compute the commutator $[A_r^*A_r,G]$ for $G \in \Psib^k(X)$. Recall the convention regarding generic operators $E_r, T_r$ discussed before Lemma \ref{lem:hyperbolicHpphi}.

%let $B_r \in \Psib^{s+1}(X)$ have principal symbol
%\begin{equation}\label{bdef}
%b_r = \bsymbol{s+1}(B_r) = 2j_r |\zeta_{n-1}|^{1/2}\delta^{-1/2} (\chi'_0 \chi_0)^{1/2} \chi_1.
%\end{equation}

% In order to prove Proposition \ref{prop:hyperbolicregion}, it suffices to show that $\|\tilde B_r u \|_{\Ltwo(X)}$ is uniformly bounded as $r\rightarrow 0$ for suitable choices of $\beta$ and $\delta$. Indeed, according to \eqref{eq:Qn-1control}, 
% \[
% \mathcal{Q}(BJ_r u, B J_r u) \leq C (\|\tilde B_r u \|^2_{\Ltwo(X)} + \mathcal{E}_0(BJ_r u, BJ_r u))
% \]
%provided $\delta > 0$ is sufficiently small. On the other hand,  since $s\geq 1/2$, Lemma \ref{lem:dirichletform2} applies to the family $BJ_r$, and the inductive hypothesis implies that the right hand side is uniformly bounded in $r \in (0,1)$. As in the proof of Theorem \ref{theo:ellipticregion}, this implies that $BJ_0u \in \tSob^1_\loc(X)$, where $BJ_0 \in \Psib^s(X)$ is elliptic at $q_0$.

\begin{lemm} \label{lem:commutatorexpansion2}
Let $G \in \Psib^{k}(X)$. Given $\beta > 0$, there exists $\delta_0 >0$ such that for each $\delta \in (0,\delta_0)$,
	\[
	i[A_r^* A_r, G] = B_r^*D_r B_r + E_r + T_r,
	\]
	where $E_r \in \Psib^{2s+k}(X), \ T_r \in \Psib^{2s+k-1}(X)$, and $D_r \in \Psib^{k-2}(X)$. Moreover, there exists $C_0 > 0$ such that for every $r \in (0,1)$,
	\[
	\rho^{2-k}|\bsymbol{k-2}(D_r)| \leq C_0(\beta\delta + \delta + \beta^{-1}).
	\] 
\end{lemm}
\begin{proof} 
Let $g = \bsymbol{k}(G)$, so the principal symbol of $i[A_r^*A_r,G] \in \Psib^{2s+k}(X)$ is $2a_r\{a_r, g\}$. The claim is that we can write
\[
\{a_r^2,g\} = d_r b_r^2 + e_r
\]
with $d_r, \,e_r$ representing the principal symbols of the operators $D_r, \, E_r$. The lower order term $T_r$ then arises since we have arranged equality at the level of principal symbols. Recall that 	
\[
a_r = \chi_0(2-\phi/\delta) \chi_1(2+\sigma/\delta)j_r,
\]
so the Poisson bracket with $g$ is a sum of terms with derivatives landing on either $\chi_0,\,\chi_1$, or $j_r$. 

Fix a cutoff $\psi \in \CcI(\bS^*X)$ such that $\psi = 1$ on $\{ |\hat \sigma| \leq 2\delta, \, \omega^{1/2} \leq 2\beta\delta\}$ and $\supp \psi \subset \{ |\hat \sigma| < 3\delta, \, \omega^{1/2} < 3\beta\delta\}$.

\begin{inparaenum}
	
	\item When $\chi_0$ is differentiated we obtain a term $-2\rho^{-1}\{\phi,g\}b_r^2$. Now set
	\[
	d_{1,r} = -2 \psi \rho^{-1} \{\phi,g\},
	\]
	to which Lemma \ref{lem:hyperbolicHpphi} applies. The term $d_{1,r}$ will partly comprise $d_r$; note that $d_{1,r}$ is in fact independent of $r$.
	
	\item The second constituent of $d_r$, in addition to $d_{1,r}$ above, arises when $j_r$ is differentiated. Thus we have a term of the form $a^2 \{j_r^2,g\}$.  By construction $\chi_0(s) = s^2 \chi_0'(s)$ for $s>0$, so  
	\[
	\rho^{1/2} a = \rho^{1/2}(2-\phi/\delta) (\chi_0'\chi_0)^{1/2}\chi_1 = \delta^{1/2}(2-\phi/\delta)\rho b.
	\]
	Also note that $|2-\phi/\delta| < 4$ on $\supp b$. Using that $j_r$ is elliptic on $V_0$, given $\delta_0 >0$ sufficiently small we can write 
	\[
	a^2 \{j_r^2,g\} = d_{2,r} b_r^2,
	\] 
	where $\rho^{2-k}|d_{2,r}| \leq C_0\delta$. Finally, we let $d_r = d_{1,r} + d_{2,r}$.
	
	\item The term $e_r$ arises when $\chi_1$ is differentiated, hence has the desired support properties.
	\end{inparaenum}
\end{proof}

Now we expand $2i \Im \mathcal{E}_0(u, A_r^* A_r u)$, recalling from the end of Section \ref{subsect:dirichletmicrolocal} that
\begin{equation} \label{eq:commutatorexpansion}
\begin{split}  
2i \Im \mathcal{E}_0(u, A_r^*A_ru) &= \langle Q_0 u,  Q_0 A_{1,r}u \rangle - \langle Q_0 A_{1,r} u, Q_0 u \rangle \\ &+ \langle Q_0 u, A_{0,r} u \rangle - \langle A_{0,r} u, Q_0 u \rangle
\\ &+ \langle [Q_\alpha \hat g^{\alpha\beta}Q_\beta, A_r^*A_r]u,u \rangle,
\end{split} 
\end{equation}
where $\bsymbol{-1}(A_{1,r}) = (1/i)\pa_\sigma (a_r^2)$ and $\bsymbol{0}(A_{0,r}) = (1/i)\pa_x (a_r^2)$. 
\begin{lemm} \label{lem:hyperbolicpositivecommutator}
	There exist $C_1,c,\beta,\delta_0 > 0$, a cutoff $\chi \in \CcI(X)$, and an operator $G_2 \in \Psib^{s}(X)$ with 
	\[
	\opWFb(G_2) \subset W \cap \{\sigma < 0\},
	\] 
 such that
	\begin{multline*}
	c\| \tilde B_r u \|^2_{\tSob^1(X)}  \leq - 2 \Im \mathcal{E}_0(u, A_r^*A_ru) + C_1 \| G_2 u \|_{\tSob^1(X)}^2\\
+ C_1( \|G_0 P_{\Robin} u\|_{\tdotSob^{-1}(X)}^2 + \| G_1 u \|_{\tSob^1(X)}^2 + \| \chi u \|^2_{\tSob^{1,m}(X)} + \|\chi  P_{\Robin} u\|_{\tdotSob^{-1,m}(X)}^2)
	\end{multline*}
	for every $\delta \in (0,\delta_0)$.
\end{lemm} 
\begin{proof}
Note that $\opWFb(G_2) \cap \WFb^{1,s}(u) = \emptyset$ by the hypotheses on $u$. According to \eqref{eq:Qn-1control}, 
 \[
 \mathcal{Q}(\tilde B_r u, \tilde B_r u) \leq C (\| B_r u \|^2_{\Ltwo(X)} + \mathcal{E}_0( \tilde B_r u, \tilde B_r u) + \| G_1 u \|_{\tSob^1(X)}^2)
 \]
provided $\delta > 0$ is sufficiently small. Here we used the fact that $B_r$ and $D_{\eta_{n-1}}\tilde B_r$ have the same principal symbol (up to a sign which is irrelevant when taking norms).  On the other hand,  Lemma \ref{lem:dirichletform2} applies to the family $\{ \tilde B_r: r\in (0,1)\}$, and thus
\begin{multline*}
 \mathcal{Q}(\tilde B_r u, \tilde B_r u) \leq C \| B_r u \|^2_{\Ltwo(X)}
 \\ + C_1 \|G_0 P_{\Robin} u\|_{\tdotSob^{-1}(X)}^2 + \| G_1 u \|_{\tSob^1(X)}^2 + \| \chi u \|^2_{\tSob^{1,m}(X)} + \|\chi  P_{\Robin} u\|_{\tdotSob^{-1,m}(X)}^2).
\end{multline*}
Therefore it suffices to bound $\| B_r u \|_{\Ltwo(X)}$ using \eqref{eq:commutatorexpansion}.

Arguing as in Lemma \ref{lem:commutatorexpansion2}, we can expand $i A_{1,r} = \tilde B_r^* \tilde B_r + E_r + T_r$,
	where $E_r \in \Psib^{2s}(X)$ and $T_r \in \Psib^{2s-1}(X)$. Therefore
	\[
	\langle Q_0 u, iQ_0 A_{1,r} u \rangle = \langle Q_0 u, Q_0 (\tilde B_r^* \tilde B_r + E_r + T_r)u \rangle.
	\]
	The pairing involving $\tilde B_r^* \tilde B_r $ can be re-expressed in terms of the Dirichlet form itself:
	\begin{equation} \label{eq:mainhyperbolicterm}
	\langle Q_0, iQ_0A_{1,r}u \rangle = \mathcal{E}_0(u, \tilde B_r^* \tilde B_r  u) - \langle \hat g^{\alpha\beta}Q_{\alpha} u, Q_\beta \tilde B_r^* \tilde B_r u \rangle +  \langle Q_0u, Q_0(E_r + T_r )u \rangle,
	\end{equation}
	recalling that $E_r,\,T_r$ are allowed to change from line to line. The last term on the right hand side of \eqref{eq:mainhyperbolicterm} is bounded by a constant times $\| G_1 u \|_{\tSob^1(X)}^2 + \| G_2 u \|_{\tSob^1(X)}^2 + \| \chi u \|^2_{\tSob^{1,m}(X)}$. As for $\mathcal{E}_0(u, \tilde B_r^* \tilde B_r  u)$, it is bounded by acceptable terms by arguing as in Lemmas \ref{lem:dirichletform1} and \ref{lem:dirichletform2}. Finally, we can write
	\[
	\langle \hat g^{\alpha\beta}Q_{\alpha} u, Q_\beta \tilde B_r^* \tilde B_r u \rangle = \| CB_r u \|_{\Ltwo(X)}^2.
	\]
	modulo acceptable errors, where $C$ has principal symbol \eqref{eq:Csymbol}.
	
	It remains to consider the second and third lines of \eqref{eq:commutatorexpansion}. For the third term, we apply Lemma \ref{lem:commutatorexpansion2} to write
	\[
	\langle i[Q_\alpha \hat g^{\alpha\beta}Q_\beta, A_r^*A_r]u,u \rangle = B_r^* D_r B_r + E_r + T_r,
	\]
	where $D_r \in \Psib^0(X)$ satisfies 
	\[
	|\bsymbol{0}(D_r)| \leq C_0(\beta\delta + \delta + \beta^{-1}).
	\] 
	Given $\varepsilon>0$, we can first fix $\beta >0$ large and then take $\delta_1 > 0$ sufficiently small so that 
	\[
	|\langle [Q_\alpha \hat g^{\alpha\beta}Q_\beta, A_r^*A_r]u,u \rangle| \leq \varepsilon \| B_r u \|_{\Ltwo(X)}^2 + C( \| G_1 u \|_{\tSob^1(X)}^2 + \| G_2 u \|_{\tSob^1(X)}^2 + \| \chi u \|^2_{\tSob^{1,m}(X)}).
	\]
As	for the term $\langle Q_0, iA_{0,r}u\rangle$, note that we can write $iA_{0,r} = B_r^* D_r' B_r + T_r$, where $T_r \in \Psib^{2s}(X)$ and $D'_r \in \Psib^0(X)$ satisfies 
\[
|\bsymbol{0}(D_r')| \leq C_0 \beta \delta.
\]
Thus we can similarly bound
\[
|\langle Q_0, iA_{0,r}u\rangle| \leq \varepsilon \| B_r u \|_{\Ltwo(X)}^2 + C( \| G_1 u \|_{\tSob^1(X)}^2 + \| G_2 u \|_{\tSob^1(X)}^2 + \| \chi u \|^2_{\tSob^{1,m}(X)}).
\]
We can now take $\varepsilon$ sufficiently small, noting that $C$ is elliptic on $\opWFb(\mathcal{B})$, where $\mathcal{B} = \{B_r: r \in (0,1)\}$. 
\end{proof}

The final step in the proof of Proposition \ref{prop:hyperbolicregion} is to bound $\Im \mathcal{E}_0(u, A_r^*A_ru)$. 

\begin{lemm} \label{lem:hyperbolicerrorterms}
Given $\varepsilon > 0$, there exists $\beta  >0$ and $\delta_0 > 0$ such that
\begin{multline*}
\Im \mathcal{E}_0(u,A_r^* A_r u) \leq \varepsilon \| \tilde B_r u \|^2_{\tSob^1(X)} + C_1 \| G_2 u \|_{\tSob^1(X)}^2\\
+ C_1( \|G_0 P_{\Robin} u\|_{\tdotSob^{-1}(X)}^2 + \| G_1 u \|_{\tSob^1(X)}^2 + \| \chi u \|^2_{\tSob^{1,m}(X)} + \|\chi  P_{\Robin} u\|_{\tdotSob^{-1,m}(X)}^2)
\end{multline*}
for every $\delta \in (0,\delta_0)$.
\end{lemm}
\begin{proof}
Let $\Lambda_{-1/2} \in \Psib^{1/2}(X)$ be elliptic. First, we bound $\mathcal{E}_{\Robin}(u,A_r^* A_r u)$ similar to the beginning of proof of Lemma \ref{lem:dirichletform2}:
\[
|\mathcal{E}_{\Robin}(u,A_r^* A_r u)| \leq \varepsilon \| \Lambda_{-1/2} A_r u \|^2_{\tSob^1(X)} +  C_\varepsilon(\| G_0 f\|^2_{\tdotSob^{-1}(X)} + \| G_1 u \|^2_{\tSob^{1,m}(X)} +\| \chi u \|^2_{\tSob^{1,m}(X)}). 
\]
On the other hand, as in the proof of Lemma \ref{lem:commutatorexpansion2} we can write $\Lambda_{-1/2}A_r = LB_r + T_r$ with $L \in \Psib^0(X)$ and $T_r \in \Psib^{s-1}(X)$; thus we can bound
\[
\| \Lambda_{-1/2} A_r u \|^2_{\tSob^1(X)} \leq C(\| \tilde B_r u \|_{\tSob^1(X)}^2 + \| G_1 u \|^2_{\tSob^1(X)}).
\]
It thus remains to bound the difference between $\Im \mathcal{E}_R(u,A_r^* A_r u)$ and $\Im \mathcal{E}_0(u,A_r^* A_r u)$. There are two terms to consider. Since $S_F$ is real-valued, the first is
\[
\langle A_r x^{-2}S_F u, A_r u \rangle - \langle A_r u, A_r x^{-2}S_F u \rangle = \langle A_r^*[A_r,x^{-2}S_F] u, u \rangle -  \langle u, A_r^*[A_r,x^{-2}S_F] u \rangle. 
\]
Note that $A_r^*[A_r,x^{-2}S_F]$ is uniformly bounded in $\Psib^{2s}(X)$, so if $G_1 \in \Psib^{s-1/2}(X)$ is elliptic on $\opWFb(A)$, then
\[
|\Im \langle  x^{-2} S_F u ,A_r^* A_r u\rangle| \leq C( \| G_1 u \|_{\tSob^1(X)}^2 + \| \chi u \|_{\tSob^{1,m}(X)}^2).
\]
Of course this actually true if the order of $G_1$ is merely $s-1$. 

Next, we must consider the boundary terms. In the setting of smooth boundary value problems, it is easy to see that the analogous boundary terms are bounded by (the appropriate analogue of) a multiple of  $\| G_1 u \|_{\tSob^1(X)}^2 + \| \chi u \|_{\tSob^{1,m}(X)}^2$; see Remark \ref{rem:traceinterpolation}. Unfortunately, it is not obvious how to do this in general. Instead, write
\begin{multline*}
\langle \beta \gamma_- u, \gamma_-(A_r^*A_r u) \rangle_{\pa X} - \langle \gamma_-(A_r^*A_r u), \beta \gamma_- u\rangle_{\pa X} 
\\ = \langle \gamma_-([A_r^* A_r,\beta]u), \gamma_- u \rangle_{\pa X} + \langle \beta \gamma_-((\tilde A_r^* \tilde A_r - A_r^* A_r)u),\gamma_- u \rangle_{\pa X},
\end{multline*}
where $\tilde A_r$ has the same principal symbol as $A_r$. Here we have extended $\beta$ arbitrarily to a function on $X$; it is particularly convenient to choose this extension to be independent of $x$. Now consider $[A_r^* A_r,\beta] \in \Psib^{2s}(X)$. Lemma \ref{lem:commutatorexpansion2} applies to this commutator with $m=0$. Because we are assuming that $\beta$ is independent of $x$, we can take $E_r =0$. Thus we can write
\begin{equation} \label{eq:boundarycommutator1}
i[A_r^* A_r,\beta] = \tilde B_r \tilde D_r \tilde B_r + T_r,
\end{equation}
where $T_r \in \Psib^{2s-1}(X)$, and $\tilde D_r \in \Psib^{0}(X)$. Then
\[
\langle \gamma_-([A_r^* A_r,\beta]u), \gamma_- u \rangle_{\pa X} = \langle \gamma_-  (\tilde D_r \tilde B_r u), \gamma_-(\tilde B_r u) \rangle_{\pa X}.
\]
modulo terms bounded by $\| G_1 u \|_{\tSob^1(X)}^2 + \| \chi u \|_{\tSob^{1,m}(X)}^2$. Now write
\[
\langle \gamma_-  (\tilde D_r \tilde B_r u), \gamma_-(\tilde B_r u) \rangle_{\pa X} = \langle \widehat{N}(-i\nu_-)(\tilde D_r)\widehat{N}(-i\nu_-)(\tilde B_r)\gamma_- u, \widehat{N}(-i\nu_-)(\tilde B_r)\gamma_- u \rangle_{\pa X}.
\]
In coordinates $(x,y,\sigma,\eta)$, the total symbol of $\widehat{N}(-i\nu_-)(\tilde D_r) \in \Psi^0(\pa X)$ is just 
\[
\bsymbol{0}(\tilde D_r)(0,y,-i\nu_-,\eta) \in S^0(T^*\pa X),
\] 
hence $\widehat{N}(-i\nu_-)(\tilde D_r)$ forms a bounded family. Using \eqref{eq:traceinterpolation}, we can bound 
\[
|\langle \gamma_-  (\tilde D_r \tilde B_r u), \gamma_-(\tilde B_r u) \rangle_{\pa X}| \leq \varepsilon\|\tilde B_r u \|^2_{\tSob^1(X)} + C(\| G_1 u \|_{\tSob^1(X)}^2 + \| \chi u \|_{\tSob^{1,m}(X)}^2).
\] 
Notice that the supremum of $|\bsymbol{0}(\tilde D_r)|$ can in fact be made small by choosing $\beta,\delta$ appropriately, but that is not needed here.

Next, consider $\langle \beta \gamma_-((\tilde A_r^* \tilde A_r - A_r^* A_r)u),\gamma_- u \rangle_{\pa X}$. Note that $\tilde A_r^* \tilde A_r = x^{2\nu_-}A_r^*A_r x^{-2\nu_-}$, and hence
\[
\tilde A_r^* \tilde A_r - A_r^* A_r = x^{2\nu_-}[A^*_r A_r, x^{-2\nu_-}].
\]
Thus the principal symbol of $\tilde A_r^* \tilde A_r - A_r^* A_r \in \Psib^{2s}(X)$ is just $2i\nu_-\pa_\sigma(a_r^2)$, and we can write
\begin{equation} \label{eq:boundarycommutator2}
\tilde A_r^* \tilde A_r - A_r^* A_r = (2i\nu_-)\tilde B_r^* \tilde B_r + E_r + T_r
\end{equation}
as in Lemma \ref{lem:hyperbolicpositivecommutator}. Again using \eqref{eq:traceinterpolation}, the corresponding boundary term can be estimated by $\varepsilon\|\tilde B_r u \|^2_{\tSob^1(X)} + C(\|G_2 u \|_{\tSob^1(X)}^2  + \| G_1 u \|_{\tSob^1(X)}^2 + \| \chi u \|_{\tSob^{1,m}(X)}^2)$.
\end{proof}

Combining Lemmas \ref{lem:hyperbolicpositivecommutator} and \ref{lem:hyperbolicerrorterms} shows that $\tilde B_r u$ is uniformly bounded in $\tSob^1(K)$ for an appropriate compact set $K \subset X$. Since the limiting operator $B_0 \in \Psib^s(X)$ as $r\rightarrow 0$ is elliptic at $q_0$, it follows that $q_0 \notin \WFb^{1,s}(u)$ by the same argument as in the proof of Theorem \ref{theo:ellipticregion}. Arguing inductively, this establishes Proposition \ref{prop:hyperbolicregion} for finite $s$. For infinite $s$ one merely needs to be more careful in choosing the operators $G_0,\, G_1,\, G_2$ at each step of the induction, and is done exactly as in the proof of \cite[Proposition 6.2]{vasy2008propagation}. \qed

%	In the smooth setting, the boundary terms can be bounded by the inductive hypothesis. This is because 
%\[
%\langle u, v \rangle_{\pa X}  = \langle Tu, v \rangle + \langle u, T' v\rangle 
%\]
%for some operators $T,T' \in \Diff^1(X)$. If $B_r \in \Psib^{s-1/2}(X)$ and $A_r \in \Psib^{s+1/2}(X)$, then we can write
%\[
%\langle B_ru, A_rv\rangle = \langle TB_ru, A_r v \rangle + \langle B_ru, T'A_r v\rangle.
%\]
%Clearly the first term is bounded by the inductive hypothesis, while the second term is bounded by taking $\Lambda_{-1}$ elliptic and writing (modulo lower order terms)
%\[
%\langle B_ru, T'A_r v\rangle = \langle \Lambda_1^* B_r, T'  \Lambda_{-1} A_r v \rangle
%\]
%I'm not sure how to prove an analogous result in our setting using twisted derivatives.

\subsection{The glancing region} \label{subsect:glancingregion}

Next we consider the glancing set $\gl$. Fix a boundary coordinate patch as in Section \ref{subsect:hyperbolicregion}. If $q_0 \in \gl \cap T^* \pa X \cap \bT^*_U X$, then in coordinates $(x,y,\sigma,\eta)$,
\[
q_0 = (0,y_0,0,\eta_0), \quad \hat g^{\alpha\beta}(0,y_0)(\eta_{0})_{\alpha}(\eta_0)_{\beta}= 0.
\] 
Since $\eta_{n-1} \neq 0$ at $q_0$, we continue to use projective coordinates $(x,y,\hat \sigma, \hat \eta_a)$ on $\bS^*X$ near $\kappa(q_0)$. We need the following result for the Dirichlet form:

\begin{lemm} \label{lem:dirichletformglancing}
	Let $U \subset X$ be a boundary coordinate patch, and $m \leq 0$. Let $\mathcal{A} = \{A_r: r \in (0,1)\}$ be a bounded subset of $\Psib^s(X)$ with compact support in $U$ such that 
	\[
	A_r \in \Psib^{m}(X) \text{ for each } r \in (0,1).
	\]
	Let $V_\delta = \{q \in \bT^*_U X \setminus 0 : \hat g^{\alpha \beta} \zeta_{\alpha}\zeta_{\beta} \leq \delta |\zeta_{n-1}|^2 \}$, and assume that $\opWFb(\mathcal{A}) \subset V_\delta$. 
	Let $G_0\in \Psib^{s}(X)$ and $ G_1 \in \Psib^{s-1/2}(X)$ both be elliptic on $\opWFb(\mathcal{A})$, with compact support in $U$.  Then there exists $C_\varepsilon>0$ and $\chi \in \CcI(U)$ such that
	\begin{multline*}
	\| Q_0 A_r u \|^2_{\Ltwo(X)} - \varepsilon\mathcal{Q}(A_r u, A_r u) \leq  2\delta \|Q_{n-1}A_r u \|^2_{\Ltwo(X)}  \\
	+ C_\varepsilon ( \| \chi u\|^2_{\tSob^{1,m}(X)} + \| \chi P_{\Robin} u\|^2_{\tdotSob^{-1,m}(X)}  + \| G_1 u \|^2_{\tSob^1(X)} + \| G_0 P_{\Robin} u  \|^2_{\tdotSob^{-1}(X)}) 
	\end{multline*}
	for every $r \in (0,1)$ and every $u \in \tSob^{1,m}_{\loc}(X)$, provided
	\[
	\WFb^{1,s-1/2}(u) \cap \opWFb(G_1) = \emptyset, \quad \WFb^{-1,s}(P_{\Robin}u ) \cap \WFb'(G_0) = \emptyset. 
	\]
\end{lemm}
\begin{proof}
	We have $\| Q_0 A_r u \|^2 =  \langle \hat g^{\alpha\beta}Q_\alpha A_r u , Q_\beta A_r u \rangle + \mathcal{E}_0(A_r u,A_r u)$. According to Lemma \ref{lem:dirichletform2},
	\begin{multline*}
	\| Q_0 A_r u \|^2 - \varepsilon\mathcal{Q}(A_r u, A_r u)  \leq   \langle Q_\beta \hat g^{\alpha\beta}Q_\alpha A_r u , A_r u \rangle \\ + C_\varepsilon ( \| \chi u\|^2_{\tSob^{1,m}(X)} + \| \chi P_{\Robin} u\|^2_{\tdotSob^{-1,m}(X)} +  \| G_1 u \|^2_{\tSob^1(X)} + \| G_0 f  \|^2_{\tdotSob^{-1}(X)}).
	\end{multline*}
	Choose $F \in \Psib^0(X)$ such that $\opWFb(Q_{n-1}FQ_{n-1} + Q_\beta \hat g^{\alpha\beta}Q_\alpha) \cap \opWFb(\mathcal{A}) = \emptyset$ and $\opWFb(F) \subset V_\delta$. In particular 
	\[
	\sup |\bsymbol{0}(F)|\leq \delta,
	\]
	which implies that $|\langle Q_\beta \hat g^{\alpha\beta}Q_\alpha A_r u , A_r u \rangle| \leq 2\delta \| Q_{n-1} A_r u\|^2_{\Ltwo(X)} + C\| \chi u\|^2_{\tSob^{1,m}(X)} $ as desired.
\end{proof}

To formulate the relevant estimates in the glancing region, we follow \cite[Section 7]{vasy2008propagation} and introduce in local coordinates the map $\tilde \pi : T^*_U X \rightarrow T^* (U\cap \pa X)$ given by
\[
\tilde \pi(x,y,\xi,\eta) = (y,\eta).
\]
Let $\mathsf{W}$ denote the ``gliding vector field'' on $T^*(U \cap \pa X)$ given in coordinates by
\begin{align*}
\mathsf{W}(y,\eta) &= \sum_{i=1}^{n-1}  (\pa_{\eta_i}\hat p) (0,y,0,\eta) \pa_{y^i} - (\pa_{y^i} \hat p)(0,y,0,\eta) \pa_{\eta_i}. 
\end{align*} 
Note that $\mathsf W$ is just the b-Hamilton vector field of $\hat g^{\alpha\beta}\eta_{\alpha}\eta_{\beta} \in S^2(\bT^*X)$ (which is hence tangent to $T^*\partial X$) restricted to $T^*\partial X$.

We then prove the following analogue of \cite[Proposition 7.3]{vasy2008propagation}:

\begin{prop} \label{prop:glancingregion}
	Let $u \in \tSob^{1,m}_\loc(X)$ with $m\leq 0$. If $K \subset \bS^*_U X$ is compact and
	\[
	K \subset (\gl \cap T^*\pa X) \setminus \WFb^{-1,s+1}(P_{\Robin} u),
	\] 
	then there exist $C_0,\delta_0 >0$ such that for each $q_0 \in K$ and $\delta \in (0,\delta_0)$ the following holds. Let $\alpha_0 \in \chare$ be such that $\pi(\alpha_0) = q_0$. If the conditions
	\begin{equation} \label{eq:alphacondition}
	\alpha \in \chare, \quad |\tilde \pi (\alpha) - \exp(-\delta W)(\tilde \pi(\alpha_0))| \leq C_0 \delta^2, \quad  |x(\alpha)| \leq C_0\delta^2
	\end{equation}
	imply $\pi(\alpha) \notin \WFb^{1,s}(u)$, then $q_0 \notin \WFb^{1,s}(u)$.
\end{prop}

The proof Proposition \ref{prop:glancingregion} goes through nearly exactly as in \cite[Section 7]{vasy2008propagation}, working directly with the Dirichlet form, just as we did in Section \ref{subsect:hyperbolicregion}. For this reason we only briefly sketch a proof of the proposition, referring to  \cite[Section 7]{vasy2008propagation} for more details.

Fix $q_1 = (0,y_1,0,\eta_1) \in K$. Note that $\rho^{-1} \mathsf{W}$ descends to a vector field on $S^*\pa X$, at least near $(y_1,\eta_1)$, recalling that $\rho = |\eta_{n-1}|$. Note that $\rho^{-1}\mathsf{W}$ does not vanish near $(y_1,\eta_1)$, since 
\[
\rho^{-1} \mathsf{W}y^{n-1} = 2\,\sgn(\eta_{n-1})
\]
near $q_1$. Thus in a neighborhood of $(y_1,(\hat \eta_1)_a) \in S^*\pa X$ we straighten the $\rho^{-1}\mathsf{W}$ flow, finding homogeneous degree zero functions $\rho_1,\ldots, \rho_{2n-3}$ on $T^*\pa X$ with linearly independent differentials such that
\[
\rho^{-1}\mathsf{W}\rho_1 = 1, \quad \rho^{-1}\mathsf{W}\rho_i = 0 \text{ for } i= 2,\ldots, 2n-3.
\]
In fact, since $\mathsf{W}$ annihilates $\hat p(0,y,0,\eta)$, and since this function has a non-vanishing differential, we can always take 
\[
\rho_{2}(y,\eta) = \hat p(0,y,0,\eta)
\]
The functions $\rho_1,\ldots,\rho_{2n-3}$ are then extended to be independent of $(x,\sigma)$, so that $(x,\hat \sigma, \rho_1,\ldots \rho_{2n-3})$ form a valid coordinate system near $\kappa(q_1)$, say in some neighborhood $V$. Now for $q_0 \in K\cap V$, introduce the function
\[
\omega_0 = \sum_{i=1}^{2n-3} (\rho_i - \rho_i(q_0))^2, \quad \omega = x^2 + \omega_0.
\]
We then define $\phi_0 = \rho_1 + (\beta^2\delta)^{-1}\omega_0$ and $\phi = \rho_1 + (\beta^2\delta)^{-1}\omega$. With $\chi_0,\,\chi_1$ denoting the same cutoff functions as in Section \ref{subsect:hyperbolicregion}, set
\[
a = \chi_0(2-\phi/\delta)\chi_1(1+(\rho_0+\delta)/(\beta\delta)).
\]
Note that when $a$ is differentiated, any derivatives falling onto $\chi_1$ yield a term supported on
\begin{equation} \label{eq:glancingcontrolregion}
\{-\delta\beta \leq \rho_1 \leq -\delta, \ \omega^{1/2} \leq 2\beta\delta\}.
\end{equation}
Suppose we take $\beta = c_0 \delta$ for some fixed $c_0 > 0$. If $\alpha \in \chare$ and $\pi(\alpha)$ is contained in the set \eqref{eq:glancingcontrolregion}, then $\alpha$ satisfies \eqref{eq:alphacondition} for $C_0 > 0$ sufficiently large depending on $c_0$. 

Recall that the restriction of the Poisson bracket $\rho^{-1}\{\hat g^{\alpha\beta}\eta_{\alpha}\eta_{\beta},\phi\}$ to $\bT^*_{\pa X}X$ is $\rho^{-1}\mathsf{W}\phi =1$. Since $|x| \leq \omega^{1/2}$,
\[
| \rho^{-1}\{\hat g^{\alpha\beta}\eta_{\alpha}\eta_{\beta}, \phi\} - 1| \leq C_1(1+ \beta^{-2}\delta^{-1}\omega^{1/2})\omega^{1/2}.
\]
Observe that $C_1 > 0$ is not only uniform on $V$ for $q_0$ fixed, but it is also uniform on $V$ as $q_0$ ranges over $V_0 \cap K$, where $V_0\subset V$ is a neighborhood of $\kappa(q_1)$. Since $K$ is compact, it suffices to prove Proposition \ref{prop:glancingregion} with $V_0 \cap K$ replacing $K$, and as we shall see, the uniformity of $C_1$ implies the uniformity of $c_0$ in the preceding paragraph.

Let $A$ have principal symbol $a$. With $J_r$ denoting the same operator as in Section \ref{subsect:hyperbolicregion}, let $A_r = AJ_r$. Also define $B_r$ and $\tilde B_r$ as in Section \ref{subsect:hyperbolicregion}. We let $E_r, \, T_r$ denote operators as in \eqref{eq:genericapriori}, \eqref{eq:genericlowerorder}, except that everywhere $\hat \sigma$ should be replaced with $\rho_1$, and the right hand side of \eqref{eq:genericapriori} should be replaced with \eqref{eq:glancingcontrolregion}. Finally, let $G_0,G_1$ be the same as in Section \ref{subsect:hyperbolicregion}. We then have the analogue of Lemma \ref{lem:hyperbolicpositivecommutator}:

\begin{lemm} \label{lem:glancingpositivecommutator}
	There exist $C_1,c,\beta,\delta_0 > 0$, a cutoff $\chi \in \CcI(X)$, and an operator $G_2 \in \Psib^{s}(X)$ with 
	\[
	\opWFb(G_2) \subset W \cap \{-2\delta\beta < \rho_1 < -\delta/2,\ \omega^{1/2} < 3\beta \delta\}
	\] 
	such that
	\begin{multline*}
	c\| \tilde B_r u \|^2_{\tSob^1(X)}  \leq - 2 \Im \mathcal{E}_0(u, A_r^*A_ru) + C_1 \| G_2 u \|_{\tSob^1(X)}^2\\
	+ C_1( \|G_0 P_{\Robin} u\|_{\tdotSob^{-1}(X)}^2 + \| G_1 u \|_{\tSob^1(X)}^2 + \| \chi u \|^2_{\tSob^{1,m}(X)} + \|\chi  P_{\Robin} u\|_{\tdotSob^{-1,m}(X)}^2)
	\end{multline*}
	for every $\delta \in (0,\delta_0)$ and $c_0\delta < \beta < 1$.
\end{lemm} 
\begin{proof}
	As in the proof of Lemma \ref{lem:hyperbolicpositivecommutator}, consider the expansion \eqref{eq:commutatorexpansion}. In this case, since $a$ is independent of $\sigma$, we can write $A_{1,r} = T_r$ with $T_r \in \Psib^{2s-1}(X)$. As for $A_{0,r}$, we have that $\bsymbol{2s+1}(A_{0,r}) = (1/i)\pa_x a_r$, and since the only dependence on $x$ is through $\phi$, we can write
	\[
	iA_{0,r} = \tilde B_r^* D_r Q_{n-1}\tilde B_r +T_r,
	\]
	where $T_r \in \Psib^{2s}(X)$ and $D_r \in \Psib^{0}(X)$. Furthermore, since $|x| \leq \omega^{1/2}$,
	\[
	\sup |\bsymbol{1}(D_r)| \leq C\beta^{-1}
	\]
	Finally, consider the term $i[Q_\alpha \hat g^{\alpha\beta} Q_{\beta},A_r^*A_r] \in \Psib^{2s+2}(X)$, which in analogy with Lemma \ref{lem:commutatorexpansion2} we can write as
	\[
	i[Q_\alpha \hat g^{\alpha\beta} Q_{\beta},A_r^*A_r] = B_r^* (1+D_r') B_r^* + E_r + T_r.
	\]
As in Lemma \ref{lem:commutatorexpansion2}, $D_r'$ has principal symbol $d_{1,r} + d_{2,r}$, corresponding to when $\chi_0$ or $j_r$ is differentiated. When $\chi_0$ is differentiated we are left with a term
	\[
	-2\rho^{-1}\{\hat g^{\alpha\beta}\eta_{\alpha}\eta_{\beta}, \phi\} b_r^2,
	\]
and hence $\sup |d_{1,r}| \leq C(\beta\delta + \delta)$; in this case $d_{1,r}$ can be taken independent of $r$. To handle the terms where $j_r$ is differentiated we argue as in Lemma \ref{lem:hyperbolicpositivecommutator} to see that $\sup |d_{2,r}| \leq C\delta$. 

As observed in Lemma \ref{lem:hyperbolicpositivecommutator}, it suffices to control $\|B_r u\|^2_{\Ltwo(X)}$. Write
\[
\langle i[Q_\alpha \hat g^{\alpha\beta} Q_{\beta},A_r^*A_r]u,u\rangle = \| B_r u \|^2_{\Ltwo(X)} + \langle D_r' B_r u , B_ru \rangle + \langle E_r u, u, \rangle + \langle T_r u,u\rangle.
\]
Modulo error terms bounded by the a priori hypotheses, we can write
\[
|\langle Q_0u, iA_{0,r}u \rangle|  =  \langle \tilde B_r Q_0 u, D_r Q_{n-1}\tilde B_r u \rangle = \langle Q_0 \tilde B_ru, D_r Q_{n-1}\tilde B_r u\rangle
\]
and hence 
\begin{align*}
|\langle Q_0u, iA_{0,r}u \rangle| \leq  \| Q_0 \tilde B_r u \|_{\Ltwo(X)} \| \tilde D_r Q_{n-1} \tilde B_r u \|_{\Ltwo(X)}+ C(\|G_1 u \|_{\tSob^1(X)}^2 + \|\chi u \|_{\tSob^{1,m}(X)}^2)
\end{align*}
Now $\opWFb(\tilde B_r) \subset \{\omega^{1/2} \leq 2\beta\delta\}$, and since $|\rho_2(y,\eta)| = |\hat p(0,y,0,\eta)| \leq \omega^{1/2}$ and $|x| \leq \omega^{1/2}$, it follows that $|\hat g^{\alpha\beta}\eta_\alpha \eta_{\beta}| \leq \beta\delta$ on $\opWFb(\tilde B_r)$. If we are given $\varepsilon > 0$ and let $c_0\delta < \beta < 1$ for $c_0 > 0$ sufficiently large, then an application of Lemma \ref{lem:dirichletformglancing},
\begin{multline*}
|\langle Q_0u, iA_{0,r}u \rangle| \leq \varepsilon \mathcal{Q}(\tilde B_r u, \tilde B_r u)  + C_\varepsilon \|G_2 u\|^2_{\tSob^1(X)}\\ + 
C_\varepsilon ( \| \chi u\|^2_{\tSob^{1,m}(X)} + \| \chi P_{\Robin} u\|^2_{\tdotSob^{-1,m}(X)}  + \| G_1 u \|^2_{\tSob^1(X)} + \| G_0 P_{\Robin} u  \|^2_{\tdotSob^{-1}(X)}). 
\end{multline*}
Since we have controlled $\| B_r u \|_{\Ltwo(X)}^2 - \varepsilon \mathcal{Q}(\tilde B_r u, \tilde B_r u)$, the proof is complete.
\end{proof}

The proof of Proposition \ref{prop:glancingregion} now follows by combining Lemma \ref{lem:glancingpositivecommutator} and the analogue of Lemma \ref{lem:hyperbolicerrorterms}. Note that estimating the boundary terms as in the latter lemma is done slightly differently. For the analogue of \eqref{eq:boundarycommutator1} there appears an extra term $E_r \in \Psib^{2s}(X)$ (which is of course harmless). The analogue of \eqref{eq:boundarycommutator2} is simpler: since $a_r$ is independent of $\sigma$, 
\[
\tilde A_r^* \tilde A_r - A_r^* A_r = T_r
\]
with $T_r \in \Psib^{2s-1}(X)$.

We now proceed inductively, the difference being that at each step of the iteration the commutant must be modified slightly. This is done exactly as in \cite[Section 7]{vasy2008propagation}.

\subsection{Propagation of singularities}

Combining Theorem \ref{theo:ellipticregion} and Propositions \ref{prop:hyperbolicregion}, \ref{prop:glancingregion} allows us to prove Theorem \ref{theo:GBBpropagation}. 

Given the ingredients discussed above, the details of proof are \emph{identical} to those of \cite[Theorem 8.1]{vasy2008propagation}, hence are omitted.

\section{Propagators and their singularities}\label{sec:propagators}

\subsection{Retarded and advanced propagators}  \label{subsect:propagators}

We begin by discussing well-posedness for the Klein--Gordon equation on an aAdS spacetime. In the case of Dirichlet boundary conditions, well-posedness for the forward problem was first studied by Vasy \cite{vasy2012wave} (see also \cite{holzegel2012well} for a related study of the Cauchy problem). In the setting of Robin boundary conditions, well-posedness of the Cauchy problem was studied by Warnick \cite{warnick2013massive}, including certain results on higher-order conormal regularity. It should also be noted that \cite{warnick2013massive} considers the situation where the metric is even modulo $\mathcal{O}(x^3)$, whereas \cite{vasy2012wave} does not.

We will need a more refined study of the forward Klein--Gordon problem, akin to the results of \cite{vasy2012wave}, in the case of Robin boundary conditions. In order to give a global formulation we make the following assumptions as in \cite{vasy2012wave,wrochna2017holographic} (cf. \cite[Section 24.1]{hormander1994analysis}):

\begin{hypo} We assume $(X,g)$ is an aAdS spacetime with the following properties.
\begin{enumerate} \itemsep6pt 
\item[$(\rm TF)$] There exists $t\in \cf(X;\RR)$ such that the level sets of $t$ are spacelike with respect to $\hat g$.
\item[$(\rm PT)$] The map $t : X \rightarrow \RR$ is proper.
\end{enumerate}
\end{hypo}

Given a choice of $t$, we orient $\chare$ by declaring $dt$ to be future-oriented. Let $\chare^\pm$ denote the corresponding future/past light-cones. This yields a decomposition
\[
\cchare = \cchare^+ \cup \cchare^-
\]
of the compressed characteristic set. The main well-posedness result we need is the following:
\begin{theo} \label{theo:wellposed}
Let $s, t_0 \in \RR$. If $f \in \tdotSob^{-1,s+1}_\loc(X)$ is supported in $\{ t \geq t_0 \}$, then there exists a unique $u \in \tSob^{1,s}_\loc(X)$ supported in $ \{t\geq t_0\}$ such that 
\[
P_\Robin u = f.
\] 
Furthermore, for each $K \subset X$ compact, there exists $K' \subset X$ compact and $C>0$ such that 
\[
\| u \|_{\tSob^{1,s}(K)} \leq C \| f \|_{\tdotSob^{-1,s+1}(K)}.
\]
\end{theo} 

Observe that Theorem \ref{theo:wellposed} exhibits the loss of one \emph{conormal} derivative from the source term to the solution. Now $\tSob^{1,s}_\comp(X) \subset \tSob^{0,s+1}_\comp(X)$ for each $s\in \RR$, so by transposition
\[
\tSob^{0,s}_\loc(X) \subset \tdotSob^{-1,s+1}_\loc(X).
\]
In particular, given $f \in \Ltwo_\loc(X)$, we obtain a solution $u \in \tSob^{1}_\loc(X)$ of $x^{-2}Pu = f$ satisfying Robin boundary conditions in the strong sense. Replacing $t$ with $-t$ yields a result for the backward problem as well. 

We sketch a proof of this result adapting the approach of \cite{vasy2012wave}, which along the way gives a different perspective on the twisted stress-energy tensor introduced in \cite{holzegel2013decay}. As usual, the key step is an appropriate energy estimate. 

We make a preliminary reduction: we will need that the level sets of $t$ meet $\pa X$ orthogonally with respect to $\hat g$. Note that the level sets of $t$ automatically meet $\partial X$ transversally, but since $t$ is provided by the problem there is no reason to assume the  intersection is orthogonal. On the other hand, to prove Theorem \ref{theo:wellposed}, given $\delta>0$ it suffices to replace $t'$ with an analogous function $t'$ satisfying the required orthogonality property, such that $\big|t-t'\big| < \delta$ on any given compact subset of $X$. This is arranged in \cite[Lemma 4.9]{vasy2012wave}, and we henceforth assume $dt$ and $dx$ are orthogonal along $\pa X$.

 Let $u \in \tSob^{1,1}_\loc(X)$. Given an appropriate real b-vector field $V \in \mathcal{V}_\be(X)$ with compact support, set $V' = FVF^{-1} \in \Diffb^{1}(X)$ and compute $2 \Re \langle P_R u , V'u \rangle$; note that this pairing makes sense by our assumption on the conormal regularity of $u$. For illustrative purposes, let us assume that the Robin function $\beta$ and $S_F$ both vanish identically; these terms can easily be handled in general. In particular $\langle P_R u ,v \rangle = \mathcal{E}_0(u,v)$. Setting $A = iV'$ and applying \eqref{eq:twistedcommutator},
\begin{align*}
2 \Re \mathcal{E}_0 (u, V'u) &= \langle \hat g^{ij} Q_j u , [Q_i,V']u \rangle + \langle \hat g^{ij} [ Q_j,V']u, Q_i u \rangle \\ &+ \langle [\hat g^{ij} ,V']Q_j u, Q_i u \rangle+ \langle (V'+(V')^*) \hat g^{ij} Q_j u , Q_i u \rangle.
\end{align*}
The commutator $[Q_i,V']$ can be computed exactly using \eqref{eq:Qcommutator}: if $V = V^j \pa_{z^j}$, then
\begin{align*}
[Q_i,V'] &= F[D_{z^i},F^{-1} V' F]F^{-1} \\ &= F[ D_{z^i}, V] F^{-1} 
\\ & = \pa_{z^i}(V^k)Q_k.
\end{align*}
It is interesting to observe that if one instead uses $V$ rather than its twisted version $V'$ (despite the fact that both are b-operators), there arise error terms that cannot be estimated appropriately.

We also have $[\hat g^{ij},V'] = [\hat g^{ij}, V]$. Finally, $(V')^* = F^{-1}V^*F$, and since the adjoint is taken with respect to $x^2 dg$,
\begin{align*}
(V')^* &= -F^{-1}VF - \divop_{\hat g} V + (n-2)x^{-1}V(x) \\
&= -V + FV(F^{-1}) - \divop_{\hat g}V + (n-2)x^{-1}V(x).
\end{align*}
In particular, this shows that 
\begin{equation} \label{eq:energycommutator}
2 \Re \mathcal{E}_0(u,V'u) = \langle B^{ij} Q_i u, Q_j u \rangle
\end{equation}
where 
\[
B = (\divop_{\hat g} V + 2FV(F^{-1}) + (n-2)x^{-1}V(x))\cdot \hat g^{-1} - \mathcal{L}_V \hat g^{-1}
\]
is symmetric. Now we take $V = fW$ for a suitable function $f$ and a real b-vector field $W$; expanding out the tensor $B$ yields
\begin{multline*}
B = (Wf + f\cdot \divop_{\hat g} W + 2FfV(F^{-1}) + (n-2)fx^{-1}W(x))\hat g^{-1} \\ - f\mathcal{L}_W \hat g^{-1} - 2 (\nabla_{\hat g} f) \otimes_s W,
\end{multline*}
where $\otimes_s$ denotes the symmetric tensor product. Observe that all the terms in $B$ where $f$ is differentiated can be written in the form $-2T_{\hat g}(W, \nabla_{\hat g} f)$, where 
\[
T_{\hat g}(W, \nabla_{\hat g} f) =  (\nabla_{\hat g} f) \otimes_s W - \tfrac{1}{2} \hat g(\nabla_{\hat g} f,W) \cdot  \hat g^{-1}
\]
is the stress-energy tensor (with respect to $\hat g$) of a scalar field associated to $W$ and $\nabla_{\hat g}f$.

We make our choices of $W$ and $f$ as follows. First, set $W = \nabla_{\hat g} t$, which is indeed in $\mathcal{V}_\be(X)$ by the assumption that the level sets of $t$ meet $\pa X$ orthogonally. Fix $t_0 < t_1$, and let $\delta > 0$ be small. Assume that $u \in \tSob^{1,1}_\loc(X)$ satisfies
\[
\supp u \subset \{ t_0 + \delta  \leq t \leq t_1\}.
\]
Let $\chi \in \CcI(\RR)$ satisfy the following properties:
\begin{itemize} \itemsep6pt 
	\item $\supp \chi \subset [t_0,t_1 + \delta]$,
	\item $\chi \geq 0$ everywhere and $\chi(s) > 0$ for $s \in [t_0+\delta,t_1]$,
	\item $\chi'(s) \leq 0$ for $s \in [t_0+\delta,t_1+\delta]$.
\end{itemize}
We then set $f = e^{-2\gamma t}( \chi \circ t)$. Notice that $\nabla_{\hat g} f = (\chi' \circ t - 2\gamma \chi\circ t)e^{-2\gamma t} \nabla_{\hat g} t$. Since $\nabla_{\hat g} t$ is strictly timelike, the stress-energy tensor term controls
\[
\gamma \int_X e^{-2\gamma t} H(d_Fu,d_F \bar u) \, x^2 d g.
\]
We can also control the $\Ltwo(X)$ norm of $u$ by writing 
\[
2 \Re \langle u, Vu \rangle = - \langle \divop_{\hat g}(V)u + (2-n)x^{-1}V(x) u,u \rangle.
\]
Again using that $-\divop_{\hat g}V = -Wf - f\cdot \divop_{\hat g} W$, we can bound the first term on the right-hand side 
\[
-Wf = -\hat g(\nabla_{\hat g}t, \nabla_{\hat g}f) \geq c_0 \gamma e^{-\gamma t} (\chi\circ t).
\]
This allows us to control $\gamma \| e^{-\gamma t} u \|^2_{\Ltwo(X)}$, and hence $\gamma \| e^{-\gamma t}u \|_{\tSob^{1}(X)}^2$ as well. All of the remaining terms can be absorbed into the latter positive term for large $\gamma > 0$ by Cauchy--Schwarz. Thus we obtain an estimate of the form
\[
\gamma \| e^{-\gamma t} u \|_{\tSob^{1}(X)} \leq C\gamma^{-1} \| e^{-\gamma t}P_R u \|_{\tdotSob^{-1,1}(X)}.
\]
Equipped with this energy estimate, standard functional-analytic arguments as in \cite[Section 4 \& Theorem 8.12]{vasy2012wave} and Theorem \ref{theo:GBBpropagation} on propagation of singularities allow one to deduce Theorem \ref{theo:wellposed}.

%\begin{lemm} 
%Given $t_0 < t_1 < t_2$, there exists $\gamma_0$ such that 
%\[
%\| e^{-\gamma t} u \|_{\tSob^{1}} \leq C\| e^{-\gamma t} P_R u \|_{\tdotSob^{-1}(X)}
%\]
%for all $u \in \tSob^{1,1}_\loc(X)$ with $\supp u \subset \{ t_1 \leq t \leq t_2 \}$ and $\gamma_0 > 0$. \todo{add uniformity statement in terms of $|t_2-t_0|$ and $\gamma_0$}.
%\end{lemm} 

%\begin{rema}
%	The hypothesis that $W$ be tangent to $\partial X$ (reflecting the fact that we are measuring conormal regularity relative to $\tSob^1(X)$) is related to our poor understanding of the Klein--Gordon equation with inhomogeneous boundary conditions. For instance, in order to derive lossless estimates for the Dirichlet problem using multiplier methods (as in \cite[Section 24.1]{hormander1994analysis} or \cite{garding1977probleme}), one uses a vector field $W$ that is \emph{not} tangent to $\partial X$; this is easily seen to cause problems as soon as $\nu \neq 1/2$. The importance of using a multiplier tangent to $\partial X$ is not emphasized in \cite{holzegel2014boundedness,warnick2013massive,warnick2015quasinormal}.
%\end{rema}

%Examples include the universal cover of AdS (discuss some black hole example?).

Let us denote by $\tSob^{1}_{\pm}(X)$ the space of future/past supported elements of $\tSob^{1}_{\loc}(X)$, i.e.
\beq\label{eq:dsdfsd}
\tSob^{1}_{\pm}(X)= \big\{ u\in \tSob^{1}_{\loc}(X) :    \supp u \subset \{\pm t\geq \pm t_0\} \mbox{ for some } t_0\in\rr\big\},
\eeq
and let us define $\tdotSob^{-1}_{\pm}(X)$ analogously. By hypothesis $\pt$, the elements of the intersection $\tSob^{1}_{+}(X)\cap \tSob^{1}_{-}(X)$ are compactly supported in $X$.

\begin{corr} There exist unique \emph{retarded/advanced propagators} $P_{{\Robin},\pm}^{-1}$, i.e.~continuous operators
\beq\label{pp0}
P_{\Robin,\pm}^{-1}: \tdotSob^{-1,s+1}_{\pm}(X)\to \tSob^{1,s}_{\pm}(X)
\eeq
such that $P_{\Robin} P_{\Robin,\pm}^{-1}= 1$ on $\tdotSob^{-1}_{\pm}(X)$ and $P_{{\Robin},\pm}^{-1} P_{\Robin}  = 1$ on  $\tSob^{1}_{\pm}(X)$. 
Furthermore, $P_{{\Robin},\pm}^{-1}$ maps continuously 
\[
P_{\Robin,\pm}^{-1}: \tdotSob^{-1,\infty}_{\rm c}(X)\to \tSob^{1,\infty}_{\loc}(X).
\]
\end{corr}

By uniqueness, the formal adjoint of $P_{{\Robin},\pm}^{-1}$ equals $P_{{\Robin},\mp}^{-1}$.

\subsection{Symplectic space of solutions} We continue to focus on Robin (or Neumann) boundary conditions. The difference of the two propagators,
\beq\label{eq:defG}
G_{\Robin} := P_{{\Robin},+}^{-1}-P_{{\Robin},-}^{-1} : \tdotSob^{-1,s+1}_{\rm c}(X)\to \tSob^{1,s}_{\loc}(X)
\eeq
will be called the \emph{causal propagator} (associated to the choice of boundary conditions). A natural space of solutions is given by
\[
\left\{ u \in \tSob^{1,\infty}_{\loc}(X) : P_{\Robin}u =0\right\}.
\]
Exactly as in \cite{wrochna2017holographic} we can show that this space is the range of $G_{\Robin}$ acting on a quotient space which is suitable for field quantization. A similar statement is true for Dirichlet boundary conditions.

\begin{prop}\label{prop:symp} The causal propagator \eqref{eq:defG} induces a bijection
\beq\label{eq:iso1}
[G_{\Robin}]:\frac{ \tdotSob^{-1,\infty}_{\rm c}(X) }{P_{\Robin}  \tSob^{1,\infty}_{\rm c}(X)}\longrightarrow \left\{ u \in \tSob^{1,\infty}_{\loc}(X) :  P_{\Robin}u =0\right\}.
\eeq
Moreover, $i (\cdot|  G_{\Robin}\cdot)_{L^2}$ induces a non-degenerate Hermitian form on the quotient space  $\tdotSob^{-1,\infty}_{\rm c}(X) /P_{\Robin}  \tSob^{1,\infty}_{\rm c}(X)$.
\end{prop}

We remark that a similar statement was obtained in \cite[Section 4]{dappiaggi2018algebraic} in the case of the Poincar\'e patch of AdS; cf.~\cite[Proposition 34]{dappiaggi2018fundamental} for static spacetimes with smooth timelike boundary.
	
		\begin{rema} We can use the present framework and Proposition \ref{prop:symp} to extend the recent results in \cite{dybalski2018mechanism} on quantum holography to the case of Neumann and Robin boundary conditions. This can be done by replacing the space $H_{0,\b,{\comp}}^{-1,\infty}(X)$, resp.~$H_{0,\b,{\comp}}^{1,\infty}(X)$ considered therein by $\tdotSob^{-1,\infty}_{\comp}(X)$, resp.~$\tSob^{1,\infty}_{\comp}(X)$, and by replacing the map $\partial_+$ therein by the pair of trace maps $(\gamma_-,\gamma_+)$. In this way one obtains a direct analogue of \cite[Theorem 3.7]{dybalski2018mechanism}, the statement of which is non-trivial if one has the following \emph{unique continuation property} for some open set $O\subset \pX$: for any $u\in\tSob^{1,-\infty}_{\loc}(X)$ solving $P_{\Robin}u=0$, if $\gamma_-u=0$ and $\gamma_+u=0$ then $u=0$ on some non-empty $V(O)\subset X$. Results of this type were obtained by Holzegel and Shao \cite{holzegel2016unique,holzegel2017unique} in the case of high regularity solutions.
\end{rema}

\subsection{Microlocal regularity of traces} We will need more precise mapping statements for $\gamma_\pm$ in the case of high conormal regularity. We show that as in the case of conventional traces, the wavefront set of $\gamma_\pm u$ is controlled in term of the $\b$-wavefront set of $u$; cf.~\cite[Proposition 3.9]{wrochna2017holographic} for a similar result, proved therein only in the Dirichlet case and assuming $Pu=0$. Here we are assuming $\nu \in (0,1)$.

To handle the $\gamma_+$ case, we need the following observation: if $u \in \maxdom^\infty$ and $B \in \Psibeven^0(X)$, then $FBF^{-1}u \in \maxdom^\infty$ as well by Lemma \ref{lem:preservesmaxdomain}. In particular, $\gamma_+ (FBF^{-1}u)$ can be computed by the formula
\[
\gamma_+ (FBF^{-1}u) = x^{1-2\nu}\partial_x(BF^{-1}u)|_{\pa X}.
\]
Now write
\begin{align*}
x^{-2\nu}x\pa_x BF^{-1} &= x^{-2\nu}([B,x\pa_x] F^{-1} + Bx\pa_x F^{-1}) \\
&= x^{2-2\nu}B' F^{-1}u + B'' x^{1-2\nu}\pa_x(F^{-1}u),
\end{align*}
where $B' = x^{-2}[B,x\pa_x] \in \Psib^0(X)$ since $B \in \Psibeven^0(X)$, and $B'' = x^{-2\nu}Bx^{2\nu} \in \Psib^0(X)$. Since $u \in \maxdom^\infty$ and $\nu \in (0,1)$,
\[
x^{2-2	\nu}B'F^{-1}u|_{\pa X} = 0
\]
by \eqref{eq:graphnormexpansion} and \eqref{eq:sobolevembedding}. Furthermore, $B'' x^{1-2\nu}\pa_x(F^{-1}u)|_{\pa X} = \widehat{N}(B)(-2i\nu)(\gamma_+u)$. Since $\maxdom^\infty$ is dense in $\maxdom^k$, we conclude that if $B \in \Psibeven^0(X)$ and $u \in \maxdom^k$, then
\begin{equation} \label{eq:gamma_+commute}
\gamma_+(FBF^{-1}u) = \widehat{N}(B)(-2i\nu)(\gamma_+ u).
\end{equation}
The other observation we need is the following:
\begin{lemm} \label{lem:maxdomWF}
	Let $u \in \maxdom^{-\infty}$, and suppose that $q \notin \WFb^{1,\infty}(u) \cup \WFb^{0,\infty}(x^{-2}Pu)$. If $B \in \Psibeven^0(X)$ and $\opWFb(B)$ is a sufficiently small conic neighborhood of $q$, then 
	\[
	x^{-2}PAu \in \tSob^{0,\infty}_\loc(X)
	\]
	 if we set $A = FBF^{-1}$.
\end{lemm}
\begin{proof}
By taking $\opWFb(B)$ to be a sufficiently small conic neighborhood of $q$, we may assume that $Au \in \tSob^{1,\infty}_\loc(X)$. Tracing through the proof of Lemma \ref{lem:preservesmaxdomain} and further shrinking the microsupport of $B$ as necessary, we see that $x^{-2}PAu \in \tSob^{0,\infty}_\loc(X)$ as well; this is a consequence of the microlocality the operations involved in Lemma \ref{lem:preservesmaxdomain}.
\end{proof}

\begin{lemm}\label{prop:wf} If $u \in\mathcal{H}^{1}_{\loc}(X)$, then
	\beq\label{eq:wfgamm}
	\wf( \gamma_- u )\subset \WFb^{1,\infty}(u) \cap T^*\pa X.
	\eeq
	Moreover, if $Pu\in x^2 \mathcal{L}^2_{\loc}(X)$, then
	\[
	\wf(\gamma_+u) \subset (\WFb^{1,\infty}(u) \cup \WFb^{0,\infty}(x^{-2}Pu)) \cap T^*\pa X.
	\] 
\end{lemm}
\begin{proof}
	The proof for $\gamma_-$ is standard: let $q \in T^*\pa X$, and suppose that $q \notin \WFb^{1,\infty}(u)$, meaning that there exists $B\in\Psi^0_\b(X)$ elliptic at $q$ such that $Bu\in \mathcal{H}^{1,\infty}_\loc(X)$. Then 
	\[
	\widehat{N}(B)(-i\nu_-) \gamma_- u  = \gamma_- Bu \in \CI(X),
	\]
	which finishes the argument since $\widehat{N}(B)(-i\nu_-)$ is elliptic at $q$.
	
	Now consider the $\gamma_+$ case. Let $q \in T^* \pa X$, and suppose that $q \notin \WFb^{1,\infty}(u) \cup \WFb^{0,\infty}(x^{-2}Pu)$. By Lemma \ref{lem:maxdomWF} we may find $B \in \Psibeven^0(X)$ elliptic at $q$ such that
	\[
	FBF^{-1}u \in \maxdom^\infty.
	\]
	In particular, $\gamma_+(FBF^{-1}u) \in \CI(\pa X)$ by Lemma \ref{lem:gammaplus}. Finally, according to \eqref{eq:gamma_+commute},
	\[
	\widehat{N}(B)(-2i\nu)(\gamma_+ u) = \gamma_+(FBF^{-1}u) \in \CI(\pa X),
	\]
	which shows that $q \notin \wf(\gamma_+u)$.
\end{proof}

\subsection{Holographic Hadamard condition}

%In what follows we focus on Neumann and Robin boundary conditions, and drop the `$\Robin$' subscripts subsequently.

Let $\nu \in (0,1)$. We denote by $\cWb$ the set of bounded operators  from  $\tdotSob^{-1,-\infty}_{\rm c}(X)$ to $\mathcal{H}^{1,\infty}_{\rm loc}(X)$. 
%Note that we have for instance $\Psi_\b^{-\infty}(X)\subset\cWb$. 
Following \cite{wrochna2017holographic} (though using different spaces of distributions) we introduce an operatorial $\be$-wave front set which is a subset of $\bee S^*X \times \bee S^*X$. 

\begin{defi} \label{def:wfbp} Suppose $\Lambda:\tdotSob^{-1,-\infty}_{\rm c}(X)\to\mathcal{H}^{1,-\infty}_{\rm loc}(X)$ is continuous. We say that $(q_1,q_2)\in \bee S^* X \times \bee S^* X$ is not in $\WFbop(\Lambda)$ if there exist  $B_i\in \Psi_\b^0(X)$, elliptic at $q_i$ ($i=1,2$), and such that $B_1 \Lambda B_2^*\in \cWb$.
\end{defi}

%As in \cite[Lemma 5.2]{wrochna2017holographic}, for $B \in \Psib^m(X)$ one also has the straightforward relation
%\[
%\WFbop(B)= \{ (q,q) : q\in \wf'_\b(B) \}.
%\]
%According to Definition \ref{def:wfbp}, $\WFbop(\Lambda)$ contains no information about potential singularities located at the zero section of either component. However, using the mapping properties of $\b$-pseudo\-differential operators (and of $\Lambda$) we can show as in \cite{wrochna2017holographic} the following lemma.
%
%\begin{lemm} Suppose $\WFbop(\Lambda)=\emptyset$. Then $\Lambda\in\cWb$.
%\end{lemm}

The notion of \emph{holographic Hadamard two-point functions} introduced in \cite{wrochna2017holographic} has the following straightforward adaptation to the present case.

\begin{defi} We say that two continuous operators $\Lambda_\Robin^\pm : \tdotSob^{-1,-\infty}_{\rm c}(X)\to \mathcal{H}^{1,-\infty}_{\loc}(X)$ are (Robin) \emph{two-point functions} if:
\beq\label{eq:2ptfct}
\bea i) \quad & P_R \Lambda^\pm_\Robin = \Lambda^\pm_\Robin P_R =0,\\
ii) \quad & \Lambda^+_\Robin -\Lambda^-_\Robin = i G_R \,\mbox{ and }\, \Lambda^\pm_\Robin \geq 0 \mbox{ on } \tdotSob^{-1,\infty}_{\rm c}(X).
\eea
\eeq
We say that $\Lambda^\pm_\Robin$ are \emph{holographic Hadamard two-point functions} if in addition they satisfy
\beq\label{eq:holhad}
\WFbop(\Lambda^\pm_\Robin )\subset \dot\cN^\pm\times \dot\cN^\pm.
\eeq
\end{defi}

%Similarly, we can define Dirichlet two-point functions $\Lambda_D^\pm : \tSob_\comp^{1,-\infty}(X) \rightarrow \tdotSob^{1,-\infty}_\loc(X)$ by replacing $P_R$ with the Dirichlet realization $P_D$.

 The first definition is a direct adaptation of the standard definition of two-point functions to the setup provided by Proposition \ref{prop:symp}. We refer to e.g.~\cite[Section 7.1]{gerard2017hadamard} for a brief introduction to two-point functions, field quantization and for remarks on the relation with the more commonly used real formalism. We point out that once some two-point functions $\Lambda^\pm$ are given, the general formalism of quasi-free states and of the GNS representation applies, and as an outcome one obtains \emph{quantum fields} (which are not discussed here in any detail).
	
	The second definition provides a replacement for the celebrated \emph{Hadamard condition}, which is widely used on globally hyperbolic spacetimes, and which is formulated in terms of the (usual, smooth) wave front set since the work of Radzikowski \cite{radzikowski1996micro}. The main interest for Hadamard two-point functions comes from Radzikowski's theorem, which asserts their uniqueness modulo smoothing operators. Thus, their singularities have a universal form that comes from the local geometry and which can be subtracted by a renormalization procedure. In our setup, Radzikowski's theorem is replaced by the following result.

\begin{theo}\label{thm:holo1} Holographic Hadamard two-point functions exist and are unique modulo $\mathcal{W}_\b^{-\infty}(X)$.
\end{theo}
\begin{proof} The proof is largely analogous to that of \cite[Theorem 5.11]{wrochna2017holographic} and \cite[Proposition 5.13]{wrochna2017holographic}, so we only sketch it. 

As in \cite[Theorem 5.9]{wrochna2017holographic}, one can reduce the existence problem to a neighborhood of a time slice. By a spacetime deformation argument, this allows one to reduce the problem further to the case of a \emph{standard static} spacetime. By the same argument one can assume without loss of generality that $S_F\geq \lambda x^2$ for some $\lambda>0$ and that $\beta=0$.

Then $P$ equals $f_1 (\partial_t^2 +L) f_2$, for some smooth multiplication operators $f_1,f_2\in\CI(X)$, where the spatial part $L$ is a differential operator on $S$ associated with a quadratic form like $\mathcal{E}_0$, except that the signature is Riemannian. We can associate to $L$ a positive self-adjoint operator consistent with Neumann boundary conditions, as discussed in Appendix \ref{app:elliptic}.
%The mapping properties of $\one_t\otimes L^{-1/2}$ can be deduced from Lemma \label{lem:app}. 
From this point on, the construction of $\Lambda^\pm$ can be done exactly as in \cite[Lemma 4.5]{wrochna2017holographic} and the proof of the holographic condition from \cite[Theorem 5.9]{wrochna2017holographic} can be repeated verbatim. The only subtle point is the mapping properties of $L^{-1/2}$ on $\tSob^{1,\infty}(S)$, which is discussed in Lemma \ref{lemm:squareroot}
\end{proof}

%The same result also holds for the Dirichlet realization $P_D$ provided we replace $\mathcal{W}_\be^{-\infty}(X)$ with the space of continuous operators $\tSob^{-1,-\infty}_\comp(X) \rightarrow \tdotSob^{1,\infty}_\loc(X)$.

We now turn our attention to singularities of parametrices for $P$. If $q_1,q_2\in \bee S^*X$, then  we write $q_1\dot\sim q_2$ if $q_1,q_2\in \dot \cN$ and $q_1,q_2$ can be connected by a $\GBB$. We can also define the backward flow-out of a point $q \in \cchare $ to be 
\[
\mathcal{F}_{q} = \{ q' \in \cchare: \, q' \dot\sim q, \, t(q') \leq t(q)\}.
\]
Since $\GBB$s exhibit possible branching behavior at glancing points, it is not completely trivial that $\mathcal{F}_q$ is closed; this instead follows from the compactness of the set of $\GBB$s with values in a fixed compact subset of $\bS^*X$ (as discussed in \cite[Proposition 5.5]{vasy2008propagation}).

With our new propagation of singularities result at hand, it is straightforward to repeat the proof of \cite[Theorem 5.12]{wrochna2017holographic} to obtain the following result. In this setting, by parametrices we mean bounded operators from $\tdotSob^{-1,-\infty}_{\rm c}(X)$ to $\mathcal{H}^{1,-\infty}_{\loc}(X)$ that are inverses of $P$ modulo errors in $\cWb$.

\begin{theo}\label{prop:wfs} If $\nu \in (0,1)$, then:
\beq\label{eq:wfPpm}
\WFbop(P_{R,\pm}^{-1})\setminus\diag^*\subset \{ (q_1,q_2) : q_1 \dot\sim q_2, \ \pm t(q_1)>\pm t(q_2) \},
\eeq
where $\diag^* = \{ (q_1,q_1)\in \bee S^* X\times \bee S^*X\}$. Furthermore, suppose that $\Lambda^\pm_\Robin$ are holographic Hadamard two-point functions. Then 
\beq\label{eq:mlkfn}
\WFbop(\Lambda^\pm_\Robin)\subset \{ (q_1,q_2)\in\dot\cN^\pm\times\dot\cN^\pm : q_1 \dot\sim q_2\}. 
\eeq
Moreover, setting $P^{-1}_{R, \rm F}:= i^{-1} \Lambda^+_\Robin + P_{R,-}^{-1}$ and $P^{-1}_{R,\rm \overline{F}}:= -i^{-1} \Lambda^-_\Robin + P_{R,-}^{-1}$, we have
\beq \label{eq:wfF} 
\bea
\WFbop(P_{\Robin,\rm F}^{-1})\setminus \diag^* &\subset \{ (q_1,q_2):  q_1 \dot\sim q_2, \mbox{ and } \pm t( q_1) < \pm t( q_2) \mbox{ if } q_1\in\dot\cN^\pm \},\\
\WFbop(P_{\Robin, \rm \overline{F}}^{-1})\setminus\diag^*&\subset \{ (q_1,q_2) : q_1 \dot\sim q_2, \mbox{ and } \mp t(q_1) < \mp t( q_2) \mbox{ if } q_1\in\dot\cN^\pm \}.
\eea
\eeq
Furthermore, the respective condition in \eqref{eq:wfPpm} or \eqref{eq:wfF} characterizes $P_{R,+}^{-1}$, $P_{R,-}^{-1}$, $P_{R,\rm F}^{-1}$ and  $P_{R,\rm \overline{F}}^{-1}$ uniquely modulo terms in $\cWb$ among parametrices of $P_R$.
\end{theo}

The analogue of this theorem also holds for the Dirichlet realization $P_D$ (and any $\nu > 0$).
%\begin{proof}
%	\todo{I am sketching my understanding of the proof of the first part about the wavefront set of $P_{R,+}^{-1}$.}
%	Let $(q_1,q_2) \in \bS^*X \times \bS^*X$ be such that $q_1 \neq q_2$. Ellipticity implies that if $q_1 \notin \cchare$, then $(q_1,q_2) \notin \WFbop(P_{R,+}^{-1})$, so we may assume $q_1 \in \cchare$. The other obvious part is that if $t(q_1) < t(q_2)$, then $(q_1, q_2) \notin \WFbop(P_{R,+}^{-1})$. 
%	
%	So now suppose that $t(q_1) \geq t(q_2)$ and that the backward flow-out $\mathcal{F}_{q_1}$ is disjoint from $q_2$. Since the flow-out is closed, we can find a compactly supported $B_2 \in \Psib^0(X)$ elliptic at $q_2$ such that $\opWFb(B_2)$ is also disjoint from $\mathcal{F}_{q_1}$. Now let $f \in \tdotSob^{-1,k}_\comp(X)$, and set
%	\[
%	u = P_{R,+}^{-1}B_2^*f.
%	\]
%	Then $\WFb^{-1,\infty}(P_Ru)$ is disjoint from the backwards flow-out by construction. It then suffices to show that $q_1 \notin \WFb^{1,\infty}(u)$ (actually we need that a $u$-independent neighborhood of $q_1$ is disjoint from $\WFb^{1,\infty}(u)$, but the idea is the same). Indeed, if $q_1 \in \WFb^{1,\infty}(u)$, then by propagation of singularities and our observation about $\WFb^{-1,\infty}(P_Ru)$, there must be a $\GBB$ $\gamma$ carrying wavefront set from $q_1$ in the backward flow-out. However, if we flow backwards for long enough, then $\gamma$ enters the region where $u=0$, since $u$ has support in $t\geq t_0$ for some $t_0$. This is a contradiction, so $q_1 \notin \WFb^{1,\infty}(u)$.
%\end{proof}

\subsection{Induced two-point functions at the boundary} We continue to assume $\nu \in (0,1)$. Any continuous operator  $\Lambda: \tdotSob^{-1,-\infty}_\comp(X)\to \mathcal{H}_\loc^{1,-\infty}(X)$ induces an operator on the boundary:
\[
\gamma_- \Lambda \gamma_-^* : \mathcal{E}'(X) \rightarrow \mathcal{D}'(\pa X).
\]
Defining $\gamma_+ \Lambda \gamma_+^*$ is more delicate and requires additional hypotheses; sufficient conditions are given in the next lemma. Observe that $\maxdom^\infty$ is dense in $\tSob^{1,\infty}_\loc(X)$, so we can view $\tdotSob_\comp^{-1,-\infty}(X)$ as a dense subspace of the dual $(\maxdom^\infty)'$.

\begin{lemm} \label{lem:Lambdaextension}
Let $\nu \in (0,1)$, and suppose that
\begin{gather*}
P\Lambda : \tdotSob_\comp^{-1,-\infty}(X) \rightarrow x^2 \tSob^{0,-\infty}_\loc(X), \quad P\Lambda^* :  \tdotSob_\comp^{-1,\infty}(X) \rightarrow x^2 \tSob^{0,\infty}_\loc(X), 
\\ 
P(P\Lambda)^* : x^{-2}\tSob^{0,\infty}_\comp(X) \rightarrow x^2 \tSob^{0,\infty}_\loc(X).
\end{gather*}
Then $\Lambda$ extends to a continuous map $(\maxdom^{\infty})' \rightarrow \maxdom^{-\infty}$.
\end{lemm}
\begin{proof}
	First observe that $\Lambda^* : \tdotSob^{-1,\infty}_\comp(X) \rightarrow \tSob^{1,\infty}_\loc(X)$, so by the mapping properties of $P\Lambda^*$ we conclude that
	\[
	\Lambda^* : \tdotSob_\comp^{-1,\infty}(X) \rightarrow \maxdom^{\infty}.
	\]
This shows that $\Lambda$ admits an extension
	\[
	\Lambda : (\maxdom^{\infty})' \rightarrow \tSob^{1,-\infty}_\loc(X).
	\]
	Similarly, the mapping properties of $P\Lambda$ and $P(P\Lambda)^*$ imply that in fact $\Lambda : (\maxdom^\infty)' \rightarrow \maxdom^{-\infty}$.
\end{proof}

Since $\gamma_+ : \maxdom^{\infty} \rightarrow \CI(\pa X)$, we can view $\gamma_+^* : \mathcal{E}'(\pa X) \rightarrow (\maxdom^{\infty})'$, which allows us to define 
\[
\gamma_+ \Lambda \gamma_+^* : \mathcal E'(\pa X) \rightarrow \mathcal{D}'(\pa X).
\]
We show that the wave front sets of $\gamma_\pm \Lambda \gamma_\pm^*$ can be estimated in terms of $\WFbop(\Lambda)$.

% If $\Gamma\subset \bee T^* X \times \bee T^* X$, we denote by $\Gamma\traa{\pX\times\pX}$ the intersection $\Gamma\cap (T^* {\pX}\times T^*{\pX})$ defined using the embedding of  $T^*\pX$ in $\bee T^*_{\pX} X$. 

\begin{lemm} Let $\nu \in (0,1)$ and $\Lambda: \tdotSob_\comp^{-1,-\infty}(X)\to \tSob_\loc^{1,-\infty}(X)$ be a continuous operator. Then
\beq\label{eq:wfwwe}
\mathrm{WF}'(\gamma_- \Lambda \gamma_-^*)\cap (S^*\pX \times S^*\pX) \subset \WFbop(\Lambda) \cap (S^*\pX \times S^*\pX) .
\eeq
Furthermore, if $\Lambda$ restricts to a continuous map $ \tdotSob_\comp^{-1,\infty}(X)\to \tSob_\loc^{1,\infty}(X)$ and satisfies 
\begin{gather*}
P\Lambda : \tdotSob_\comp^{-1,-\infty}(X) \rightarrow x^2 \tSob^{0,\infty}_\loc(X), \quad P\Lambda^* :  \tdotSob_\comp^{-1,-\infty}(X) \rightarrow x^2 \tSob^{0,\infty}_\loc(X),
\\ P(P\Lambda)^*: x^{-2}\tSob^{0,-\infty}_\comp(X) \rightarrow x^2 \tSob^{0,\infty}_\loc(X),
\end{gather*}
then the $\gamma_+$ analogue of \eqref{eq:wfwwe} is true.
\end{lemm}
\begin{proof} We focus on the more delicate $\gamma_+$ case, which is essentially the microlocalization of the proof of Lemma \ref{lem:Lambdaextension}. Notice that we are assuming stronger mapping properties as compared to Lemma \ref{lem:Lambdaextension}. Now suppose $(q_1,q_2)\notin \WFbop(\Lambda) \cap (S^*\pa X \times S^*\pa X)$, so that there exists $B_i \in \Psibeven^0(X)$ elliptic at $q_i$ such that 
	\[
	A_1 \Lambda A_2^*\in\mathcal{W}_\b^{-\infty}(X),
	\]
	 where we have set $A_i = FB_i F^{-1}$. Note that $\Lambda : (\maxdom^\infty)' \rightarrow \maxdom^{-\infty}$, hence the same is true of $A_1 \Lambda A_2^*$. The claim is that $A_1 \Lambda A_2^*$ extends to a map $(\maxdom^\infty)' \rightarrow \maxdom^\infty$. First, we show that the range of $A_1 \Lambda A_2^*$ in this extended sense is contained in $\tSob^{1,\infty}_\loc(X)$. To see this, note that 
	\[
	(A_1 \Lambda A_2^*)^* = A_2\Lambda^* A_1^* : \tdotSob^{-1,-\infty}_\comp(X)  \rightarrow\tSob^{1,\infty}_\loc(X)
	\]
since $A_1 \Lambda A_2^* \in \mathcal{W}_\be^{-\infty}$, and then Lemma \ref{lem:maxdomWF} shows that $(A_1 \Lambda A_2^*)$ has its range contained in $\maxdom^\infty$. Thus $A_1\Lambda A_2^*$ maps $(\maxdom^\infty)' \rightarrow \tSob^{1,\infty}_\loc(X)$. It then suffices to show that 
\[
PA_1\Lambda A_2^*: (\maxdom^{\infty})' \rightarrow  x^2\tSob^{0,\infty}_\loc(X).
\]
Again applying Lemma \ref{lem:maxdomWF}, we see that $PA_1 \Lambda A_2^*$ maps $\tdotSob^{-1,-\infty}_\comp(X) \rightarrow x^2\tSob^{0,\infty}_\loc(X)$. It then suffices to show that 
\begin{equation} \label{eq:PPA1}
P(PA_1\Lambda A_2^*)^*: x^{-2}\tSob^{0,-\infty}_\comp(X) \rightarrow x^2\tSob^{0,\infty}_\loc(X).
\end{equation}
Using our previously established mapping properties, we can write this composition as
\begin{align*}
PA_2 (P A_1 \Lambda )^* &= PA_2 (P\Lambda )^*A_1^* + PA_2([P,A_1]\Lambda)^*.
\end{align*}
According to Lemma \ref{lem:maxdomWF}, we see that $PA_2 (P\Lambda)^*A_1^*$ has the requisite mapping properties. Now following the proof of Lemma \ref{lem:preservesmaxdomain}, we can write $[P,A_1] = B' P + Q B'' +  B'''$, where 
\[
B' \in \Psib^{-1}(X), \quad B'' \in x^2 \Psib^0(X), \quad B''' \in x^2\Psib^{1}(X), \quad Q \in \Diff_\nu^1(X).
\]
We then write
\begin{equation} \label{eq:Pcommutatordecomp}
([P,A_1]\Lambda)^* = ((B'P + QB'' + B''' )\Lambda)^* = (P\Lambda)^*B' + \Lambda^* (B'' Q)^* + \Lambda^* (B''')^*.
\end{equation}
Write $B''' = B'''C + x^2 R$, where $C \in \Psib^0(X)$ satisfies $\opWFb(1-C) \cap \opWFb(B''') = \emptyset$ and $R \in \Psib^{-\infty}(X)$. Consider the term 
\[
\Lambda^* (B''')^* = \Lambda^* C^* (B''')^* + \Lambda^* R^* x^2.
\] Note that $(B''')^*$ maps $x^{-2}\tSob^{0,-\infty}_\comp(X) \rightarrow \tdotSob^{-1,-\infty}_\comp(X)$, and that
\[
A_2\Lambda^* C^* = (C\Lambda A_2)^* : \tdotSob^{-1,-\infty}_\comp(X) \rightarrow \tSob^{1,\infty}_\loc(X)
\]
if $C$ has sufficiently small wavefront set near $q_1$, which can be arranged by choosing $A_1$ appropriately. Similarly,
\[
A_2 \Lambda^* R^* = (R\Lambda A_2)^*: \tdotSob^{-1,-\infty}_\comp(X) \rightarrow \tSob^{1,\infty}_\loc(X).
\]
This shows that $A_2\Lambda^* (B''')^*$ maps $x^{-2}\tSob^{0,-\infty}_\comp(X) \rightarrow \tSob^{1,\infty}_\loc(X)$, and since $P\Lambda^*(B''')^*$ maps $x^{-2}\tSob^{0,-\infty}_\comp(X) \rightarrow x^2\tSob^{0,\infty}_\loc(X)$, a final application of Lemma \ref{lem:maxdomWF} shows that 
\[
PA_2\Lambda^*(B''')^* :  x^{-2}\tSob^{0,-\infty}_\comp(X) \rightarrow x^2\tSob^{0,\infty}_\loc(X).
\]
The desired mapping properties of the other terms arising from \eqref{eq:Pcommutatordecomp} are obtained similarly, which establishes \eqref{eq:PPA1}.

We have shown that $A_1 \Lambda A_2^*$ extends to a map $(\maxdom^\infty)' \rightarrow \maxdom^{\infty}$. In particular,
\[
\gamma_+ (A_1 \Lambda A_2^*) \gamma_+^* : \mathcal{E}'(\pa X) \rightarrow \CI( \pa X).
\]
Finally, notice that we can use \eqref{eq:gamma_+commute} to write $\gamma_+ A_i = \widehat{N}(B_i)(-2i\nu)\gamma_+$ on $\maxdom^\infty$, where $\tilde B_i = \widehat{N}(B_i)(-2i\nu) \in \Psib^0(\pa X)$ is elliptic at $q_i$. Thus
\[
\tilde B_1 \gamma_+ \Lambda \gamma_+^* \tilde B_2^* \in \Psi^{-\infty}(\pa X),
\]
which finishes the proof.
\end{proof}

We can now conclude as in \cite[Theorem 5.16]{wrochna2017holographic} the following result, which applies to holographic Hadamard two-point functions of the form $\Lambda^\pm_R$.

\begin{theo}\label{thm:holo2} Suppose $(X,g)$ is an asymptotically AdS spacetime and $\nu \in (0,1)$. If $\Lambda^\pm_R$ is a pair of holographic Hadamard two-point functions then
\begin{align*}
%\beq\label{betterwf}
\wf'(\gamma_+ \Lambda_R^\pm \gamma_+^*) \cap (S^*\pX \times S^*\pX) &\subset\WFbop(\Lambda_R^\pm) \cap (S^*\pX \times S^*\pX) \\ &\subset (\dot\cN^\pm\times\dot\cN^\pm) \cap (S^*\pX \times S^*\pX),
%\eeq
\end{align*}
and the same is true for $\gamma_- \Lambda_R^\pm \gamma_-^*$. Furthermore, if $\tilde\Lambda_R^\pm$ is another pair of holographic Hadamard two-point functions then $\gamma_+(\tilde\Lambda_R^\pm-\Lambda_R^\pm)\gamma_+^*$ and $\gamma_-(\tilde\Lambda_R^\pm-\Lambda_R^\pm)\gamma_-^*$ have smooth Schwartz kernel. 
\end{theo}

The operators $\gamma_- \Lambda_R^\pm \gamma_-^*$ and $\gamma_+ \Lambda^\pm_R \gamma_+^*$ are interpreted as two-point functions of an induced theory at the boundary, in the formalism of \emph{generalized free fields}, see \cite{sanders2010equivalence} (the word ``generalized'' refers to the fact that they are not solutions of a natural differential equation). Theorem \ref{thm:holo2} asserts that these two-point functions satisfy a generalized version of the Hadamard condition which is nevertheless sufficient in applications (see also \cite{sanders2010equivalence}).

%Since $\wf'_\b(\Lambda)$ is invariant under the componentwise, fiberwise $\rr_+$-action of dilations, we may replace each copy of $\be T^* X\setminus\!\zero$ by the quotient
%\[
%\be S^*X\defeq (\be T^*X\setminus\zero) / \rr_+
%\] 
%by the fiberwise $\rr_+$-action of dilations. We will often do so without stating it explicitly; this is especially useful when discussing neighborhoods. \medskip

%\begin{lemma}\label{lem:altdefwfb} For any $q_1,q_2\in \be S^* X$, $(q_1,q_2)\notin\wf'_\b(\Lambda)$ if and only if there exist neighbourhoods $\Gamma_i$ of $q_i$ such that for all $B_i\in\Psi_\b^0(X)$ elliptic at $q_i$ satisfying $\wf^\Psi_\b(B_i)\subset \Gamma_i$, $i=1,2$, $B_1 \Lambda B_2^*\in\cWb$.
%\end{lemma}

%\begin{lemma}\label{wfs} Let $\Lambda,\tilde\Lambda:\Hc^{k_2,-\infty}(X)\to\Hl^{k_1,-\infty}(X)$, then
%\[
%\wf_\b'(\Lambda+\tilde\Lambda)\subset \wf_\b'(\Lambda)\cup \wf_\b'(\tilde\Lambda).
%\]
%\end{lemma}
%\proof If $(q_1,q_2)\notin \wf_\b'(\Lambda)$ and $(q_1,q_2)\notin \wf_\b'(\tilde\Lambda)$ then by Lemma \ref{lem:altdefwfb} we can choose $B_1,B_2$ elliptic at resp. $q_1,q_2$ such that both $B_1 \Lambda B_2^*$ and $B_1 \tilde\Lambda B_2^*$ belong to $\cWb$. Hence  $B_1 (\Lambda+\tilde\Lambda) B_2^*$ belongs to $\cWb$ and thus $(q_1,q_2)\notin \wf_\b'(\Lambda+\tilde\Lambda)$.\qed

\appendix

\section{The elliptic setting} \label{app:elliptic}

\subsection{Mapping properties} In this appendix we consider the analogue of the operator $P_{\Robin}$ in the Euclidean signature. The manifold with boundary will be denoted by $S$ and will be assumed compact. With $x \in \CI(S)$ a boundary defining function, fix $F \in x^{\nu_-}\CI(S)$. Given a smooth Riemannian metric $h$ on $S$, we then consider the quadratic form 
\[
\mathcal{L}(u,v) = \int H(d_F u, d_F v) + \alpha u \bar v \, x^2 dh,
\]
where $H$ is the sesquilinear pairing on one-forms induced by the metric, and $\alpha >0$ is a constant chosen to make $\mathcal{L}$ coercive on $\tSob^1(S)$. We can associate to $\mathcal{L}$ an operator 
\[
L : \tSob^1(S) \rightarrow \tdotSob^{-1}(S),
\] 
which extends to a positive self-adjoint operator $L : D(L)  \rightarrow \Ltwo(S)$ with form domain $D(L^{1/2}) = \tSob^1(S)$. The key mapping property we need is that $L^{-1/2}$ preserves conormality:
\begin{lemm} \label{lemm:squareroot}
The operator $L^{-1/2}$ maps $\tSob^{1,\infty}(S) \rightarrow \tSob^{1,\infty}(S)$ continuously.
\end{lemm} 
\begin{proof}
	First, note that for $u \in \tSob^{1,\infty}(S)$ we have a trivial estimate 
	\[
	\| u \|_{\tSob^{1}(S)} \leq \| (L + \mu) u \|_{\tdotSob^{-1}(S)}
	\]
	for $\mu > 0$.
Arguing as in Lemma \ref{lem:dirichletform2}, induction then shows that for each integer $s \geq 0$,
\begin{equation} \label{eq:resolventbounduniform}
\| u \|_{\tSob^{1,s}(S)} \leq \| (L + \mu) u \|_{\tdotSob^{-1,s}(S)}.
\end{equation}
We also have the estimate 
\begin{equation} \label{eq:resolventbound}
\| (L + \mu)^{-1} \|_{\tSob^1(S) \rightarrow \tSob^1(S)} \leq C\mu^{-1},
\end{equation}
since $L^{\pm 1/2}$ commutes with the resolvent and the estimate is clearly true for $\Ltwo(S)$ replacing $\tSob^1(S)$. The claim is that for each $s\geq 0$ an integer we also have 
\begin{equation} \label{eq:resolventboundconormal}
\| (L + \mu)^{-1} \|_{\tSob^{1,s}(S) \rightarrow \tSob^{1,s}(S)} \leq C\mu^{-1},
\end{equation}
for some $C >0$ depending on $s$. Notice that the analogue of Theorem \ref{theo:ellipticregion} (in Riemannian signature) applied to $L$ shows that $L+\mu$ is invertible $\tdotSob^{-1,s}(S) \rightarrow \tSob^{1,s}(S)$ for each $s \in \RR$, so \eqref{eq:resolventboundconormal} is well-defined.  Let $A_s \in \Psib^s(S)$ be elliptic and $A_{-s} \in \Psib^{-s}(X)$ be a parametrix, so $A_s A_{-s} + R = 1$ and $A_{-s}A_s + R' = 1$ for some $R,R' \in \Psib^{-\infty}(S)$. Given $u \in \tSob^{1,\infty}(S)$,
\[
\| (L+\mu)^{-1} u \|_{\tSob^{1,s}(S)} \leq C\| A_s (L+\mu)^{-1} u\|_{\tSob^{1}(S)} + C\| (L+\mu)^{-1} u \|_{\tSob^{1}(S)}.
\]
The second term on the right-hand side we bound by $\mu^{-1} \| u\|_{\tSob^{1}(S)} \leq \mu^{-1} \| u \|_{\tSob^{1,s}(S)}$, using \eqref{eq:resolventbound}. Now write $u = (A_{-s}A_s + R')u$. We then bound
\[
\|A_s(L+\mu)^{-1}A_{-s}A_su \|_{\tSob^{1}(S)} \leq \| A_s(L+\mu)^{-1}A_{-s} (L+\mu) \|_{\tSob^{1}(S)\rightarrow \tSob^{1}(S)}\| (L+\mu)^{-1} A_s u \|_{\tSob^{1}}. 
\]
Now write 
\[
A_s(L+\mu)^{-1}A_{-s}(L+\mu) = A_s (L+\mu)^{-1}A_{-s} L + A_sA_{-s} - A_s (L+\mu)^{-1}LA_{-s},
\]
and apply \eqref{eq:resolventbounduniform} to see that this operator mapping $\tSob^{1}(S) \rightarrow \tSob^{1}(S)$ is uniformly bounded in $\mu$. On the other hand, $\| (L+\mu)^{-1} A_s u \|_{\tSob^{1}} \leq C\mu^{-1} \| u \|_{\tSob^{1,s}(S)}$, which shows that
\[
\|A_s(L+\mu)^{-1}A_{-s}A_su \|_{\tSob^{1}(S)} \leq C\mu^{-1}\| u \|_{\tSob^{1,s}(S)}.
\]
The term $\| A_s(L+\mu)R'u\|_{\tSob^{1}(S)}$ is bounded similarly.

Since we can write
\[
L^{-1/2} = \frac{1}{\pi} \int_0^\infty \mu^{-1/2}(L+\mu)^{-1/2} \, d\mu, 
\]
the estimate \eqref{eq:resolventboundconormal} shows that $L^{-1/2}$ indeed maps $\tSob^{1,\infty}(S) \rightarrow \tSob^{1,\infty}(S)$
\end{proof}

\section{Even b-calculus} \label{app:even}

\subsection{Even b-pseudodifferential operators}

Let $(X,g)$ be an aAdS spacetime which is even modulo $\mathcal{O}(x^{2k+1})$. We sketch a construction of the b-pseudodifferential operators of order $m$ on $X$ that are even modulo $\mathcal{O}(x^{2k+3})$, which we denote by $\Psibeven^m(X)$; the construction closely mirrors that of an even subcalculus by Albin in the 0-calculus \cite{albin2007renormalized}. The operators we consider are described in terms of their Schwartz kernels, and for this reason we assume familiarity with the construction of the usual b-calculus in terms of conormal distributions on the b-stretched product \cite[Chapter 4, 5]{melrose1993atiyah}.

\begin{defi}
We say that $f \in \CI(X)$ is even modulo $\mathcal{O}(x^{2k+3})$, written $f \in \CIeven(X)$, if in any special coordinate system $(x,y)$ the Taylor of expansion of $f$ at $\pa X$ contains only even terms modulo $\mathcal{O}(x^{2k+3})$
\end{defi} 

 This space is well defined (i.e., independent of the choice of special coordinates) in view of \eqref{eq:changeofspecialcoords}. Similarly, we can define the space $\CIodd(X)$ of odd functions modulo $\mathcal{O}(x^{2k+4})$.

Next we consider b-pseudodifferential operators. We write $X^2_\be$ for the b-stretched product, which recall is obtained by blowing up the corner $\partial X \times \partial X$ in $X\times X$. For simplicity we will neglect various b-half-density factors. Let $(x,y)$ and $(x',y')$ be special local coordinates  on $X$, such that $x = x'$. Thus $(x,y,x',y')$ are valid local coordinates on $X^2$ near $\pa X \times \pa X$. Near the front face in $X^2_\be$, local coordinates are given by
\[
\tau = \frac{x-x'}{x+x'}, \quad r= x+x', \quad y, \quad y',
\] 
where $r$ is a bdf for the front face. 

\begin{defi}
We say that a smooth function $f \in \CI(X^2_\be)$ is even modulo $\mathcal{O}(r^{2k+3})$, written $f\in \CIeven(X^2_\be)$, if the Taylor expansion of $f$ at the front face $\{r=0\}$ in coordinates $(\tau,r,y,y')$ contains only even terms modulo $\mathcal{O}(r^{2k+3})$.
\end{defi} 
 In other words, $f \in \CI(X^2_\be)$ is even modulo $\mathcal{O}(r^{2k+3})$ if we can write
\[
f(\tau,r,y,y') = f_0(\tau,r^2,y,y') + r^{2k+3}f'(\tau,r,y,y')
\]
for smooth functions $f_0, f'$. Again by \eqref{eq:changeofspecialcoords}, this definition is independent of the choice of special coordinates in either factor. We can now define the space of even  b-pseudodifferential operators, recalling that that elements of  $\Psib^{m}(X)$ have Schwartz kernels on $X^2_\be$ that are conormal to the lifted diagonal. In local coordinates $(\tau,r,y,y')$, we can view these as distributions conormal to $\{\tau = 0, \, y=y'\}$ with smooth parametric dependence on $r$.

\begin{defi}
	If $A \in \Psib^{m}(X)$, then we say that $A \in \Psibeven^{-\infty}(X)$ if its Schwartz kernel $\mathcal{K}_A$ in local coordinates $(\tau,r,y,y')$ as above has a Taylor expansion at the front face $\{r=0\}$ containing only even terms modulo $\mathcal{O}(r^{2k+3})$.
\end{defi}

The Schwartz kernel $\mathcal{K}_A$ also has the usual infinite order of vanishing at the side faces. Roughly speaking, this definition means that the total symbol of $\mathcal{K}_A$ in local coordinates $(r,y,\sigma,\eta)$ is even in $r$ modulo $\mathcal{O}(r^{2k+3})$.

Next, we discuss composition of even b-pseudodifferential operators. First, we note that 
\[
\CIeven(X^2_\be) \cdot \CIeven(X^2_\be) \subset \CIeven(X^2_\be),
\]
and that the lifts of $\CIeven(X)$ functions to $X^2_\be$ from either the left or right factors land in $\CIeven(X^2_\be)$. Using partitions of unity $\chi_i \in \CIeven(X)$, one can reduce to composition of even operators with localized Schwartz kernels, just as considered in \cite[Section 5.9]{melrose1993atiyah}. It is then straightforward to see that even operators are closed under composition:

\begin{prop}
$A \in \Psibeven^m(X)$ and $B \in \Psibeven^{m'}(X)$, then $AB \in \Psibeven^{m+m'}(X)$.
\end{prop}

On $\RR^n_+$ consider the quantization procedure given by
\[
\Opb(a)u(x,y) = \int e^{i ((x/x'-1)\sigma + \langle y-y',\eta \rangle)}\phi(x/x')a(x,y,\sigma,\eta)  \frac{dx'}{x'}dy' d\sigma d\eta,
\]
where $\phi \in \CcI(\RR)$ satisfies $\supp \phi \subset (1/2,2)$ and $\phi(s) = 1$ near $s=1$. If $a \in S^m(\RR^n_+)$ is of the form 
\[
a(x,y,\sigma,\eta) = a_0(x^2,y,\sigma,\eta) + x^{2k+3}a'(x,y,\sigma,\eta),
\]
then $\Opb(a) \in \Psibeven^m(\RR^n_+)$ irrespective of the choice of aAdS metric on $\RR^n_+$ of the form 
\[
\frac{-dx^2 + k^{\alpha\beta}(x,y)dy^{\alpha\beta}}{x^2}.
\] 
By using an even partition of unity to patch these local quantization procedures together, we can quantize \emph{even} symbols on an arbitrary aAdS spacetime $(X,g)$. In particular, we can always construct \emph{elliptic} operators $A \in \Psibeven^m(X)$.

\subsection{The indicial family}

Apart from composition and the existence of elliptic elements, the other key property of even b-pseudodifferential operators is an improved statement about the kernel of the indicial family map. Given $A \in \Psib^m(X)$, recall that $\widehat{N}(A)(s)$ is defined invariantly as follows: first, one restricts the kernel $\mathcal{K}_A$ to the front face in $X^2_\be$. The front face is then identified with the inward pointing spherical normal bundle to $\pa X \times \pa X$, which is identified with $\pa X \times \pa X \times \RR_+$. Finally, the indicial family is the Mellin transform of the resulting function in the $\RR_+$ factor. The vanishing of $\widehat{N}(A)$ is thus equivalent to the statement that $A \in x \Psib^m(X)$.

Now suppose that $(X,g)$ is an aAdS spacetime; in particular, it is trivially even modulo $\mathcal{O}(x)$. The even calculus on $X$ is thus defined modulo cubic terms. In particular, we have the following:
\begin{lemm}
	If
 $A \in \Psibeven^m(X)$ and $\widehat{N}(A) = 0$, then $A \in x^2\Psib^m(X)$.
\end{lemm} 

This implies the following corollary:

\begin{lemm} \label{lem:quadraticvanishing}
	Let $(X,g)$ be an asymptotically AdS spacetime. If $A \in \Psib^m(X)$ has compact support in a coordinate patch with special coordinates $(x,y^1,\ldots,y^{n-1})$, then 
	\[
	[xD_x,A] \in x^2 \Psib^m(X).
	\]
\end{lemm} 
\begin{proof}
	The operator $xD_x$ is always even, which implies that $xD_xA,\, AxD_x$ and hence $[xD_x,A]$ are also even. The proof is finished by observing that $\widehat{N}([xD_x,A]) = 0$.
\end{proof}

	\bibliographystyle{alphanum}
	
	\bibliography{central_bibliography}

\end{document}